\documentclass[12pt]{amsart}

\usepackage{amsfonts, amsmath, amsthm, amssymb, amstext,mathabx,dsfont}
\usepackage{fullpage,tikz,soul,tikz-cd,graphicx,float,caption}
\usepackage{multicol,color,enumerate}
\usepackage[utf8]{inputenc}
\usepackage{wasysym,mathrsfs}
\usetikzlibrary{matrix, arrows, calc}
\usepackage[colorlinks=true,linkcolor=red!70!black,citecolor=green!70!black,urlcolor=magenta!70!black]{hyperref}

\renewcommand{\tilde}{\widetilde}

\let\temp\emptyset
\let\emptyset\varnothing
\let\varnothing\temp 

\theoremstyle{plain}
\newtheorem{theorem}{Theorem}[section]
\newtheorem{proposition}[theorem]{Proposition}
\newtheorem{corollary}[theorem]{Corollary}
\newtheorem{lemma}[theorem]{Lemma}
\newtheorem*{lemma*}{Lemma}
\newtheorem{conjecture}[theorem]{Conjecture}

\newtheorem*{truefact*}{Fact}
\newtheorem{question}[theorem]{Question}
\newtheorem{claim}[theorem]{Claim}

\theoremstyle{definition}
\newtheorem{definition}[theorem]{Definition}

\theoremstyle{remark}
\newtheorem*{remark}{Remark}

\newcommand{\bb}[1]{\expandafter\newcommand\expandafter{\csname #1\endcsname}{{\mathbb {#1}}}} % turns \foo into \mathbb{foo}

\bb C
\bb R
\bb Q
\bb Z
\bb N

\bb F

\newcommand{\thus}{{\Rightarrow}}
\renewcommand{\a}{\alpha}
\renewcommand{\b}{\beta}
\newcommand{\g}{\gamma}
\renewcommand{\d}{\delta}
\newcommand{\ep}{\varepsilon}

\newcommand{\mrm}[1]{\expandafter\newcommand\expandafter{\csname #1\endcsname}{{\mathrm {#1}}}}
\mrm{Hom}
\mrm{Aut}
\mrm{End}
\mrm{Spec}
\mrm{Gal}
\mrm{Alg}
\mrm{Ext}
\mrm{Tor}
\mrm{Ann}

\mrm{PGL}
\mrm{PSL}
\mrm{GL}
\mrm{SL}
\mrm{PG}
\mrm{PSO}
\mrm{PS}
\mrm{PSU}

\newcommand{\ins}{\subseteq}

\mrm{lcm}

\renewcommand{\S}{\mathfrak S}

\renewcommand{\mod}{\bmod}

\title{Odd moments in the distribution of primes}
\author{Vivian Kuperberg}
\address{Department of Mathematics, Stanford University, Stanford, CA, USA}
\email{viviank@stanford.edu}

\thanks{The author is supported by NSF GRFP grant DGE-1656518, and would like to thank Kannan Soundararajan for many helpful comments and discussions, as well as R\'egis de la Bret\`eche, Alexandra Florea, Andrew Granville, Zeev Rudnick, Yuval Wigderson, and the anonymous referee for helpful feedback.}

\begin{document}

\begin{abstract}
Montgomery and Soundararajan showed that the distribution of $\psi(x+H) - \psi(x)$, for $0 \le x \le N$, is approximately normal with mean $ \sim H$ and variance $\sim H \log (N/H)$, when $N^{\delta} \le H \le N^{1-\delta}$. Their work depends on showing that sums $R_k(h)$ of $k$-term singular series are $\mu_k(-h \log h + Ah)^{k/2} + O_k(h^{k/2-1/(7k) + \ep})$, where $A$ is a constant and $\mu_k$ are the Gaussian moment constants. We study lower-order terms in the size of these moments. We conjecture that when $k$ is odd, $R_k(h) \asymp h^{(k-1)/2}(\log h)^{(k+1)/2}$. We prove an upper bound with the correct power of $h$ when $k = 3$, and prove analogous upper bounds in the function field setting when $k =3$ and $k = 5$. We provide further evidence for this conjecture in the form of numerical computations.
\end{abstract}
\maketitle

\section{Introduction}

What is the distribution of primes in short intervals? Cram\'er \cite{cramermodel} modeled the indicator function of the sequence of primes by independent random variables $X_n$, for $n \ge 3$, which are $1$ (``$n$ is prime'') with probability $\frac 1{\log n}$, and $0$ (``$n$ is composite'') with probability $1 - \frac 1{\log n}$. Cram\'er's model predicts that the distribution of $\psi(n+h)-\psi(n)$, a weighted count of the number of primes in an interval of size $h$ starting at $n$, follows a Poisson distribution when $n$ varies in $[1,N]$ and when $h \asymp \log N$. Gallagher \cite{gallaghershortintervals} proved that this follows from a quantitative version of the Hardy-Littlewood prime $k$-tuple conjecture: namely, that if $\mathcal D = \{d_1, d_2, \dots, d_k\}$ is a set of $k$ distinct integers, then 
\[\sum_{n \le N} \prod_{i=1}^k \Lambda(n+d_i) = (\mathfrak S(\mathcal D) + o(1))N,\]
where $\mathfrak S(\mathcal D)$ is the singular series, a constant dependent on $\mathcal D$ given by
\[\mathfrak S(\mathcal D) = \prod_p \left(1-\frac 1p \right)^{-k} \left(1-\frac{\nu_p(\mathcal D)}{p}\right),\]
where $\nu_p(\mathcal D)$ denotes the number of distinct residue classes modulo $p$ among the elements of $\mathcal D$. The singular series is also given by the formula
\begin{equation}\label{eq:sdsums}
\S(\mathcal D) = \sum_{\substack{q_1, \dots, q_k \\ 1 \le q_i < \infty}} \left(\prod_{i=1}^k \frac{\mu(q_i)}{\phi(q_i)}\right) \sum_{\substack{a_1, \dots, a_k \\ 1 \le a_i \le q_i \\ (a_i,q_i) = 1 \\ \sum a_i/q_i \in \Z}} e\left(\sum_{i=1}^k \frac{a_id_i}{q_i}\right).
\end{equation}

The Hardy-Littlewood prime $k$-tuple conjectures give us a better lens through which to understand the distribution of primes: by understanding sums of singular series. For example, Gallagher used the estimate that
\[\sum_{\mathcal D \subset [1,h]} \mathfrak S(\mathcal D) \sim \sum_{\mathcal D \subset [1,h]} 1\]
to prove that the Hardy-Littlewood conjectures imply Poisson behavior in intervals of logarithmic length. Our concern is the distribution of primes in somewhat longer intervals; namely, those of size $H$ where $H=o(N)$ and $H/\log N \to \infty$ as $N \to \infty$. In this setting, the Cram\'er model would predict that the distribution of $\psi(n+H)-\psi(n)$ for $n \le N$ is approximately normal, with mean $\sim H$ and variance $\sim H\log N$. However, the Hardy-Littlewood prime $k$-tuple conjecture gives a different answer in this case. In \cite{MontgomerySoundararajanPrimesIntervals}, Montgomery and Soundararajan provide evidence based on the Hardy-Littlewood prime $k$-tuple conjectures that the distribution ought to be approximately normal with variance $\sim H \log \frac NH$. They consider the $K$th moment $M_K(N;H)$ of the distribution of primes in an interval of size $H$, given by
\[M_K(N;H) = \sum_{n=1}^N(\psi(n+H)-\psi(n)-H)^K.\]
They conjecture that these moments should be given by the Gaussian moments
\[M_K(N;H) = (\mu_K + o(1))N\left(H \log \frac NH \right)^{K/2},\]
where $\mu_K = 1 \cdot 3 \cdots (K-1)$ if $K$ is even and $0$ if $K$ is odd, uniformly for $(\log N)^{1 + \d} \le H \le N^{1-\d}$. Their technique relies on more refined estimates of sums of the singular series constants $\mathfrak S(\mathcal D)$. Instead of the von Mangoldt function $\Lambda(n)$, they consider sums of $\Lambda_0(n) = \Lambda(n) - 1$, where the main term has been subtracted from the beginning. The corresponding form of the Hardy-Littlewood conjecture states that
\[\sum_{n \le N} \prod_{i=1}^k \Lambda_0(n+d_i)  = (\mathfrak S_0(\mathcal D) + o(1))N\]
as $N \to \infty$, where $\mathfrak S_0(\mathcal D)$ is given by
\[\mathfrak S_0(\mathcal D) = \sum_{\mathcal J \ins \mathcal D} (-1)^{|\mathcal D \setminus \mathcal J|} \mathfrak S(\mathcal J),\]
and in turn
\[\mathfrak S(\mathcal D) = \sum_{\mathcal J \ins \mathcal D} \mathfrak S_0(\mathcal J).\]
We can combine this with Equation \ref{eq:sdsums} to see that
\begin{equation}\label{eq:s0d}
\S_0(\mathcal D) = \sum_{\substack{q_1, \dots, q_k \\ 1 < q_i < \infty}} \left(\prod_{i=1}^k \frac{\mu(q_i)}{\phi(q_i)}\right) \sum_{\substack{a_1, \dots, a_k \\ 1 \le a_i \le q_i \\ (a_i,q_i) = 1 \\ \sum a_i/q_i \in \Z}} e\left(\sum_{i=1}^k \frac{a_id_i}{q_i}\right).
\end{equation}
Montgomery and Soundararajan considered the sum 
\begin{equation}\label{eq:rkh}
R_k(h) := \sum_{\substack{d_1, \dots, d_k \\ 1 \le d_i \le h \\ d_i \text{ distinct}}} \mathfrak S_0(\mathcal D),
\end{equation}
showing that for any nonnegative integer $k$, for any $h > 1$, and for any $\ep > 0$,
\begin{equation}\label{eq:msthm2}
R_k(h) = \mu_k(-h \log h + Ah)^{k/2} + O_k(h^{k/2-1/(7k) + \ep}),
\end{equation}
where $A = 2 - \g - \log 2\pi$. Their estimate on $R_k(h)$ implies their bound on the moments. For more on the distribution of primes in short intervals, see for example \cite{MR2217796} and \cite{GranvilleLumley}, as well as \cite{MontgomerySoundararajanPrimesIntervals}.

For all $k$, the optimal error term in \eqref{eq:msthm2} is expected to be smaller. In the case of the variance, this was studied in \cite{MR1954710}. In this paper, we restrict our attention to the cases when $k$ is odd. We conjecture the following, which was mentioned by Lemke Oliver and Soundararajan in \cite{lemkeoliversound}.
\begin{conjecture}\label{conj:rkoddbound}
Let $k \ge 3$ be an odd integer, and let $h > 1$. With $R_k(h)$ defined as above,
\[R_k(h) \asymp h^{(k-1)/2}(\log h)^{(k+1)/2}.\]
\end{conjecture}
The conjectured power of $\log h$ here comes from numerical evidence, which we present in Section \ref{sec:numerics}.
For $k$ odd, we do not know, even heuristically, which terms contribute to the main term in $R_k(h)$; for this reason, we do not know what the constant should be in front of the asymptotic in Conjecture \ref{conj:rkoddbound}. Nevertheless, our goal in this paper is to provide evidence for Conjecture \ref{conj:rkoddbound}. When $k = 3$, we can show an upper bound with the correct power of $h$. 
\begin{theorem}\label{thm:threetermintegersums}
For $h \ge 4$ and $R_3$ defined in \eqref{eq:rkh},
\[R_3(h) \ll h(\log h)^5.\]
\end{theorem}

Another source of evidence for Conjecture \ref{conj:rkoddbound} is the analog of this problem in the function field setting, which is also studied in \cite{MR3158533}. As we discuss in Section \ref{sec:funcfieldmv}, we can consider analogous questions over $\F[T]$ where $\F$ is a finite field, instead of over $\Z$. To state the analog, we first revisit the techniques of Montgomery and Soundararajan in the integer case. Upon expanding Equation \ref{eq:rkh} using Equation \ref{eq:s0d}, we get
\begin{align*}
R_k(h) &= \sum_{\substack{d_1, \dots, d_k \\ 1 \le d_i \le h \\ d_i \text{ distinct}}} \sum_{\substack{q_1, \dots, q_k \\ 1 < q_i < \infty}} \left(\prod_{i=1}^k \frac{\mu(q_i)}{\phi(q_i)} \right) \sum_{\substack{a_1, \dots, a_k \\ 1 \le a_i \le q_i \\ (a_i,q_i) = 1 \\ \sum a_i/q_i \in \Z}} e\left(\sum_{i=1}^k \frac{a_id_i}{q_i}\right) \\
&= \sum_{\substack{q_1, \dots, q_k \\ 1 < q_i < \infty}} \left(\prod_{i=1}^k \frac{\mu(q_i)}{\phi(q_i)} \right) \sum_{\substack{a_1, \dots, a_k \\ 1 \le a_i \le q_i \\ (a_i,q_i) = 1 \\ \sum a_i/q_i \in \Z}} \prod_{i=1}^k E\left(\frac{a_i}{q_i}\right),
\end{align*}
where $E(\a) = \sum_{m=1}^h e(m\a)$. The sums $E(\a)$ approximately detect when $\|\a\| \le \frac 1h$. 

This expression for $R_k(h)$ is closely related to a quantity studied by Montgomery and Vaughan in \cite{MontgomeryVaughanReducedResidues}. They considered the related problem of the $k$th moment of reduced residues modulo a fixed $q$, given by
\[m_k(q;h) = \sum_{n=1}^q \Big(\sum_{\substack{1 \le m \le h \\ (m+n,q) = 1}} 1 - h \frac{\phi(q)}{q}\Big)^k.\]
The moment $m_k$ satisfies $m_k(q;h) = q \left(\frac{\phi(q)}{q}\right)^k V_k(q;h)$, where $V_k(q;h)$ is the ``singular series sum,''
\[V_k(q;h) = \sum_{\substack{d_1, \dots, d_k \\ 1 \le d_i \le h}} \sum_{\substack{q_1, \dots, q_k \\ 1 < q_i|q}} \left(\prod_{i=1}^k \frac{\mu(q_i)}{\phi(q_i)} \right) \sum_{\substack{a_1, \dots, a_k \\ 1 \le a_i \le q_i \\ (a_i,q_i) = 1 \\ \sum a_i/q_i \in \Z}} e\left(\sum_{i=1}^k \frac{a_id_i}{q_i}\right),\]
which differs from $R_k(h)$ only in that the $q_i$ are now constrained to divide a fixed $q$. In this paper as well as in the work of Montgomery and Soundararajan, estimating $V_k(q;h)$ when $q$ is a product of primes $p \le h^{k+1}$ is a key step towards estimating $R_k(h)$. Similarly, understanding $m_k(q;h)$ is closely related to understanding $R_k(h)$. For example, Conjecture \ref{conj:rkoddbound} predicts that $R_k(h) \asymp h^{(k-1)/2}(\log h)^{(k+1)/2}$ when $k$ is odd; this conjecture is closely related to the prediction that when $q$ is a product of primes $p \le h^A$ for a fixed power $A$, and when $k$ is odd, then we should have $m_k(q;h) \asymp q(h/(\log h))^{(k-1)/2}$. In \cite{MontgomeryVaughanReducedResidues}, Montgomery and Vaughan predict that $m_k(q;h) \ll q(h/(\log h))^{(k-1)/2}$ in this setting. In the function field setting, we study an analog of the moments $m_k(q;h)$.

Let $\F_q$ be a finite field with $q$ elements, and let $Q$ be a fixed monic polynomial in $\F_q[t]$. Note that $Q$ in the function field case serves the same role as $q$ in the integer case, since $q$ now represents the size of the field. The moment $m_k(Q;h)$, an analog of the $k$th moment of reduced residues in short intervals which is defined precisely in \eqref{eq:defn-of-mk-in-fcn-field}, is the $k$th moment of the distribution of polynomials that are relatively prime to $Q$ lying in intervals of size $q^h$ in the function field $\mathbb F_q[t]$. In this case an ``interval'' of size $q^h$ centered at a polynomial $G(t)$ consists of all polynomials $F(t)$ such that $F(t) \equiv G(t) \mod t^h$. We can adapt the methods of Montgomery--Vaughan to prove a bound on $m_k(Q;h)$ that has the same shape as the bounds of Montgomery--Vaughan and Montgomery--Soundararajan.  
\begin{theorem}\label{thm:mvfuncfield}
For any fixed $k \ge 3$ and for $Q \in \F_q[t]$ squarefree, for $h \ge 2$,
\[m_k(Q;h) \ll \begin{cases}|Q|(q^h)^{k/2}\left(\frac{\phi(Q)}{|Q|}\right)^{k/2}\left(1 + (q^h)^{-1/(k-2)}\left(\frac{\phi(Q)}{|Q|}\right)^{-2^k+k/2}\right) &\text{ if } k \text{ is even}\\ |Q|((q^h)^{k/2-1/2} + (q^h)^{k/2-1/(k-2)})\left(\frac{\phi(Q)}{|Q|}\right)^{-2^k+k/2} &\text{ if } k \text{ is odd}.\end{cases}\]
\end{theorem}

The function field exponential sums are cleaner than their integer analogs, making this proof more streamlined than the proof of Montgomery--Vaughan. As a result, the bound is tighter; in fact, for $k = 3$, Theorem \ref{thm:mvfuncfield} already yields a bound where the exponent of $q^h$ is $1$. This is of the same shape as Theorem \ref{thm:threetermintegersums}, where the exponent of $h$ is $1$.

Using a more involved argument we can achieve a bound on the fifth moment of reduced residues in short intervals.

\begin{theorem}\label{thm:funcfieldfifth}
Let $h \ge 2$ and let $Q = \prod_{\substack{P \text{irred.} \\ |P| \le q^{6h}}} P$. For all $\ep > 0$,
\[m_5(Q;h) \ll |Q|q^{2h + \ep}.\]
\end{theorem}

As discussed above, Conjecture \ref{conj:rkoddbound} would predict in the integer case that for $k$ odd and $q = \prod_{p \le h^A} p$, we have $m_k(q;h) \asymp q(h/(\log h))^{(k-1)/2}$. In the function field case, we have a polynomial $Q(t)$ in place of the modulus $q$, and an interval of size $q^h$ instead of one of size $h$, so the analog of Conjecture \ref{conj:rkoddbound} would predict that $m_5(Q;h) \asymp |Q| q^{2h}(\log q^{h})^{-2}$. In particular, Theorem \ref{thm:funcfieldfifth} matches the exponent of $q^h$ in this prediction. Our techniques do not quite succeed in proving such a bound for any higher odd moments, as we note in Section \ref{sec:funcfieldfifth}. However, we do get as a corollary the following bound on sums of singular series in function fields. The sum $R_k(q^h)$ of singular series in function fields is defined very analogously to the sum $R_k(h)$ in the integer setting; a precise definition is given in \eqref{eq:fcn-field-def-of-rk-qh}.

\begin{corollary}\label{cor:singseriesff}
Let $h \ge 2$ and let $Q = \prod_{\substack{P \text{ irred.} \\ |P| \le q^{6h}}} P$. Then
\[R_3(q^h) \ll V_3(Q;h) +q^h\left(\frac{|Q|}{\phi(Q)}\right)^2 \ll q^h \left(\frac{|Q|}{\phi(Q)}\right)^{8},\]
and for all $\ep > 0$,
\[R_5(q^h) \ll V_5(Q;h) + \left(\frac{|Q|}{\phi(Q)}\right)^{21/2}q^{2h} \ll q^{(2+\ep)h}.\]
\end{corollary}

This paper is organized as follows. In Section \ref{sec:threeterminteger} we prove Theorem \ref{thm:threetermintegersums}. In Section \ref{sec:funcfieldmv}, we discuss the analogous problem in $\mathbb F_q[T]$, and adapt the framework of Montgomery and Vaughan to the function field setting to prove Theorem \ref{thm:mvfuncfield}. In Section \ref{sec:funcfieldfifth} we prove Theorem \ref{thm:funcfieldfifth}. Finally, in Section \ref{sec:numerics} we provide numerical evidence for Conjecture \ref{conj:rkoddbound}, and in Section \ref{sec:toys} we discuss toy problems, further directions of inquiry, and possible applications of these questions.

\section{Three-term integer sums: Proof of Theorem \ref{thm:threetermintegersums}}
\label{sec:threeterminteger}

Our goal is to bound
\[R_3(h) = \sum_{\substack{d_1, d_2, d_3 \\ 1 \le d_i \le h \\ d_i \text{ distinct}}} \mathfrak S_0(\mathcal D).\]
Expanding $\S_0(\mathcal D)$ as an exponential sum yields
\[R_3(h) = \sum_{\substack{d_1, d_2, d_3 \\ 1 \le d_i \le h \\ d_i \text{ distinct}}} \sum_{\substack{q_1, q_2,q_3 \\ 1 < q_i < \infty}} \left(\prod_{i=1}^3 \frac{\mu(q_i)}{\phi(q_i)}\right) \sum_{\substack{a_1, a_2, a_3 \\ 1 \le a_i \le q_i \\ (a_i, q_i) = 1 \\ \sum a_i/q_i \in \Z}} e\left(\sum_{i=1}^3 \frac{a_id_i}{q_i}\right).\]

Our argument will follow the same thread as that of Montgomery and Soundararajan \cite{MontgomerySoundararajanPrimesIntervals}, which in turn relies on the analysis of Montgomery and Vaughan \cite{MontgomeryVaughanReducedResidues} of the distribution of reduced residues. To that end, we consider $V_3(q;h)$, which is approximately the third centered moment of the number of reduced residues mod $q$ in an interval of length $h$. Precisely, $V_3(q;h)$ is given by
\begin{equation}\label{eq:defn-of-v-3}
V_3(q;h) = \sum_{\substack{d_1, d_2, d_3 \\ 1 \le d_i \le h}} \sum_{\substack{q_1, q_2, q_3 \\ 1 < q_i |q}} \left(\prod_{i=1}^3 \frac{\mu(q_i)}{\phi(q_i)}\right)\sum_{\substack{a_1, a_2, a_3 \\ 1 \le a_i \le q_i \\ (a_i,q_i) = 1 \\ \sum a_i/q_i \in \Z}} e \left( \sum_{i=1}^3 \frac{a_id_i}{q_i} \right). 
\end{equation}
This is very similar to the above expression for $R_3(h)$; the two differences are that the outer sum in $R_3(h)$ is taken over \emph{distinct} $d_i$'s, whereas the outer sum for $V_3(q;h)$ is not, and that the summands $q_i$ range over all integers for $R_3(h)$, but are restricted to factors of $q$ for $V_3(q;h)$. 

\begin{theorem}\label{threetermintegerreducedresidues}
Let $h \ge 4$ and let $q$ be the product of primes $p \le h^4$. Then
\[V_3(q;h) \ll h\left(\log h\right)^5.\]
\end{theorem}

We use Theorem \ref{threetermintegerreducedresidues} to establish Theorem \ref{thm:threetermintegersums}. In order to derive Theorem \ref{thm:threetermintegersums}, it suffices to show that terms arising from transforming $V_3(q;h)$ into $R_3(h)$ do not contribute more than $O(h(\log h)^5)$; in fact they contribute on the order of $h(\log h)^2$, which is the conjectured asymptotic size of $R_3(h)$. We begin with this derivation of Theorem \ref{thm:threetermintegersums} from Theorem \ref{threetermintegerreducedresidues}.

In order to account for terms where $d_1,d_2,d_3$ are not necessarily distinct, we make the following definition.
\begin{definition}
Let $k \ge 2$, and let $\mathcal D = \{d_1,\dots,d_k\}$ be a $k$-tuple of not necessarily distinct integers, and fix $q$ a squarefree integer. Then the \emph{singular series at $\mathcal D$ with respect to $q$} is given by
\[\mathfrak S(\mathcal D; q) : = \sum_{q_1, \dots, q_k |q} \left(\prod_{i=1}^k \frac{\mu(q_i)}{\phi(q_i)} \right) \sum_{\substack{a_1, \dots, a_k \\ 1 \le a_i \le q_i \\ (a_i,q_i) = 1 \\ \sum a_i/q_i \in \Z}} e\left( \sum_{i=1}^k \frac{a_id_i}{q_i}\right).\]
Just as for $\S(\mathcal D)$, one can subtract off the main term of $\S(\mathcal D;q)$ to define 
\[\S_0(\mathcal D; q) := \sum_{\mathcal J \subset \mathcal D} (-1)^{|\mathcal D \setminus \mathcal J|} \S(\mathcal J;q).\]
Combining this with the definition for $\S(\mathcal D;q)$ yields the formula
\begin{equation}\label{s0dq}
\S_0(\mathcal D;q) = \sum_{1<q_1, \dots, q_k |q} \left(\prod_{i=1}^k \frac{\mu(q_i)}{\phi(q_i)} \right) \sum_{\substack{a_1, \dots, a_k \\ 1 \le a_i \le q_i \\ (a_i,q_i) = 1 \\ \sum a_i/q_i \in \Z}} e\left( \sum_{i=1}^k \frac{a_id_i}{q_i}\right).
\end{equation}
\end{definition}
If the $d_i$ are not all distinct, this expression converges for any fixed $q$ but not in the $q \to \infty$ limit. The singular series at $\mathcal D$ with respect to $q$ is equal to a finite Euler product.

\begin{lemma}\label{integerSeulerproduct}
Let $k \ge 2$, and let $\mathcal D = \{d_1,\dots,d_k\}$ be a $k$-tuple of not necessarily distinct integers, and fix $q$ a squarefree integer. Then
\[\mathfrak S(\mathcal D; q) = \prod_{p|q} \left( 1 - \frac 1p\right)^{-k}\left(1-\frac{\nu_p(\mathcal D)}{p}\right),\]
where $\nu_p(\mathcal D)$ is the number of distinct residue classes mod $p$ occupied by elements of $\mathcal D$.
\end{lemma}
This lemma is proven in \cite[Lemma 3]{MontgomerySoundararajanPrimesIntervals}; it is stated there for sets with distinct elements, but their proof holds in this setting as well.  They note first that $\mathfrak S(\mathcal D;q)$ is multiplicative in $q$, so that it suffices to check the lemma for primes $p$. For a given prime $p$, they express the condition that $\sum_{i=1}^k \frac{a_i}{q_i} \in \Z$ in terms of additive characters mod $p$, and then rearrange to get the result.

Consider the following expression for $\mathfrak S_0$, which is \cite[Equation (45)]{MontgomerySoundararajanPrimesIntervals}. For all $y \ge h$, 
\[\mathfrak S_0(\mathcal D) = \sum_{\substack{q_1, q_2, q_3 \\ q_i > 1 \\ p|q_i \Rightarrow p \le y}} \prod_{i=1}^3 \frac{\mu(q_i)}{\phi(q_i)} A(q_1, q_2, q_3;\mathcal D) + O\left(\frac{(\log y)}{y}\right), \]
where
\[A(q_1, q_2, q_3;\mathcal D) = \sum_{\substack{a_1,a_2,a_3 \\ 1 \le a_i \le q_i \\ (a_i,q_i) = 1 \\ \sum a_i/q_i \in \Z}}e\left(\sum_{i=1}^3 \frac{d_ia_i}{q_i} \right). \]
Apply this to $R_3(h)$ with $y = h^4$ and $q = \prod_{p \le y} p$ to get
\begin{align*}
R_3(h) &= \sum_{\substack{q_1, q_2, q_3 \\ q_i > 1 \\ q_i|q}} \prod_{i=1}^3 \frac{\mu(q_i)}{\phi(q_i)} S(q_1, q_2, q_3;h) + O(1),
\end{align*}
where
\[S(q_1, q_2, q_3;h) := \sum_{\substack{d_1, d_2, d_3 \\ 1 \le d_i \le h \\ d_i \text{ distinct}}} A(q_1,q_2,q_3;\{d_1,d_2,d_3\}) = \sum_{\substack{d_1, d_2, d_3 \\ 1 \le d_i \le h \\ d_i \text{ distinct}}} \sum_{\substack{a_1,a_2,a_3 \\ 1 \le a_i \le q_i \\ (a_i,q_i) = 1 \\ \sum a_i/q_i \in \Z}}e\left(\sum_{i=1}^3 \frac{d_ia_i}{q_i} \right).\]

If the condition that the $d_i$ should be distinct were omitted, then the main term in $R_3(h)$ would be exactly $V_3(q;h)$. So, it suffices to remove this condition.

Put $\d_{i,j} = 1$ if $d_i = d_j$ and $0$ otherwise, so that
\[\prod_{1 \le i < j \le 3} (1-\d_{i,j}) = \begin{cases}1 &\text{ if the $d_i$ are all pairwise distinct} \\ 0 &\text{ otherwise},\end{cases}\]
and
\[S(q_1, q_2, q_3;h) = \sum_{\substack{d_1, d_2, d_3 \\ 1 \le d_i \le h}} \left(\prod_{1 \le i < j \le 3}(1-\d_{i,j})\right) \sum_{\substack{a_1,a_2,a_3 \\ 1 \le a_i \le q_i \\ (a_i,q_i) = 1 \\ \sum a_i/q_i \in \Z}}e\left(\sum_{i=1}^3 \frac{d_ia_i}{q_i} \right).\]

Expanding the product over the $\delta_{i,j}$ yields
\[1 - \d_{1,2} - \d_{1,3} - \d_{2,3} + \d_{1,2}\d_{2,3} + \d_{1,3}\d_{1,2} + \d_{2,3}\d_{1,3} - \d_{1,2}\d_{2,3}\d_{1,3}.\]
Note that the last four terms each require precisely that $d_1 = d_2 = d_3$ in order to be nonzero; each of these can be written as $\d_{1,2,3}$, so that their sum is $2 \d_{1,2,3}$. The following lemma addresses the contribution of these last four terms.
\begin{lemma}
Let $h \ge 4$ be an integer. Then
\[2 \sum_{d \le h} \sum_{\substack{q_1,q_2,q_3 \\ q_i > 1 \\ q_i|q}}\prod_{i=1}^3 \frac{\mu(q_i)}{\phi(q_i)} \sum_{\substack{a_1,a_2, a_3 \\ 1 \le a_i \le q_i \\ (a_i,q_i) = 1 \\ \sum a
_i/q_i \in \Z}} e\left(\sum_{i=1}^3 \frac{d a_i}{q_i}\right) = 2h\left(\frac{q}{\phi(q)}\right)^2-6h\frac{q}{\phi(q)} +4h.\]
\end{lemma}
\begin{proof}
Note that the left-hand expression is precisely $2\sum_{d \le h} \S_0(\{d,d,d\};q)$. Expanding $\S_0$ and applying Lemma \ref{integerSeulerproduct} yields
\begin{align*}
2\sum_{d \le h} \S_0(\{d,d,d\};q) &= 2 \sum_{d \le h} (\S(\{d,d,d\};q) - 3\S(\{d,d\};q) + 3\S(\{d\};q) - 1) \\
&= 2 \sum_{d \le h} \left(\prod_{p|q} \left(1-\frac 1p\right)^{-2} - 3\prod_{p|q}\left(1-\frac 1p\right)^{-1} + 2\right) \\
&= 2h \frac{q^2}{\phi(q)^2} - 6h \frac{q}{\phi(q)} + 4h,
\end{align*}
as desired.
\end{proof}

Now consider the contribution to $R_3(h)$ from the terms $-\d_{1,2}$, $-\d_{1,3}$, and $-\d_{2,3}$. Via relabeling, it suffices to only consider the term with $-\d_{1,2}$, which is nonzero when $d_1 = d_2$ and otherwise $0$. 
\begin{lemma}
Let $h \ge 4$ be an integer. Then
\begin{align*}
\sum_{d, d_3 \le h} \sum_{\substack{q_1,q_2,q_3 \\ q_i > 1 \\ q_i|q}}\prod_{i=1}^3 \frac{\mu(q_i)}{\phi(q_i)} &\sum_{\substack{a_1,a_2, a_3 \\ 1 \le a_i \le q_i \\ (a_i,q_i) = 1 \\ \sum a_i/q_i \in \Z}} e\left(d\left(\frac{a_1}{q_1} + \frac{a_2}{q_2}\right)\right) e\left(\frac{d_3a_3}{q_3}\right) \\
&= \left(\frac{q}{\phi(q)} - 2\right)\left(h\frac{q}{\phi(q)} - h \log h + Bh + O(h^{1/2 + \ep})\right)
\end{align*}
\end{lemma}
\begin{proof}
As in the previous lemma, we note that the left-hand side is $\sum_{d, d_3 \le h} \S_0(\{d,d,d_3\};q)$. We again expand and apply Lemma \ref{integerSeulerproduct}, to get
\begin{align*}
\sum_{d, d_3 \le h} \S_0(\{d,d,d_3\};q) &= \sum_{d,d_3 \le h} (\S(\{d,d,d_3\};q) - 2\S(\{d,d_3\};q) - \S(\{d,d\};q) + 2) \\
&=  \left(\frac{q}{\phi(q)} - 2\right) \left(\sum_{d,d_3 \le h} \S(\{d,d_3\};q) - h^2\right).\\
\end{align*}
By \cite[Lemma 4]{MontgomerySoundararajanPrimesIntervals}, 
\[ \sum_{d,d_3 \le h} \S(\{d,d_3\};q) = \sum_{q_1|q} \frac{\mu(q_1)^2}{\phi(q_1)^2} \sum_{\substack{1\le a \le q_1 \\ (a,q_1) = 1}} \left|E \left(\frac a{q_1}\right)\right|^2 = h \frac{q}{\phi(q)} + h^2 - h \log h + Bh + O(h^{1/2 + \ep}),\]
with $B = 1-\g-\log 2\pi$. Thus our expression becomes
\[=\left(\frac{q}{\phi(q)} - 2\right)\left(h\frac{q}{\phi(q)} - h \log h + Bh + O(h^{1/2 + \ep})\right), \]
as desired.
\end{proof}

Combining these computations yields
\begin{align*}
R_3(h) = V_3(q;h) &+ 2h\left(\frac{q}{\phi(q)}\right)^2 -6h\frac{q}{\phi(q)}+4h \\
&-3\left(\frac{q}{\phi(q)}-2\right)\left(h\frac{q}{\phi(q)}-h\log h + Bh + O(h^{1/2+\ep})\right) \\
= V_3(q;h) &- h\left(\frac{q}{\phi(q)}\right)^2 + 3h\log h \frac{q}{\phi(q)} - 3Bh\frac{q}{\phi(q)} \\
&- 6h \log h + 6Bh +4h + O\left(h^{1/2+\ep}\frac{q}{\phi(q)}\right)
\end{align*}
By Theorem \ref{threetermintegerreducedresidues}, $V_3(q;h) \ll h(\log h)^5,$
so $R_3(h) \ll h(\log h)^5$, which completes the proof of Theorem \ref{thm:threetermintegersums}.

\subsection{Preparing for the proof of Theorem \ref{threetermintegerreducedresidues}}\label{subsec:lemmas-for-fraction-sums}

The rest of this section will be devoted to the proof of Theorem \ref{threetermintegerreducedresidues}; here we begin by fixing some notation and proving several preparatory lemmas. Specifically, Lemmas \ref{lem:n1n2n3boundswith-one-alpha}, \ref{lem:n1n2n3bounds-with-one-alpha-all-added}, \ref{lem:n1n2n3bounds-no-alpha-all-added}, and \ref{lem:n1n2n3bounds-with-one-alpha-n2-subtracted} are general results on adding integer reciprocals along hyperplanes. Lemmas \ref{lem:prep-for-small-case-all-denoms-big}, \ref{lem:prep-for-small-case-first-denom-small}, and \ref{lem:triple_sum_F} rely on these general results to prove bounds on specific sums that will appear in the proof of Theorem \ref{threetermintegerreducedresidues}.

We begin with a reparametrization of variables into a system of common divisors. Let $(q_1,q_2,q_3)$ be a triple in the sum in \eqref{eq:defn-of-v-3} defining $V_3(q;h)$. The contribution of the $(q_1,q_2,q_3)$ term to $V_3(q;h)$ is zero unless there are nontrivial solutions to 
\[\frac{a_1}{q_1} + \frac{a_2}{q_2} + \frac{a_3}{q_3} \in \mathbb Z,\]
or equivalently
\[a_1q_2q_3 + a_2q_1q_3 + a_3q_1q_2 \equiv 0 \mod q_1q_2q_3,\]
where $(a_i,q_i) = 1$ for all $i$. This implies that $q_1|q_2q_3$ (and likewise $q_2|q_1q_3$ and $q_3|q_1q_2$), since reducing mod $q_1$ shows that $a_1q_2q_3 \equiv 0 \mod q_1$, and by assumption $(a_1,q_1) = 1$. Since $q$ is squarefree, so are $q_1,q_2,$ and $q_3$, so we can reparametrize as follows. Let $g = \gcd(q_1,q_2,q_3)$ be the product of all primes dividing all three $q_i$'s. Define $x = \gcd(q_2/g,q_3/g)$, $y = \gcd(q_1/g,q_3/g)$, and $z = \gcd(q_1/g,q_2/g)$. Then $q_1 = gyz$, $q_2 = gxz$, and $q_3 = gxy$, with $g,x,y,z$ pairwise coprime and squarefree. This reparametrization is the same as writing the system of \emph{relative greatest common divisors} for $q_1, q_2$, and $q_3$; see for example \cite{MR4091066} for more details.

Then
\[V_3(q;h) = \sum_{\substack{g,x,y,z|q \\ gxy,gxz,gyz > 1}} \frac{\mu(g)\mu(gxyz)^2}{\phi(g)\phi(gxyz)^2} \sum_{\substack{a_1, a_2, a_3 \\ 0 \le a_1 < gyz, \dots \\ (a_1,gyz) = \cdots = 1\\ \tfrac{a_1}{gyz} + \tfrac{a_2}{gxz} + \tfrac{a_3}{gxy} \in \Z}} E\left(\frac{a_1}{gyz}\right) E\left(\frac{a_2}{gxz}\right)E\left(\frac{a_3}{gxy}\right). \]
We start by taking absolute values, using the bound that for all $0 \le \a < 1$,  $|E(\a)| \le F(\a)$, where
\begin{equation}\label{eq:defn-of-F-alpha}
F(\a) := \min\{h,\|\a\|^{-1}\},
\end{equation}
so that
\begin{equation}\label{eq:V3-bound-gxyz-and-Fs}
V_3(q;h) \le \sum_{\substack{g,x,y,z|q \\ gxy,gxz,gyz > 1}} \frac{\mu(gxyz)^2}{\phi(g)\phi(gxyz)^2} \sum_{\substack{a_1, a_2, a_3 \\ 0 \le a_1 < gyz, \dots \\ (a_1,gyz) = \cdots = 1\\ \sum a_1/gyz \in \Z}} F\left(\frac{a_1}{gyz}\right) F\left(\frac{a_2}{gxz}\right)F\left(\frac{a_3}{gxy}\right).
\end{equation}
We now split the sum $V_3(q;h)$ into three different sums, addressed separately. Let $T_1$ consist of all terms $g,x,y,z$ in \eqref{eq:V3-bound-gxyz-and-Fs} with $gx \ge h$. Let $T_2$ consist of all terms $g,x,y,z$ in \eqref{eq:V3-bound-gxyz-and-Fs} with $gx < h$, $gy < h$, and $gz < h$, and $\left\|\frac{a_2}{q_2}\right\|, \left\|\frac{a_3}{q_3}\right\| > \frac 1h$. Finally, let $T_3$ consist of all terms $g,x,y,z$ in \eqref{eq:V3-bound-gxyz-and-Fs} with $gx < h$, $gy < h$, and $gz < h$ as well as the constraints that $\left\|\frac{a_1}{gyz}\right\| \le \frac 2h$, $\left\|\frac{a_2}{gxz}\right\| \le \frac 2h$, and $\left\|\frac{a_3}{gxy}\right\| \le \frac 2h$. 

We claim that, after permuting the names of the variables as necessary, each term $g,x,y,z,a_1,a_2,a_3$ is contained in sums for $T_1$, $T_2$, or $T_3$. Terms where any of $gx$, $gy$, or $gz$ are $\ge h$ are included in a copy of $T_1$. For remaining terms we have $gx < h$, $gy < h$, and $gz < h$. If two of the three fractions $\frac{a_i}{q_i}$ satisfy $\left\|\frac{a_i}{q_i}\right\| \le \frac 1h$ (say $i = 1, 2$), then the third one must satisfy $\left\|\frac{a_3}{q_3}\right\| \le \frac 2h$ because $\frac{a_1}{q_1} + \frac{a_2}{q_2} + \frac{a_3}{q_3} \in \mathbb Z$; therefore, these terms are included up to permutation of indices in $T_3$. The remaining terms must be included, up to permuting the indices, in $T_2$. This implies in particular that
\[V_3(q;h) \ll T_1 + T_2 + T_3.\]
We will show in Lemmas \ref{lem:T1}, \ref{lem:T2}, and \ref{lem:T3} respectively that $T_1 \ll h(\log h)^5$, that $T_2 \ll h(\log h)^4(\log \log h)^2$, and that $T_3 \ll h(\log h)^4(\log \log h)^2$, which completes the proof of Theorem \ref{threetermintegerreducedresidues}. 

In what follows, it will be helpful for us to approximate fractions $\frac aq$ by a nearby multiple of $\frac 1h$; to do so, we make the following definition.

\begin{definition}\label{def:h-approximate-numerators}
Fix $h \ge 4$. Let $q > 1$ and let $1 \le a < q$ with $(a,q) = 1$. If $q > h$, the \emph{$h$-approximate numerator} $n(a,q)$ is defined to as
\begin{equation*}
n(a,q) = \left\lceil h \|a/q\| \right\rceil = \begin{cases} \left\lceil \frac{ha}{q}\right\rceil &\text{ if } \frac aq \le \frac 12 \\
\\
h - \left\lfloor \frac{ha}{q}\right\rfloor &\text{ if } \frac aq > \frac 12.\end{cases}
\end{equation*}
Meanwhile, if $q \le h$, the $h$-approximate numerator $n(a,q)$ is defined to be $a$ itself.
\end{definition}
For example, if $q > h$ and $\frac 1h < \frac aq \le \frac 2h$, say, then the $h$-approximate numerator $n(a,q) = 2$, so that $\frac 12 \frac{n(a,q)}{h} \le \frac aq \le \frac{n(a,q)}{h}$. The definition is arranged so that $n(a,q)$ is never zero when $(a,q) = 1$; if $0 < \frac aq \le \frac 1h$, then $n(a,q) = 1$. 
The key consequence of this definition is the following property.

\begin{claim}
Let $h \ge 4$. For $F(\a)$ defined in \eqref{eq:defn-of-F-alpha}, we have
\begin{equation}\label{eq:h-approximate-numerator-bound-on-F}
F\left(\frac aq \right) \le 2\left\|\frac{n(a,q)}{\min\{q,h\}}\right\|^{-1}.
\end{equation}
\end{claim}
\begin{proof}
If $q \le h$, then \eqref{eq:h-approximate-numerator-bound-on-F} states that
$\|a/q\|^{-1} \le 2\|a/q\|^{-1},
$
which is true.

For $q > h$, we restrict to considering the case when $\frac aq \in (0, \tfrac 12]$, so that $\left\|\frac aq \right\| = \frac aq$; the case when $\frac aq \in (\tfrac 12, 1)$ is analogous.
Assume first that $0 < \frac aq \le \frac 1h$. Then $F(a/q) = h$ and $n(a,q) = 1$, so that \eqref{eq:h-approximate-numerator-bound-on-F} states that $h \le 2h$, which is true. Finally assume that $\frac 1h < \frac aq$. By definition, $n(a,q) = \left\lceil ha/q\right\rceil = ha/q + e$, where $0 \le e < 1$. For any such $e$,
\begin{equation*}
\left\|\frac aq + \frac eh \right\| \le \left\|\frac aq \right\| + \frac 1h \le 2 \frac aq.
\end{equation*}
Thus
\begin{align*}
\left(\frac aq\right)^{-1} &\le 2 \left\|\frac aq + \frac eh \right\|^{-1} = 2 \left\|\frac{\lceil h \tfrac aq \rceil}{h}\right\|^{-1},
\end{align*}
which is precisely \eqref{eq:h-approximate-numerator-bound-on-F} in this case.
\end{proof}
We write $\tilde q := \min\{q,h\}$, so that $F(a/q) \le 2 \|n(a,q)/\tilde q\|^{-1}$. For any fraction $\frac aq$, we then have that $\frac aq \approx \frac{n(a,q)}{\tilde q}$ in the sense that $\left|\frac aq - \frac{n(a,q)}{\tilde q}\right| < \frac 1h$, since if $q \le h$ then $\frac aq = \frac{n(a,q)}{\tilde q}$, and if $q > h$ then this follows from the definition of $n(a,q)$.

We are now ready to proceed with several lemmas concerning sums of fractions, sums over quantities $\Big\| \frac aq \Big\|^{-1}$, and sums of $F(\a)$. The following four lemmas are general results on adding integer reciprocals of points lying close to certain hyperplanes. Loosely speaking, these lemmas will appear in our argument in the following way. For each of $T_1$, $T_2$, and $T_3$, we will have to evaluate a sum of the form
\begin{equation*}
\sum_{\substack{a_1,a_2,a_3 \\ \tfrac{a_1}{q_1} + \tfrac{a_2}{q_2} + \tfrac{a_3}{q_3} \in \mathbb Z}} F\left(\frac{a_1}{q_1}\right)F\left(\frac{a_2}{q_2}\right)F\left(\frac{a_3}{q_3}\right),
\end{equation*}
where in practice there will be further constraints on the terms $a_i$ and $q_i$. After applying \eqref{eq:h-approximate-numerator-bound-on-F} and the observation that $\frac aq \approx \frac{n(a,q)}{\tilde q}$, and dealing with a little casework on the sign of $n(a_i,q_i)$, we arrive at a sum that is roughly of the form
\begin{equation*}
8\prod_{i=1}^3 \min\{q_i,h\} \sum_{\substack{a_1,a_2,a_3 \\ \left\|\tfrac{n(a_1,q_1)}{\tilde{q_1}} + \tfrac{n(a_2,q_2)}{\tilde{q_2}} + \tfrac{n(a_3,q_3)}{\tilde{q_3}}\right\| \approx 0}} \frac 1{n(a_1,q_1)n(a_2,q_2)n(a_3,q_3)}.
\end{equation*}
In particular, in order to analyze $T_1$, $T_2$, $T_3$, we will have to understand sums of reciprocals of lattice points. Understanding the precise sums requires some amount of casework, largely coming from the cases $q_i < h$ versus $q_i \ge h$ and the cases $\tfrac{a_i}{q_i} \le \tfrac 12$ versus $\tfrac{a_i}{q_i} > \tfrac 12$. This casework is accomplished by the Lemmas \ref{lem:n1n2n3boundswith-one-alpha}, \ref{lem:n1n2n3bounds-no-alpha-all-added}, and \ref{lem:n1n2n3bounds-with-one-alpha-n2-subtracted}.

\begin{lemma}\label{lem:n1n2n3boundswith-one-alpha}
Let $\nu_2 \ge \nu_1$ and $\a_1 \ge 1$ be real numbers, and let $h \in \N$ with $h \ge 4$. Then
\begin{equation*}
\sum_{\substack{1 \le n_1 \le h/(2\a_1) \\ 1 \le n_2 \le h/2 \\ 1 \le n_3 \le h/2 \\ -\a_1n_1 + n_2 + n_3 \in [\nu_1,\nu_2]}} \frac 1{\a_1n_1n_2n_3} \ll \begin{cases}(\nu_2 - \nu_1 + 1) \frac{\log h}{\a_1}\left(\frac{2-\nu_1}{\a_1} + 1\right) &\text{ if $\nu_1 < 0$} \\
(\nu_2-\nu_1 + 1)\frac{\log h}{\a_1^2} &\text{ if $\nu_1 \ge 0$}, \end{cases}
\end{equation*}
where $n_1, n_2,$ and $n_3$ range over integers.
\end{lemma}
\begin{proof}
Since $n_2 + n_3 \ge \nu_1 + \a_1n_1$ and $n_2 + n_3 \ge 2$,
\begin{equation*}
\frac 1{\a_1n_1n_2n_3} = \frac 1{\a_1n_1(n_2 + n_3)}\left(\frac 1{n_2} + \frac 1{n_3}\right) \le \frac 1{\a_1n_1\max\{2,\nu_1 + \a_1n_1\}}\left(\frac 1{n_2} + \frac 1{n_3}\right).
\end{equation*}
The sum is then bounded by
\begin{align*}
\sum_{\substack{1 \le n_1 \le h/(2\a_1) \\ 1 \le n_2 \le h/2 \\ 1 \le n_3 \le h/2 \\ -\a_1n_1 + n_2 + n_3 \in [\nu_1,\nu_2]}} \frac 1{\a_1n_1n_2n_3}&\le \sum_{1 \le n_1 \le h/(2\a_1)} \frac{1}{\a_1n_1\max\{\nu_1 + \a_1n_1,2\}} \sum_{\substack{1 \le n_2 \le h/2 \\ 1 \le n_3 \le h/2 \\ -\a_1n_1 + n_2 + n_3\in [\nu_1, \nu_2]}} \frac 1{n_2} + \frac{1}{n_3} \\
&= \sum_{1 \le n_1 \le h/(2\a_1)} \frac{1}{\a_1n_1\max\{\nu_1 + \a_1n_1,2\}} \sum_{\substack{1 \le n_2 \le h/2 \\ 1 \le n_3 \le h/2 \\-\a_1n_1 + n_2 + n_3 \in [\nu_1,\nu_2]}} \frac {2}{n_2},
\end{align*}
where equality follows because the roles of $n_2$ and $n_3$ are symmetric. For fixed values of $n_1$ and $n_2$, the integer $n_3$ must satisfy $1 \le n_3 \le h/2$ and $n_3 \in [\nu_1 + \a_1n_1-n_2, \nu_2 + \a_1n_1 - n_2]$; the number of valid choices of $n_3$ is $\ll \nu_2 - \nu_1 + O(1)$. Thus the sum is
\begin{align*}
&\ll (\nu_2 - \nu_1 + 1)\sum_{\substack{1 \le n_1 \le h/(2\a_1)}} \frac 1{\a_1n_1\max\{\nu_1 + \a_1n_1,2\}}\sum_{\substack{1 \le n_2 \le h/2}} \frac 1{n_2} \\
&\ll (\nu_2 - \nu_1 + 1)\log h \sum_{1 \le n_1 \le h/(2\a_1)} \frac 1{\a_1n_1\max\{\nu_1 + \a_1n_1,2\}}. \\
\end{align*}

If $\nu_1 \ge 0$, then $\frac{\nu_1}{\a_1} + n_1 \ge 1$ and the sum is
\begin{align*}
&\ll (\nu_2 - \nu_1 + 1)\log h \sum_{1 \le n_1 \le h/(2\a_1)} \frac{1}{\a_1n_1(\nu_1 + \a_1n_1)} \\
&\ll (\nu_2 - \nu_1 + 1)\frac{\log h}{\a_1^2} \sum_{1 \le n_1 \le h/(2\a_1)} \frac 1{n_1(\tfrac{\nu_1}{\a_1} + n_1)} \ll (\nu_2 - \nu_1 + 1) \frac{\log h}{\a_1^2},
\end{align*}
since the sum over $n_1$ is bounded by $\sum_{n=1}^\infty \frac 1{n^2}$, and thus by a constant. This completes the proof for this case.

On the other hand, if $\nu_1 < 0$, then the sum is
\begin{align*}
&\ll (\nu_2-\nu_1 + 1)\log h \Big(\sum_{\substack{1 \le n_1 \le h/(2\a_1) \\ n_1 < \tfrac{2-\nu_1}{\a_1} + 1}} \frac 1{\a_1n_1} + \sum_{\substack{1 \le n_1 \le h/(2\a_1) \\ \nu_1 + \a_1n_1 \ge 2 + \a_1}} \frac 1{\a_1n_1(\nu_1 + \a_1n_1)} \Big) \\
&\ll (\nu_2 - \nu_1 + 1)\log h \Big( \frac 1{\a_1}\Big(\frac{2-\nu_1}{\a_1} + 1\Big) + \frac 1{\a_1^2}\sum_{\substack{1 \le n_1 \le h/(2\a_1) \\ \tfrac{\nu_1}{\a_1} + n_1 \ge \tfrac{2}{\a_1} + 1}} \frac 1{n_1(\tfrac{\nu_1}{\a_1} + n_1)}\Big).
\end{align*}
The final sum is bounded by $\sum_{n=1}^\infty \frac 1{n^2}$, and thus by a constant. This completes the proof.
\end{proof}

\begin{lemma}\label{lem:n1n2n3bounds-with-one-alpha-all-added}
Let $\nu_2 \ge \nu_1 \ge 3$ and $\a_1 \ge 1$ be real numbers, and let $h \in \N$ with $h \ge 4$. Then
\begin{equation*}
\sum_{\substack{1 \le n_1 \le h/(2\a_1) \\ 1 \le n_2 \le h/2 \\ 1 \le n_3 \le h/2 \\ \a_1n_1 + n_2 + n_3 \in [\nu_1,\nu_2]}} \frac 1{\a_1n_1n_2n_3} \ll  \frac{(\nu_2-\nu_1 + 1)}{\nu_1}\log \min\{\nu_2,h\} \left(\nu_2 - \nu_1 + 1 + \frac 1{\a_1} \log \min\{\nu_1,h\}\right),
\end{equation*}
where $n_1, n_2,$ and $n_3$ range over integers.
\end{lemma}
\begin{proof}
The first part of this proof follows along identical lines to that of Lemma \ref{lem:n1n2n3boundswith-one-alpha}, but with $\a_1$ having opposite signs. By following the first part of the argument of Lemma \ref{lem:n1n2n3boundswith-one-alpha}, we get that the sum we want to bound is
\begin{align*}
&\ll (\nu_2 - \nu_1 + 1)\sum_{\substack{1 \le n_1 \le h/(2\a_1)\\ n_1 \le (\nu_2-2)/\a_1}} \frac 1{\a_1n_1\max\{\nu_1 - \a_1n_1,2\}}\sum_{\substack{1 \le n_2 \le h/2 \\ n_2 \le \nu_2 - \a_1n_1}} \frac 1{n_2} \\
&\ll (\nu_2 - \nu_1 + 1)\log \min\{\nu_2,h\} \sum_{\substack{1 \le n_1 \le h/(2\a_1)\\ n_1 \le (\nu_2-2)/\a_1}} \frac 1{\a_1n_1\max\{\nu_1 - \a_1n_1,2\}}.
\end{align*}

If $\max\{\nu_1 - \a_1n_1,2\} = 2$, then $\nu_1 - 2 < \a_1n_1 \le \nu_2-2$. The number of such terms is $\ll \nu_2 - \nu_1$, and for these terms the summand is $\frac 1{2\a_1n_1} \ll \frac 1{\nu_1}$, so these terms provide an overall contribution of size $\ll (\nu_2-\nu_1 + 1) \log \min\{\nu_2,h\} \frac {\nu_2-\nu_1}{\nu_1}$. For the remaining terms, $\a_1 n_1 \le \nu_1-2$.

We rewrite $\frac 1{\a_1n_1(\nu_1 - \a_1n_1)} = \frac 1{\nu_1 \a_1n_1} + \frac 1{\nu_1(\nu_1 - \a_1n_1)}$, so that for the remaining terms we have
\begin{align*}
\sum_{\substack{1 \le n_1 \le h/(2\a_1)\\ n_1 \le (\nu_1-2)/\a_1}} \frac {1}{\a_1n_1\max\{\nu_1 - \a_1n_1,2\}} 
&= \sum_{\substack{1 \le n_1 \le h/(2\a_1) \\ n_1 \le (\nu_1 - 2)/\a_1}} \left(\frac 1{\nu_1 \a_1n_1} + \frac 1{\nu_1 (\nu_1 - \a_1n_1)}\right) \\
&\ll \frac {1}{\nu_1\a_1} \log \min\{\nu_1,h\} + \frac 1{\nu_1}\left(1 + \frac 1{\a_1} \log \min\{\nu_1,h\}\right).
\end{align*}
This completes the proof.
\end{proof}

\begin{lemma}\label{lem:n1n2n3bounds-with-one-alpha-n2-subtracted}
Let $\a_1 \ge 1$ and $\nu_2 \ge \nu_1$ be (possibly negative) real numbers, and let $h \in \N$ with $h \ge 4$. Then
\begin{equation*}
\sum_{\substack{1 \le n_1 \le h/(2\a_1) \\ 1 \le n_2 \le h/2 \\ 1 \le n_3 \le h/2 \\ \a_1n_1 - n_2 + n_3 \in [\nu_1,\nu_2]}} \frac 1{\a_1n_1n_2n_3} \ll (\nu_2 - \nu_1 + 1) \left(\frac{\log h}{\a_1} + 1\right)\frac{\log\max\{\nu_1,\a_1 + 1\} + 1}{\max\{\nu_1,\a_1\} + 1},
\end{equation*}
where $n_1, n_2,$ and $n_3$ range over integers.
\end{lemma}
\begin{proof}
Since $\a_1n_1 + n_3 \ge \nu_1 + n_2$ and $\a_1n_1 + n_3 \ge \a_1 + 1$, we have
\begin{equation*}
\frac 1{\a_1n_1n_2n_3} = \frac 1{n_2(\a_1n_1 + n_3)}\left(\frac 1{\a_1n_1} + \frac 1{n_3}\right) \le \frac 1{n_2\max\{\nu_1 + n_2,\a_1 + 1\}}\left(\frac 1{\a_1n_1} + \frac 1{n_3}\right).
\end{equation*}
The sum is then bounded by
\begin{align*}
\sum_{\substack{1 \le n_1 \le h/(2\a_1) \\ 1 \le n_2 \le h/2 \\ 1 \le n_3 \le h/2 \\ \a_1n_1 - n_2 + n_3 \in [\nu_1,\nu_2]}} \frac 1{\a_1n_1n_2n_3}&\le \sum_{1 \le n_2 \le h/2} \frac{1}{n_2\max\{\nu_1 + n_2,\a_1 + 1 \}} \sum_{\substack{1 \le n_1\le h/(2\a_1) \\ 1 \le n_3 \le h/2 \\ \a_1n_1 - n_2 + n_3\in [\nu_1, \nu_2]}} \left(\frac 1{\a_1n_1} + \frac{1}{n_3}\right). 
\end{align*}
For fixed values of $n_1$ and $n_2$, the integer $n_3$ must satisfy $1 \le n_3 \le h/2$ and $n_3 \in [\nu_1 - \a_1n_1+n_2, \nu_2 - \a_1n_1 + n_2]$; the number of valid choices of $n_3$ is $\ll \nu_2 - \nu_1 + O(1)$. Thus
\begin{align*}
&\sum_{1 \le n_2 \le h/2 } \frac 1{n_2\max\{\nu_1 + n_2,\a_1 + 1 \}} \sum_{\substack{1 \le n_1 \le h/(2\a_1) \\ 1 \le n_3 \le h/2 \\ \a_1n_1 - n_2 + n_3 \in [\nu_1,\nu_2]}} \frac 1{\a_1n_1} \\
&\ll \frac{(\nu_2 -\nu_1 + 1)}{\a_1} \log h \sum_{1 \le n_2 \le h/2 } \frac 1{n_2\max\{\nu_1 + n_2,\a_1 + 1 \}} \\
&\ll (\nu_2 - \nu_1 + 1)\frac{\log h}{\a_1} \frac{\log\max\{\nu_1,\a_1 + 1\} + 1}{\max\{\nu_1,\a_1\} + 1}.
\end{align*}

It remains to evaluate the $\frac 1{n_3}$ term in the sum. Since $n_3 \ge \nu_1-\a_1n_1 + n_2$, we have
\begin{align*}
&\sum_{1 \le n_2 \le h/2} \frac 1{n_2\max\{\nu_1 + n_2,\a_1 + 1 \}} \sum_{\substack{1 \le n_1 \le h/(2\a_1) \\ 1 \le n_3 \le h/2 \\ \a_1n_1-n_2 + n_3 \in [\nu_1,\nu_2]}} \frac 1{n_3} \\
&\ll \sum_{1 \le n_2 \le h/2} \frac 1{n_2\max\{\nu_1 + n_2,\a_1 + 1 \}} \sum_{\substack{1 \le n_1 \le h/(2\a_1)}} \frac{\nu_2 - \nu_1 + 1}{\lceil \nu_1 -\a_1n_1 + n_2\rceil} \\
&\ll \sum_{1 \le n_2 \le h/2} \frac {\nu_2 - \nu_1 + 1}{n_2\max\{\nu_1 + n_2, \a_1 + 1\}}\left(\frac{\log h}{\a_1}+ 1\right) \\
&\ll (\nu_2 - \nu_1 + 1) \left(\frac{\log h}{\a_1} + 1\right) \frac{\log\max\{\nu_1,\a_1 + 1\} + 1}{\max\{\nu_1,\a_1\} + 1}.
\end{align*}
This completes the proof.
\end{proof}

If $\a_1 = 1$, we have the following stronger bound.
\begin{lemma}\label{lem:n1n2n3bounds-no-alpha-all-added}
There exist absolute constants $C$ and $D$ such that for all integers $\nu \ge 3$ and $h \ge 4$,
\begin{equation*}
\sum_{\substack{1 \le n_1 \le \nu-2 \\ 1 \le n_2 \le \nu-2 \\ 1 \le n_3 \le \nu-2 \\ n_1 + n_2 + n_3 = \nu}} \frac 1{n_1n_2n_3} \le C \qquad \text{and} \qquad \sum_{\substack{1 \le n_1 \le h \\ 1 \le n_2 \le \nu + h \\ 1 \le n_3 \le \nu + h \\ n_2 + n_3 = \nu + n_1}} \frac 1{n_1n_2n_3} \le D 
\end{equation*}
where the sum ranges over integer values of $n_1, n_2, n_3$.
\end{lemma}
\begin{proof}
For real numbers $x, x'\ge 1$ with $|x-x'| \le 1$, we have $\left|\frac 1x - \frac 1{x'}\right| \le \frac 2x$. Thus 
\begin{align*}
\sum_{\substack{1 \le n_1 \le h/2 \\ 1 \le n_2 \le h/2 \\ 1 \le n_3 \le h/2 \\ n_1 + n_2 + n_3 = \nu}} \frac 1{n_1n_2n_3} &\le 8 \int_1^{\nu-2}\int_{1}^{\nu-x_1-1} \frac 1{x_1x_2(\nu-x_1-x_2)}\mathrm{d}x_2 \mathrm{d}x_1 \\
&= 8 \int_1^{\nu-2} \frac {2\ln(\nu - x_1 - 1)}{x_1(\nu-x_1)}  \mathrm dx_1 \\
&\le 16 \ln \nu \int_1^{\nu-2} \frac 1{x_1(\nu-x_1)} \mathrm{d}x_1 \\
&= 16 \ln \nu \frac{2 \ln(\nu-1)}{\nu} = 32 \frac{(\ln \nu)(\ln(\nu-1))}{\nu}. 
\end{align*}
The function $\frac{(\ln \nu)(\ln (\nu-1))}{\nu}$ has a global maximum $M$; setting $C = 16M$ completes the proof of the first claim.

For the second claim, we similarly have
\begin{align*}
\sum_{\substack{1 \le n_1 \le h \\ 1 \le n_2 \le \nu + h \\ 1 \le n_3 \le \nu + h \\ n_2 + n_3 = \nu + n_1}} \frac 1{n_1n_2n_3} &\le 8 \int_1^h \int_1^{\nu + x_1 - 1} \frac 1{x_1x_2(\nu + x_1-x_2)} \mathrm dx_2 \mathrm dx_1 \\ 
&= 16 \int_1^h \frac{\ln(\nu + x_1-1)}{x_1(\nu + x_1)} \mathrm dx_1 \\
&\le 16 D_1 + 16 \int_{10}^h \frac{\ln(x_1 - 1)}{x_1^2}\mathrm dx_1,
\end{align*}
for some constant $D_1$, since $\frac{\ln(x-1)}{x}$ is decreasing for $x\ge 10$. The integral converges to a constant as $h \to \infty$, so setting $D = 16 D_1 + 16 \int_{10}^\infty \frac{\ln(x_1-1)}{x_1^2}\mathrm dx_1$ completes the proof.
\end{proof}

The next two lemmas concern triple sums over $\left\|\frac aq \right\|^{-1}$, which arise because of their role in the definition of $F(\a)$ and make use of the previous four lemmas.

\begin{lemma}\label{lem:prep-for-small-case-all-denoms-big}
Fix an integer $h \ge 4$. Then
\begin{equation*}
\sum_{\substack{1 \le n_i \le h - 1 \\ \left\| \sum_i n_i/h\right\| \le 3/h}} \left\|\frac{n_1}{h}\right\|^{-1} \left\|\frac{n_2}{h}\right\|^{-1} \left\|\frac{n_3}{h}\right\|^{-1} \ll h^3,
\end{equation*}
where $n_1, n_2,$ and $n_3$ range over integers.
\end{lemma}
\begin{proof}
We will split into cases based on whether $n_i \le h/2$ or $n_i > h/2$, i.e. based on the value of $\left\|\frac{n_i}{h}\right\|$. 

Assume first that $1 \le n_i \le h/2$ for all $i=1,2,3$. Then $\left\|\frac{n_i}{h}\right\| = \frac{n_i}{h}$, so we have
\begin{align*}
\sum_{\substack{1 \le n_i \le h/2\\ \left\| \sum_i n_i/h \right\| \le 3/h}} \left\|\frac{n_1}{h}\right\|^{-1} \left\|\frac{n_2}{h}\right\|^{-1} \left\|\frac{n_3}{h}\right\|^{-1} &= h^3\sum_{\substack{1 \le n_i \le h/2 \\ \left\| \sum_i n_i/h\right\| \le 3/h}} \frac{1}{n_1n_2n_3}.
\end{align*}
In order to satisfy $\left\| \sum_i n_i/h\right\| \le 3/h$, we must have $n_1 + n_2 + n_3 \in \{3\} \cup [h-3,h+3] \cup [2h-3,2h+3] \cup \{3h-3\}$. There are finitely many possible integer values for $n_1 + n_2 + n_3$; for each one, by Lemma \ref{lem:n1n2n3bounds-no-alpha-all-added}, the sum over $\frac 1{n_1n_2n_3}$ is bounded by an absolute constant. Thus the lemma holds in this case.

Now consider terms where $h/2 < n_i \le h-1$ for all $i$. For each $i$, define $m_i = h-n_i$, so that $1 \le m_i \le h/2$. Then $\left\|\frac{n_i}{h}\right\| = \frac{m_i}{h}$, and $\left\|\sum_i \frac{n_i}{h}\right\| = \left\|3h - \sum_i \frac{m_i}{h}\right\| = \left\|\sum_i \frac{m_i}{h}\right\|$. Then
\begin{equation*}
\sum_{\substack{h/2 < n_i \le h-1 \\ \left\| \sum_i n_i/h \right\| \le 3/h}} \left\|\frac{n_1}{h}\right\|^{-1} \left\|\frac{n_2}{h}\right\|^{-1} \left\|\frac{n_3}{h}\right\|^{-1} \ll  \sum_{\substack{1 \le m_i \le h/2 \\ \left\|\sum_i m_i/h\right\| \le 3/h}} \left\|\frac{m_1}{h}\right\|^{-1}\left\|\frac{m_2}{h}\right\|^{-1} \left\|\frac{m_3}{h}\right\|^{-1},
\end{equation*}
which is precisely the previous case, since $1 \le m_i \le h/2$ for all $i$. Thus this case is also $\ll h^3$.

Finally consider terms where for some $i$, $n_i \in [1,h/2]$, whereas for others $n_i \in (h/2, h-1]$. As in the previous paragraph, we can always flip all three $n_i$'s with $h-n_i$. Moreover, the roles of $n_1,n_2,$ and $n_3$ are entirely symmetric. Thus it suffices to bound those terms where $n_2,n_3 \in [1,h/2]$ and $n_1 \in (h/2,h-1]$. Set $m_1 = h-n_1$. Then
\begin{align*}
\sum_{\substack{h/2 < n_1 \le h-1 \\ 1 \le n_2,n_3 \le h - 1  \\ \left\| \sum_i n_i/h \right\| \le 3/h}} \left\|\frac{n_1}{h}\right\|^{-1} \left\|\frac{n_2}{h}\right\|^{-1} \left\|\frac{n_3}{h}\right\|^{-1} &= h^3\sum_{\substack{1 \le m_1 \le h/2 \\ 1 \le n_2,n_3 \le h/2 \\ \left\| -m_1/h + n_2/h + n_3/h\right\| \le 3/h}} \frac{1}{m_1n_2n_3}.
\end{align*}
Just as before, there are finitely many possible integer values for $-m_1 + n_2 + n_3$ satisfying the constraint that $\left\|\sum_i n_i/h\right\| \le 3/h$. For each value $\nu$, by Lemma \ref{lem:n1n2n3bounds-no-alpha-all-added}, the sum 
\begin{equation*}
\sum_{\substack{1 \le m_1 \le h/2 \\ 1 \le n_2, n_3 \le h/2 \\ -m_1 + n_2 + n_3 = \nu}} \frac 1{m_1n_2n_3}
\end{equation*}
is bounded by a constant, which completes the proof.
\end{proof}

\begin{lemma}\label{lem:prep-for-small-case-first-denom-small}
Let $h \ge 4$ and $1 \le q_1 < h$ be integers. Then
\begin{align*}
&\sum_{\substack{1 \le n_1 \le q_1 - 1 \\ 1 \le n_2,n_3 \le h - 1  \\ \left\| n_1/q_1 + n_2/h + n_3/h\right\| \le 3/h}} \left\|\frac{n_1}{q_1}\right\|^{-1} \left\|\frac{n_2}{h}\right\|^{-1} \left\|\frac{n_3}{h}\right\|^{-1} \ll h^2q_1(\log h),
\end{align*}
where $n_1, n_2,$ and $n_3$ range over integers.
\end{lemma}
\begin{proof}
We will split into cases based on whether each of $\frac{n_1}{q_1}, \frac{n_2}{h}$, and $\frac{n_3}{h}$ lie in $(0,1/2]$ or $(1/2,1)$; for each cases, we will show that the bound holds.
Assume first that all three of $\frac{n_1}{q_1}$, $\frac{n_2}{h}$, and $\frac{n_3}{h}$ lie in $(0,1/2]$.
Note that $\tfrac{n_1}{q_1} +\tfrac{n_2}{h} + \tfrac{n_3}{h} \ge \tfrac 1{q_1} + \tfrac 2h > \tfrac 3h$, so the constraint that $\left\|n_1/q_1 + n_2/h + n_3/h \right\| \le 3/h$ is equivalent to the constraint that
\begin{align*}
\frac{n_1}{q_1} + \frac{n_2}{h} + \frac{n_3}{h} &\in [1-\tfrac 3h, 1 + \tfrac 3h] \cup [2-\tfrac 3h, 2 + \tfrac 3h] \cup [3-\tfrac 3h, 3] \\
\Leftrightarrow \frac{h}{\tilde q_1} n_1 + n_2 + n_3 &\in [h-3,h+3] \cup [2h-3,2h+3] \cup [3h-3,3h].
\end{align*}
These are finitely many intervals, each of bounded size.
Thus these terms are given by
\begin{align*}
\sum_{\substack{1 \le n_1 \le q_1/2 \\ 1 \le n_2,n_3 \le h/2\\ \left\|n_1/q_1 + n_2/h + n_3/h\right\| \le 3/h}} \frac{q_1h^2}{n_1n_2n_3} 
&= \sum_{\substack{[\nu_1,\nu_2] \in \{[h-3,h+3], \\ [2h-3,2h+3], [3h-3,3h]\}}}\sum_{\substack{1 \le n_1 \le q_1/2 \\ 1 \le n_2,n_3 \le h/2 \\ \tfrac{h}{q_1}n_1 + n_2 + n_3\in [h-3,h+3]}} \frac{h^3}{\tfrac{h}{q_1}n_1n_2n_3}.
\end{align*}
We apply Lemma \ref{lem:n1n2n3bounds-with-one-alpha-all-added}, with $\a_1 = h/q_1$ and $[\nu_1,\nu_2] = [h-3,h+3], [2h-3,2h+3],$ or $[3h-3,3h]$, respectively. By Lemma \ref{lem:n1n2n3bounds-with-one-alpha-all-added}, each of these three terms is 
\begin{align*}
&\ll h^3 \frac 1h \log h \Big(1 + \frac{\log h}{\a_1}\Big) \ll h^2\log h\Big(1 + \frac{q_1 \log h}{h}\Big),
\end{align*}
which is $\ll h^2q_1 \log h$, as desired. 

Now assume that all three of $\frac{n_1}{q_1}$, $\frac{n_2}{h}$, and $\frac{n_3}{h}$ lie in $(1/2,1)$. Define $m_1 = q_1 - n_1$, $m_2 = h-n_2$, and $m_3 = h-n_3$, so that 
\begin{equation*}
\sum_{\substack{q_1/2 < n_1 \le q_1-1 \\ h/2 < n_2, n_3 \le h-1 \\ \left\|\tfrac{n_1}{q_1} + \tfrac{n_2}{h} + \tfrac{n_3}{h}\right\| \le \tfrac 3h}}  \left\|\frac{n_1}{q_1}\right\|^{-1} \left\|\frac{n_2}{h}\right\|^{-1} \left\|\frac{n_3}{h}\right\|^{-1} = \sum_{\substack{1 \le m_1 \le q_1/2 \\ 1 \le m_2,m_3 \le h/2 \\ \left\|\tfrac{m_1}{q_1} + \tfrac{m_2}{h} + \tfrac{m_3}{h}\right\| \le \tfrac 3h}} \frac{h^3}{\tfrac h{q_1}m_1m_2m_3}.
\end{equation*}
This is identical to the previous case, which we have already shown to be $\ll h^2q_1\log h$.

We now tackle the cases where not all fractions lie in the same half of $(0,1)$. Assume that $\frac{n_1}{q_1} \in (1/2,1)$ but $\frac{n_2}{h}, \frac{n_3}{h} \in (0,1/2]$. Define $m_1 = q_1 - n_1$, so that
\begin{align*}
\sum_{\substack{q_1/2 < n_1 \le q_1-1 \\ 1 \le n_2, n_3 \le h/2 \\ \left\|\tfrac{n_1}{q_1} + \tfrac{n_2}{h} + \tfrac{n_3}{h}\right\| \le \tfrac 3h}}  \left\|\frac{n_1}{q_1}\right\|^{-1} \left\|\frac{n_2}{h}\right\|^{-1} \left\|\frac{n_3}{h}\right\|^{-1} &= \sum_{\substack{1 \le m_1 \le q_1/2 \\ 1 \le n_2,n_3 \le h/2\\ \left\|-\tfrac{m_1}{q_1} + \tfrac{n_2}{h} + \tfrac{n_3}{h}\right\| \le \tfrac 3h}}  \frac{h^3}{\tfrac{h}{q_1}m_1 n_2n_3}.
\end{align*}
The constraint that $\left\|-\tfrac{m_1}{q_1} + \tfrac{n_2}{h} + \tfrac{n_3}{h}\right\| \le \tfrac 3h$ is equivalent to the constraint that $-\tfrac{h}{q_1}m_1 + n_2 + n_3$ lies in one of the intervals $[-3,3]$ or $[h-3,h+3]$. 
Applying Lemma \ref{lem:n1n2n3boundswith-one-alpha} to the sum over $m_1,n_2,n_3$, with $\a_1 = \frac{h}{q_1}$ and $[\nu_1,\nu_2]$ equal to each of these intervals respectively, we get that
\begin{equation*}
\sum_{\substack{q_1/2 < n_1 \le q_1-1 \\ 1 \le n_2,n_3 \le h/2 \\ \left\|\tfrac{n_1}{q_1} + \tfrac{n_2}{h} + \tfrac{n_3}{h}\right\| \le \tfrac 3h}}  \left\|\frac{n_1}{q_1}\right\|^{-1} \left\|\frac{n_2}{h}\right\|^{-1} \left\|\frac{n_3}{h}\right\|^{-1} \ll h^3 \frac{\log h}{(h/q_1)}\left(1 + \frac 1{(h/q_1)}\right) \ll h^2q_1\log h.
\end{equation*}

If $\frac{n_1}{q_1} \in (0,1/2]$ but $\frac{n_2}{h}, \frac{n_3}{h} \in (1/2,1)$, then we can once again replace $n_1$ by $m_1 = q_1 - n_1$, $n_2$ by $m_2 = h-n_2$, and $n_3$ by $m_3 = h-n_3$ to revert to the previous case.

Finally assume that $\frac{n_1}{q_1} \in (0,1/2]$, $\frac{n_2}{h}\in(1/2,1)$, and $\frac{n_3}{h} \in (0,1/2]$. The roles of $n_2$ and $n_3$ are symmetric, and we can always replace all three $n_i$'s by the corresponding $m_i$ value, so this is the only remaining case. 

Define $m_2 = h-n_2$, so that
\begin{align*}
\sum_{\substack{1 \le n_1 \le q_1/2 \\ h/2 < n_2 \le h-1 \\ 1\le n_3 \le h/2 \\ \left\|\tfrac{n_1}{q_1} + \tfrac{n_2}{h} + \tfrac{n_3}{h}\right\| \le \tfrac 3h}}  \left\|\frac{n_1}{q_1}\right\|^{-1} \left\|\frac{n_2}{h}\right\|^{-1} \left\|\frac{n_3}{h}\right\|^{-1} &= \sum_{\substack{1 \le n_1 \le q_1/2 \\ 1 \le m_2 \le h/2 \\ 1\le n_3 \le h/2 \\ \left\|\tfrac{n_1}{q_1} - \tfrac{m_2}{h} + \tfrac{n_3}{h}\right\| \le \tfrac 3h}}  \frac{h^3}{\tfrac{h}{q_1}n_1 m_2n_3}.
\end{align*}
The constraint that $\left\|\tfrac{n_1}{q_1} - \tfrac{m_2}{h} + \tfrac{n_3}{h}\right\| \le 3h$ is equivalent to the constraint that $-\tfrac{h}{q_1}n_1 - m_2 + m_3$ lies in one of the intervals $[-3,3]$ or $[h-3,h+3]$. Applying Lemma \ref{lem:n1n2n3bounds-with-one-alpha-n2-subtracted} to the sum over $n_1,m_2,n_3$ with $\a_1 = \frac{h}{q_1}$ and $[\nu_1,\nu_2]$ equal to each of these intervals respectively, we get that
\begin{align*}
\sum_{\substack{1 \le n_1 \le q_1/2 \\ h/2 < n_2 \le h-1 \\ 1\le n_3 \le h/2 \\ \left\|\tfrac{n_1}{q_1} + \tfrac{n_2}{h} + \tfrac{n_3}{h}\right\| \le \tfrac 3h}}  \left\|\frac{n_1}{q_1}\right\|^{-1} \left\|\frac{n_2}{h}\right\|^{-1} \left\|\frac{n_3}{h}\right\|^{-1} &\ll h^3 \left(\frac{q_1 \log h}{h} + 1\right) \frac{q_1\log (h/q_1 + 1)}{h}.
\end{align*}
Since $\tfrac{\log x}{x}$ is uniformly bounded for $x \ge 1$, we have $\tfrac{q_1}{h} \log \tfrac{h}{q_1} \ll 1$, so these terms are also $\ll h^2q_1\log h$, which completes the proof.

\end{proof}

Finally, the following lemma directly bounds a sum over triple products of $F(a_i/q_i)$.

\begin{lemma}\label{lem:triple_sum_F}
Let $h \in \N$ with $h \ge 4$ and let $d_1\ge 1$ and $d_2\ge 2$ be positive integers with $d_1|d_2$ and $d_2 < h$. Then
\begin{equation*}
\sum_{\substack{1 \le n_1 < d_1 \\ 1 \le n_2 < d_2}} F\left(\frac{n_1}{d_1}\right) F\left(\frac{n_2}{d_2}\right) F\left(\frac{n_1}{d_1} - \frac{n_2}{d_2}\right) \ll hd_1^2 + d_1^2d_2 \log d_2,
\end{equation*}
where $n_1$ and $n_2$ range over integers.
\end{lemma}
\begin{proof}
Write $f:= \frac{d_2}{d_1}$. Then $\frac{n_1}{d_1} - \frac{n_2}{d_2} = \frac{fn_1 - n_2}{d_2}$. Since $d_2 < h$, $F\left(\frac{fn_1-n_2}{d_2}\right) = \left\|\frac{fn_1-n_2}{d_2}\right\|^{-1}$ unless $fn_1-n_2 = 0$. Moreover, in the range where $1 \le n_1 < d_1$ and $1 \le n_2 < d_2$, $F\left(\frac{n_1}{d_1}\right) = \left\|\frac{n_1}{d_1}\right\|^{-1}$ and $F\left(\frac{n_2}{d_2}\right) = \left\|\frac{n_2}{d_2}\right\|^{-1}$. Thus
\begin{multline*}
\sum_{\substack{1 \le n_1 < d_1 \\ 1 \le n_2 < d_2}} F\left(\frac{n_1}{d_1}\right) F\left(\frac{n_2}{d_2}\right) F\left(\frac{n_1}{d_1} - \frac{n_2}{d_2}\right) \\
= \sum_{\substack{1 \le n_1 < d_1 \\ 1 \le n_2 < d_2 \\ fn_1 = n_2}} h\left\|\frac{n_1}{d_1}\right\|^{-1}\left\|\frac{n_2}{d_2}\right\|^{-1} + \sum_{\substack{1 \le n_1 < d_1 \\ 1 \le n_2 < d_2 \\ fn_1 \ne n_2}} \left\|\frac{n_1}{d_1}\right\|^{-1}\left\|\frac{n_2}{d_2}\right\|^{-1} \left\|\frac{fn_1-n_2}{d_2}\right\|^{-1}.\\
\end{multline*}

The first sum is bounded by
\begin{equation*}
\sum_{\substack{1 \le n_1 < d_1 \\ 1 \le n_2 < d_2 \\ fn_1 = n_2}} h\left\|\frac{n_1}{d_1}\right\|^{-1}\left\|\frac{n_2}{d_2}\right\|^{-1}  = h\sum_{1 \le n_1 < d_1} \left\|\frac{n_1}{d_1}\right\|^{-2} \le 2hd_1^2 \sum_{1 \le n_1 \le d_1/2} \frac 1{n_1^2} \ll hd_1^2.
\end{equation*}

It remains to bound the second sum. As in the proofs of Lemmas \ref{lem:prep-for-small-case-all-denoms-big} and \ref{lem:prep-for-small-case-first-denom-small}, we will split into cases based on whether $\tfrac{n_1}{d_1}$ and $\tfrac{n_2}{d_2}$ are in $(0,1/2]$ or $(1/2,1)$. 

Assume first that both $n_1/d_1, n_2/d_2 \in (0,1/2]$, or that both $n_1/d_1$ and $n_2/d_2$ are in $(1/2,1)$. In the latter case, we can substitute $m_1 = d_1 -n_1$ and $m_2 = d_2-n_2$ to revert precisely to the former case, so it suffices to assume that both $n_1/d_1$ and $n_2/d_2$ are in $(0,1/2]$. Then
\begin{align*}
\sum_{\substack{1 \le n_1 \le d_1/2 \\ 1 \le n_2 \le d_2/2 \\ fn_1 \ne n_2}} \frac{d_1d_2}{n_1n_2} \left\|\frac{fn_1-n_2}{d_2}\right\|^{-1} &= \sum_{\substack{1 \le n_1 \le d_1/2 \\ 1 \le n_2 \le d_2/2 \\ fn_1 > n_2}} \frac{d_1d_2^2}{n_1n_2(fn_1-n_2)} + \sum_{\substack{1 \le n_1 \le d_1/2 \\ 1 \le n_2 \le d_2/2 \\ fn_1 < n_2}} \frac{d_1d_2^2}{n_1n_2(n_2-fn_1)}.
\end{align*}
By applying Lemma \ref{lem:n1n2n3boundswith-one-alpha} with $\a_1 = f$ and $\nu_1 = \nu_2 = 0$, the first sum is bounded by $\ll d_2^3 \frac{\log d_2}{f^2} = d_1^2d_2 \log d_2$. For the second sum, we can achieve a bound that is somewhat stronger than the bound furnished by Lemma \ref{lem:n1n2n3bounds-with-one-alpha-n2-subtracted} in this special case. Specifically we have, writing $n_3 = n_2 - fn_1$,
\begin{align*}
d_2^3 \sum_{\substack{1 \le n_1 \le d_1/2 \\ 1 \le n_2 \le d_2/2 \\ 1 \le n_3 \le d_2/2 \\ fn_1 - n_2 + n_3 = 0}} \frac{1}{fn_1n_2n_3} &= d_2^3 \sum_{1 \le n_1 \le d_1/2} \frac 1{fn_1} \sum_{\substack{fn_1 \le n_2 \le d_2/2 \\ 1 \le n_3 \le d_2/2 \\ fn_1 + n_3 = n_2}} \frac 1{n_2-n_3} \left(\frac 1{n_3} - \frac 1{n_2}\right) \\
&= d_2^3 \sum_{1 \le n_1 \le d_1/2} \frac 1{(fn_1)^2} \sum_{1 \le n_3 \le d_2/2-fn_1} \left(\frac 1{n_3} - \frac 1{n_3 + fn_1}\right) \\
&\ll d_2^3 \sum_{1 \le n_1 \le d_1/2} \frac 1{(fn_1)^2} \log d_2 \\
&\ll d_2^3\frac{\log d_2}{f^2} = d_1^2d_2 \log d_2.
\end{align*} 
Thus in this case, the second sum is $\ll d_1^2d_2\log d_2$.

Now assume that $n_1/d_1 \in (1/2,1)$ but $n_2/d_2 \in (0,1/2]$; by swapping both $n_i$'s with $m_i = d_i-n_i$, this is the same as the case that $n_1/d_1 \in (0,1/2]$ but $n_2/d_2 \in (1/2,1)$, so it is our only remaining case.

On substituting $m_1 = d_1-n_1$, the sum in this case becomes
\begin{align*}
\sum_{\substack{1 \le m_1 \le d_1/2 \\ 1 \le n_2 \le d_2/2 \\ fm_1 + n_2 < d_2}} \frac{d_1d_2}{m_1n_2} \left\|\frac{fm_1+n_2}{d_2}\right\|^{-1} &= \sum_{\substack{1 \le m_1 \le d_1/2 \\ 1 \le n_2 \le d_2/2 \\ fm_1 + n_2 \le d_2/2}} \frac{d_1d_2^2}{m_1n_2(fm_1 + n_2)} + \sum_{\substack{1 \le m_1 \le d_1/2 \\ 1 \le n_2 \le d_2/2 \\ d_2/2 < fm_1 + n_2 < d_2}} \frac{d_1d_2^2}{m_1n_2(d_2-n_2-fm_1)}.
\end{align*}
The first sum is
\begin{equation*}
\le d_1d_2^2 \sum_{\substack{1 \le m_1 \le d_1/2 \\ 1 \le n_2 \le d_2/2 \\ fm_1 + n_2 < d_2}} \frac{1}{fm_1^2n_2} \ll \frac{d_1d_2^2}{f} \log d_2 = d_1^2d_2\log d_2. 
\end{equation*}

As for the second sum, setting $n_3 = d_2-n_2-fm_1$, we can bound it by applying Lemma \ref{lem:n1n2n3bounds-with-one-alpha-all-added} where $\a_1 = f$ and $\nu_1 = \nu_2 = d_2$ to get that
\begin{align*}
d_2^3 \sum_{\substack{1 \le m_1 \le d_1/2 \\ 1 \le n_2 \le d_2/2 \\ 1 \le n_3 \le d_2/2 \\ fm_1 + n_2 + n_3 = d_2}} \frac 1{fm_1n_2n_3} &\ll d_2^3 \frac 1{d_2}\log d_2 \left( 1 + \frac{\log d_2}{f}\right) \\
&\ll d_2^2 \log d_2 + d_1d_2 \log d_2,
\end{align*}
both of which are $\ll d_1^2d_2 \log d_2.$ This completes the proof.
\end{proof}

\subsection{Bounding $T_1$: terms with $gx \ge h$}

Define
\begin{equation}\label{eq:T1-definition}
T_1 = \sum_{\substack{g,x,y,z|q \\ gx \ge h}} \frac{\mu(gxyz)^2}{\phi(g)^3\phi(xyz)^2} \sum_{\substack{a_1,a_2,a_3 \\ (a_1,gyz) = \cdots = 1 \\ a_1/gyz + \cdots \in \Z}} F\left(\frac{a_1}{gyz}\right)F\left(\frac{a_2}{gxz}\right)F\left(\frac{a_3}{gxy}\right).
\end{equation}

For these terms, the rough argument that ``the probability that each of $\tfrac{a_2}{q_2}$ and $\tfrac{a_3}{q_3}$ are sufficiently small is about $\tfrac 1h$, making the size of the sum $h^{1+\ep}$ instead of $h^{3+\ep}$'' can be made precise, although some of the counting arguments are rather involved, and rely on the lemmas of the previous section. Nevertheless, we will use this basic idea to prove the following bound.
\begin{lemma}\label{lem:T1}
Let $h \ge 4$, let $q$ be the product of primes $p \le h^4$, and define $T_1$ by \eqref{eq:T1-definition}. Then
\[T_1 \ll h(\log h)^5.\]
\end{lemma}
\begin{proof}
Recall that $q_1 = gyz$, $q_2 = gxz$, and $q_3 = gxy$. Since $gx \ge h$, $gxy$ and $gxz$ (i.e., $q_2$ and $q_3$) must also both be $\ge h$. Recall the notation that $\tilde q_i = \min\{q_i,h\}$, so that $\tilde q_2= \tilde q_3=h$. 

Since $\frac{a_1}{q_1} + \frac{a_2}{q_2} + \frac{a_3}{q_3} \in \Z$, the sum $\frac{n(a_1,q_1)}{\tilde q_1} + \frac{n(a_2,q_2)}{\tilde q_2} + \frac{n(a_3,q_3)}{\tilde q_3}$ satisfies
\begin{align*}
&\left\|\frac{n(a_1,q_1)}{\tilde q_1} + \frac{n(a_2,q_2)}{\tilde q_2} + \frac{n(a_3,q_3)}{\tilde q_3}\right\| \le \left\|\frac{a_1}{q_1} + \frac{a_2}{q_2}+\frac{a_3}{q_3}\right\| + \sum_{i=1}^3 \left\|\frac{n(a_i,q_i)}{\tilde q_i} - \frac{a_i}{q_i}\right\| \le \frac 3h,
\end{align*}
since $\left|\frac aq - \frac{n(a,q)}{\tilde q} \right| < \frac 1h$ always.
We can then bound the sum by replacing the fractions $\frac{a_i}{q_i}$ by their $h$-approximations $\tfrac{n(a_i,q_i)}{q_i}$. Precisely, we have
\begin{align*}
T_1 &= \sum_{\substack{g,x,y,z|q \\ gx \ge h}} \frac{\mu(gxyz)^2}{\phi(g)^3\phi(xyz)^2} \sum_{\substack{a_1, a_2, a_3 \\ (a_i,q_i) = 1 \\ \sum_i a_i/q_i \in \mathbb Z}} F\left(\frac{a_1}{q_1}\right)F\left(\frac{a_2}{q_2}\right)F\left(\frac{a_3}{q_3}\right) \\
&\ll \sum_{\substack{g,x,y,z|q \\ gx \ge h}} \frac{\mu(gxyz)^2}{\phi(g)^3\phi(xyz)^2} \sum_{\substack{a_1, a_2, a_3 \\ (a_i,q_i) = 1 \\ \sum_i a_i/q_i \in \mathbb Z}} \left\|\frac{n(a_1,q_1)}{\tilde q_1}\right\|^{-1} \left\|\frac{n(a_2,q_2)}{\tilde q_2}\right\|^{-1} \left\|\frac{n(a_3,q_3)}{\tilde q_3}\right\|^{-1} \\
&\ll \sum_{\substack{g,x,y,z|q \\ gx \ge h}} \frac{\mu(gxyz)^2}{\phi(g)^3\phi(xyz)^2} \sum_{\substack{1 \le n_1, n_2, n_3 \le \tilde q_i - 1 \\ \left\|\sum_i n_i/\tilde q_i\right\| \le 3/h}} \left\|\frac{n_1}{\tilde q_1}\right\|^{-1} \left\|\frac{n_2}{\tilde q_2}\right\|^{-1} \left\|\frac{n_3}{\tilde q_3}\right\|^{-1} \sum_{\substack{a_1, a_2, a_3 \\ (a_i,q_i) = 1 \\ \sum_i a_i/q_i \in \mathbb Z \\ n(a_i,q_i) = n_i}} 1.\\
\end{align*}

The inside sum is the number of triplets $a_1,a_2,a_3$ with $n(a_i,q_i) = n_i$ for all $i$, $(a_i,q_i) = 1$, and $\sum_i \tfrac{a_i}{q_i} \in \Z$. The constraint that $n(a_i,q_i) = n_i$ implies that each $a_i$ lies in an interval of length $\ll \frac{q_i}h + 1$; that is, for $q_i \ge h$, $\tfrac{q_i}{h}n_i \le a_i \le \tfrac{q_i}{h}(n_i + 1)$. 

The constraint that $\sum_i \frac{a_i}{q_i} \in \Z$, after multiplying out denominators, is equivalent to the constraint that
\begin{equation}\label{eq:three-term-modular-equation-gxyz}
a_1x + a_2y + a_3z \equiv 0 \mod{gxyz}.
\end{equation}
Once the $q_i$'s (or equivalently $g,x,y,$ and $z$) are fixed, there are $\ll \frac{q_1}h + 1$ choices of $a_1$ such that $n(a_1,q_1) = n_1$. Once $a_1$ is fixed, $a_2$ is determined mod $z$ by \eqref{eq:three-term-modular-equation-gxyz}. Since $1 \le a_2 \le gxz$, fixing $a_2$ is equivalent to choosing a congruence class mod $gx$ for $a_2$; there are $\ll \frac{gx}h + 1$ choices of this congruence class such that $a_2$ lies within the interval where $n(a_2,q_2) = n_2$. Since $gx \ge h$ by assumption, $\frac{gx}h + 1 \ll \frac{gx}{h}$. Once $a_1$ and $a_2$ have been fixed, $a_3$ is entirely determined by \eqref{eq:three-term-modular-equation-gxyz}. Thus the total number of triplets $a_1,a_2,a_3$ satisfying all constraints is $\ll \left(\frac{q_1}h + 1\right)\frac{gx}{h}$.

Thus $T_1$ is bounded by
\begin{align*}
T_1 &\ll \sum_{\substack{g,x,y,z|q \\ gx \ge h}} \frac{\mu(gxyz)^2}{\phi(g)^3\phi(xyz)^2} \left(\frac{q_1}{h} + 1\right)\frac{gx}{h}\sum_{\substack{1 \le n_i \le \tilde q_i-1 \\ \left\|\sum_i n_i/\tilde q_i\right\| \le 3/h}} \left\|\frac{n_1}{\tilde q_1}\right\|^{-1}\left\|\frac{n_2}h\right\|^{-1}\left\|\frac{n_3}h\right\|^{-1}.
\end{align*}

Consider first those terms where $\tilde q_1 = h$. Thus $\tfrac{q_1}{h} \gg 1$, and by Lemma \ref{lem:prep-for-small-case-all-denoms-big}, the inside sum is $\ll h^3$. This implies that the terms with $\tilde q_1 = h$ are bounded by
\begin{align*}
&\ll \sum_{\substack{g,x,y,z|q \\ gx \ge h}}\frac{\mu(gxyz)^2}{\phi(g)^3\phi(xyz)^2} \frac{q_1}{h}\frac{gx}{h}h^3 \\
&\ll h \sum_{\substack{g,x,y,z,|q \\ gx \ge h}} \frac{\mu(gxyz)^2}{\phi(g)^3\phi(xyz)^2}g^2xyz\text{, since $q_1 = gyz$}.
\end{align*}
Recalling that $q$ is the product of all primes $p \le h^4$, this sum is
\begin{align*}
&\ll h \prod_{p \le h^4} \left( 1 + \frac{p^2}{(p-1)^3} + \frac{3p}{(p-1)^2}\right) \ll h(\log h)^4.
\end{align*}

The remaining terms are those where $\tilde q_1 = q_1 < h$. By applying Lemma \ref{lem:prep-for-small-case-first-denom-small} to the inside sum, the terms with $\tilde q_1 = q_1 < h$ are bounded by
\begin{align*}
&\ll \sum_{\substack{g,x,y,z|q \\ gx \ge h}}\frac{\mu(gxyz)^2}{\phi(g)^3\phi(xyz)^2} \frac{gx}{h} \left(h^2q_1 \log h\right) \\
&\ll h \log h \sum_{\substack{g,x,y,z|q \\ gx \ge h}}\frac{\mu(gxyz)^2g^2xyz}{\phi(g)^3\phi(xyz)^2}\text{, since $q_1 = gyz$}, \\
&\ll h (\log h)\prod_{p \le h^4} \left( 1 + \frac{p^2}{(p-1)^3} + \frac{3p}{(p-1)^2}\right) \ll h(\log h)^5.
\end{align*}
Thus $T_1 \ll h(\log h)^4 + h(\log h)^5 \ll h(\log h)^5$, as desired.
\end{proof}

\subsection{Bounding $T_2$: terms with $gx,gy, gz$ small and $a_2,a_3$ large}

We now consider $T_2$, which is the sum of terms in \eqref{eq:V3-bound-gxyz-and-Fs} where $gx$, $gy,$ and $gz$ are all $< h$ and $\left\|\frac{a_2}{gxz}\right\| \ge \frac 1h,$  and $\left\|\frac{a_3}{gxy}\right\| \ge \frac 1h$. That is, define
\begin{equation}\label{eq:T2def}
T_2 := \sum_{\substack{g,x,y,z|q \\ x,y,z < h/g}} \frac{\mu(gxyz)^2}{\phi(g)^3\phi(xyz)^2} \sum_{\substack{a_1,a_2,a_3 \\ (a_1,gyz) = \cdots = 1 \\ a_1/gyz + \cdots \in \Z \\ \|a_2/gxz\| \ge 1/h \\ \|a_3/gxy\| \ge 1/h}} F\left(\frac{a_1}{gyz}\right)F\left(\frac{a_2}{gxz}\right)F\left(\frac{a_3}{gxy}\right).
\end{equation}
The strategy for bounding $T_2$ is very different from that used to bound $T_1$. Intuitively, since the fractions $\tfrac{a_2}{gxz}$ and $\tfrac{a_3}{gxy}$ are far from an integer, we are now considering terms where the values of $F\left(\tfrac{a_2}{gxz}\right)$ and $F\left(\frac{a_3}{gxy}\right)$ are relatively small, except perhaps at the boundary where $\tfrac{a_2}{gxz}$ and $\tfrac{a_3}{gxy}$ are very close to $\tfrac 1h$. Since the denominators are loosely constrainted to be small, there cannot be too many points on this boundary. We will prove a precise bound in the following lemma.

\begin{lemma}\label{lem:T2}
Let $h \ge 4$, let $q$ be the product of primes $p \le h^4$, and let $T_2$ be defined as in \eqref{eq:T2def}. Then
\[T_2 \ll h(\log h)^4 (\log \log h)^2.\]
\end{lemma}

\begin{proof}
We begin by reparametrizing the sum in \eqref{eq:T2def} over $a_1, a_2, a_3$. For fixed $g,x,y,z$ and fixed $a_1, a_2, a_3$ satisfying the constraints of the sums in \eqref{eq:T2def}, we will fix parameters $a,b,c$ as follows. By the Chinese Remainder theorem, and since $g, x,$ and $y$ are pairwise relatively prime, there exist unique values $1 \le a \le x$ and $1 \le b \le gy$ such that $\frac{a_3}{gxy} \equiv \frac ax - \frac b{gy} \mod 1$. Similarly, there exist unique values $1 \le a'\le x$ and $1 \le c \le gz$ such that $\frac{a_2}{gxz} \equiv \frac c{gz} - \frac{a'}{x} \mod 1$. Since $\frac{a_1}{gyz} + \frac{a_2}{gxz} + \frac{a_3}{gxy} \in \mathbb Z$, we have
\begin{equation*}
gyz\left(\frac{a_2}{gxz} + \frac{a_3}{gxy}\right) \in \mathbb Z 
\quad \thus \quad gyz \left(\frac ax - \frac b{gy} + \frac c{gz} - \frac{a'}{x}\right) \in \mathbb Z 
\quad \thus \quad gyz \frac{(a-a')}{x}\in \mathbb Z.
\end{equation*}
Since $(gyz,x) = 1$, this implies that $x|(a-a')$, and thus $a = a'$. Finally, the fact that $\frac{a_1}{gyz} + \frac{a_2}{gxz} + \frac{a_3}{gxy}\in \mathbb Z$ implies that $\frac{a_1}{gyz} \equiv -\frac{a_2}{gxz}-\frac{a_3}{gxy} \equiv \frac{b}{gy}-\frac{c}{gz} \mod 1$, so that the triple $a_1,a_2,a_3$ uniquely determines (and is uniquely determined by) a triple $a,b,c$ with $1 \le a \le x$, $1 \le b \le gy$, and $1 \le c \le gz$ such that
\begin{equation*}
\frac{a_1}{gyz} \equiv \frac{b}{gy}-\frac{c}{gz} \mod 1, \frac{a_2}{gxz} \equiv \frac{c}{gz}-\frac{a}x \mod 1, \text{ and }\frac{a_3}{gxy} \equiv \frac ax - \frac b{gy} \mod 1.
\end{equation*}

Upon moving the sums over $y$ and $z$ in \eqref{eq:T2def} inside, we get
\begin{equation*}
T_2 = \sum_{\substack{g,x|q \\ x < h/g}} \frac{\mu(gx)^2}{\phi(g)^3\phi(x)^2} \sum_{\substack{a\\ (a,x) = 1}} S_2(g,x,a),
\end{equation*}
where $S_2(g,x,a)$ denotes the sum
\begin{equation}\label{eq:pf-of-t2-def-of-s2}
S_2(g,x,a) = \sum_{\substack{y,z|q \\ y,z < h/g}} \frac{\mu(gxyz)^2}{\phi(yz)^2} \sum_{\substack{b,c \\ (b,gy) = (c,gz) = 1 \\ \\ \left\|\frac{c}{gz}-\frac ax\right\| \ge \frac 1h \\ \left\|\frac a{x} - \frac{b}{gy}\right\| \ge \frac 1h }} F\left(\frac ax - \frac b{gy}\right) F\left(\frac b{gy} - \frac c{gz}\right) F\left(\frac c{gz} - \frac ax\right).
\end{equation}

Since $gy < h$ and $gz < h$, the product $yz$ is less than $h^2$, so that
\begin{equation*}
\frac{yz}{\phi(yz)} \ll \log \log(h^2) \ll \log \log h.
\end{equation*}
Thus we can replace the expression $\tfrac 1{\phi(yz)^2}$ in \eqref{eq:pf-of-t2-def-of-s2} with $\tfrac{(\log \log h)^2}{y^2z^2}$.

Let $\ell$ and $m$ be such that $2^\ell < y \le 2^{\ell + 1}$ and $2^m < z \le 2^{m+1}$, and further define $n_{\ell}$ and $n_m$ to be variables ranging from $1$ to $g2^{\ell}$ and $1$ to $g2^{m}$ respectively. 

If $\frac{n_\ell}{g2^{\ell + 1}} \le \left\|\frac ax - \frac b{gy}\right\| \le \frac{n_\ell + 1}{g2^{\ell +1}}$, then $F\left(\frac ax - \frac b{gy}\right) \ll F\left(\frac{n_\ell}{g2^{\ell+1}}\right)$; crucially, this upper bound depends only on $\ell$ and $n_\ell$, and does not depend on $b$ or $y$. Similarly, if $\frac{n_m}{g2^{m+1}} \le \left\| \frac c{gz}-\frac ax\right\| \le \frac{n_m + 1}{g2^{m+1}}$, then $F\left(\frac c{gz} - \frac ax \right) \ll F \left(\frac{n_m}{g2^{m+1}}\right)$. Because of the assumption that $\left\|\frac ax - \frac b{gy}\right\| \ge \frac 1h$, the constraint $\frac{n_\ell}{g2^{\ell + 1}} \le \left\|\frac ax - \frac b{gy}\right\| \le \frac{n_\ell + 1}{g2^{\ell + 1}}$ is satisfied for some $n_\ell$ with $1 \le n_\ell \le g2^{\ell}$; in particular, the case that $n_\ell = 0$ is ruled out. Similarly, the case that $n_m = 0$ is ruled out by our assumptions on $\frac c{gz} -\frac ax$.

Thus
\begin{align*}
S_2(g,x,a) \ll &(\log \log h)^2 \sum_{\ell,m = 1}^{\log_2\tfrac hg} \sum_{n_\ell = 1}^{g2^{\ell + 1}-1} \sum_{n_m = 1}^{g2^{m + 1}-1} \frac{1}{2^{2\ell + 2m}} \\
&\times F\left(\frac{n_\ell}{g2^{\ell + 1}}\right)F\left(\frac{n_m}{g2^{m+1}}\right) F\left(\frac{n_\ell 2^m - n_m2^{\ell}}{g2^{\ell +m+ 1}}\right)\!\!\!\!\!\!\!\!\!\!\!\!\!\!\!\!\!\!\!\sum_{\substack{2^\ell < y \le 2^{\ell + 1} \\ 2^m < z \le 2^{m+1} \\ n_\ell \le g2^{\ell + 1}\|a/x - b/(gy)\| \le n_\ell + 1 \\ n_m \le g2^{m+1}\|c/(gz) - a/x\| \le n_m + 1}}\!\!\!\!\!\!\!\!\!\!\!\!\!\!\!\!\!\!\!\!\! 1.
\end{align*}
Define
\begin{align*}
C_{\ell,n_\ell} &= \#\left\{b,y: \tfrac by \in \left(\tfrac{ga}{x} - \tfrac{n_\ell+1}{y}, \tfrac{ga}{x} - \tfrac{n_\ell}{y}\right) \cup \left(\tfrac{ga}{x} + \tfrac{n_\ell}y, \tfrac{ga}{x} + \tfrac{n_\ell + 1}{y}\right), 1 \le b < g2^{\ell + 1}, 2^\ell < y \le 2^{\ell + 1}\right\}, \\
\end{align*}
and define $C_{m,n_m}$ in the same way, so that the inside sum of $S_2(g,x,a)$ is $C_{\ell,n_\ell}C_{m,n_m}$. The minimum spacing of two distinct points $\frac{b_1}{y_1}$ and $\frac{b_2}{y_2}$ with denominators $y_i \le 2^{\ell + 1}$ is $O(2^{-2\ell})$, so
\[C_{\ell,n_\ell} \ll \frac{2^{2\ell}}{2^\ell} \ll 2^{\ell},\]
and similarly $C_{m,n_m} \ll 2^m$. This implies that
\begin{align*}
S_2(g,x,a) \ll &(\log \log h)^2 \sum_{\ell,m = 1}^{\log_2\tfrac hg} \frac{2^{\ell+m}}{2^{2\ell + 2m}} \sum_{n_\ell = 1}^{g2^{\ell + 1}} \sum_{n_m = 1}^{g2^{m + 1}} F\left(\frac{n_\ell}{g2^{\ell + 1}}\right)F\left(\frac{n_m}{g2^{m+1}}\right) F\left(\frac{n_\ell 2^m - n_m 2^{\ell}}{g2^{\ell +m+ 1}}\right).
\end{align*}

By the symmetry of $\ell$ and $m$, we can restrict the sum to the terms where $\ell \le m$. Applying Lemma \ref{lem:triple_sum_F} to the sums over $n_\ell,n_m$ with $d_1 = g2^\ell$ and $d_2 = g2^m$ gives
\begin{align*}
S_2(g,x,a) \ll &(\log \log h)^2 \sum_{\substack{\ell,m = 1 \\ \ell \le m}}^{\log_2\tfrac hg} \frac 1{2^{\ell + m}} \left(hg^22^{2\ell} + g^32^{2\ell+m}m\right) \\
\ll &h(\log \log h)^2 g^2 \sum_{\substack{\ell,m = 1 \\ \ell \le m}}^{\log_2\tfrac hg}  \frac 1{2^{m-\ell}} + (\log \log h)^2 g^3 \sum_{\substack{\ell,m = 1 \\ \ell \le m}}^{\log_2\tfrac hg} m2^\ell \\
\ll &h(\log \log h)^2 g^2 \left(\log \frac hg\right)^2,
\end{align*}
and thus
\begin{align*}
T_2 &\ll  h(\log \log h)^2 \sum_{\substack{g,x|q \\ x < h/g}} \frac{\mu(gx)^2}{\phi(g)^3\phi(x)^2} \sum_{\substack{a\\ (a,x) = 1}} g^2\left(\log \frac hg\right)^2\\
&\ll h(\log h)^2(\log \log h)^2 \sum_{\substack{g,x|q \\ x < h/g}} \frac{\mu(gx)^2g^2}{\phi(g)^3\phi(x)} \\
&\ll h(\log h)^2(\log \log h)^2 \prod_{p \le h^4} \left( 1 + \frac{p^2}{(p-1)^3} + \frac{1}{p-1}\right)\text{, since $q = \prod_{p \le h^4}p$} \\
&\ll h(\log h)^4 (\log \log h)^2.
\end{align*}
\end{proof}

\subsection{Bounding $T_3$: terms with $gx,gy, gz$ small and each $a_i$ small}

All that remains is to analyze the sum $T_3$, which consists of the terms in \eqref{eq:V3-bound-gxyz-and-Fs} where $gx,gy,$ and $gz < h$, and for each $i$, $\left\|\frac{a_i}{q_i}\right\| \le \frac 2h$. Precisely, we define
\begin{equation}\label{eq:T3def}
T_3 := \sum_{\substack{g,x,y,z|q \\ x,y,z < h/g}} \frac{\mu(gxyz)^2}{\phi(g)^3\phi(xyz)^2} \sum_{\substack{a_1,a_2,a_3 \\ (a_1,gyz) = \cdots = 1 \\ a_1/gyz + \cdots \in \Z \\ \|a_1/gyz\| < 2/h \\ \|a_2/gxz\| < 2/h \\ \|a_3/gxy\| < 2/h}} F\left(\frac{a_1}{gyz}\right)F\left(\frac{a_2}{gxz}\right)F\left(\frac{a_3}{gxy}\right).
\end{equation}
Intuitively, there are simply not many triples of fractions $\tfrac{a_i}{q_i}$ where the denominators are not too big, each fraction is close to an integer, and the sum of all three is in $\mathbb Z$. We will make this precise in the following lemma bounding $T_3$, where the key savings come from bounding the number of satisfactory triples.

\begin{lemma}\label{lem:T3}
Let $h \ge 4$, let $q$ be the product of all primes $p \le h^4$ and define $T_3$ by \eqref{eq:T3def}. Then 
\[T_3 \ll h (\log h)^4 (\log \log h)^2.\]
\end{lemma}
\begin{proof}
Since $\left\|\frac{a_3}{gxy} \right\| < \frac 2h$, we must have $\tfrac 1{gxy} < \tfrac 2h,$
so if $y < \sqrt{\tfrac{h}{2g}}$, then $x > \sqrt{\tfrac{h}{2g}}$. By the same logic with $a_1$ and $a_2$, at most one of $x,y,z$ can be $< \sqrt{\tfrac h{2g}}$. By relabeling if necessary, we get that
\begin{equation*}
T_3 \ll \sum_{\substack{g,x,y,z|q \\ x,y,z < h/g \\ y,z \ge \sqrt{h/(2g)}}} \frac{\mu(gxyz)^2}{\phi(g)^3\phi(xyz)^2} \sum_{\substack{a_1,a_2,a_3 \\ (a_1,gyz) = \cdots = 1 \\ a_1/gyz + \cdots \in \Z \\ \|a_1/gyz\| < 2/h \\ \|a_2/gxz\| < 2/h \\ \|a_3/gxy\| < 2/h}} F\left(\frac{a_1}{gyz}\right)F\left(\frac{a_2}{gxz}\right)F\left(\frac{a_3}{gxy}\right).
\end{equation*}

As in the proof of Lemma \ref{lem:T2}, there are unique values $a,b,c$ with
\[\frac{a_1}{gyz} \equiv \frac{b}{gy} - \frac{c}{gz} \mod 1, \frac{a_2}{gxz} \equiv \frac{c}{gz} - \frac ax \mod 1, \text{ and } \frac{a_3}{gxy} \equiv \frac ax - \frac b{gy} \mod 1,\]
and we can reparametrize $T_3$ in terms of sums over $a,b,c$ instead of $a_1,a_2,a_3$. Doing so, and moving the sums over $b,y,c,$ and $z$ inside, we get that
\begin{equation*}
T_3 \ll h^3 \sum_{\substack{g,x|q \\ gx \le h}} \frac{\mu(gx)^2}{\phi(g)^3\phi(x)^2} \sum_{\substack{a \le x \\ (a,x) = 1}} S_3(g,x,a),
\end{equation*}
where
\begin{equation*}
S_3(g,x,a) := \sum_{\substack{\sqrt{h/(2g)} \le y \le h/(2g) \\ \sqrt{h/(2g)} \le z \le h/(2g)}} \frac {\mu(yz)^2}{\phi(yz)^2}\#\left\{b,c:\frac b{gy},\frac{c}{gz} \in \left(\frac ax - \frac 2h, \frac ax + \frac 2h\right)\right\}.
\end{equation*}

Since $y,z \le h$, the product $yz$ is $\le h^2$, and thus $\frac 1{\phi(yz)^2} \ll \frac{(\log \log h)^2}{y^2z^2}$, when this term appears in $S_3(g,x,a)$. In order to bound $S_3(g,x,a)$, we split the sums over $y$ and $z$ dyadically, defining $\ell$ such that $2^\ell < y \le 2^{\ell + 1}$ and $2^m < z \le 2^{m+1}$.

Then
\begin{align*}
S_3(g,x,a) &\ll (\log \log h)^2 \sum_{\ell,m = \tfrac 12(\log_2 (h/g))}^{\log_2(h/g)} \frac{C_\ell C_m}{2^{2\ell}2^{2m}},
\end{align*}
where
\[C_\ell := \#\left\{ b,y : \frac by \in \left(\frac{ga}{x} - \frac{2g}h, \frac{ga}{x} + \frac{2g}h\right), 1 \le b < y, y \le 2^{\ell + 1}\right\},\]
and $C_m$ is defined identically, with $m$ in place of $\ell$. The minimum spacing of two distinct points $\frac{b_1}{y_1}$ and $\frac{b_2}{y_2}$ with denominators at most $2^{\ell + 1}$ is $O\left(\tfrac 1{2^{2\ell}}\right)$, so $C_\ell \ll 2^{2\ell}\frac gh  + 1.$
Since $\ell \ge \tfrac 12(\log_2(h/g))$, $2^{2\ell}\frac gh \ge 1$, so in particular
$C_\ell \ll 2^{2\ell} \frac gh,$
and similarly $C_m \ll 2^{2m} \frac gh$.

Plugging this in gives
\begin{align*}
S_3(g,x,a) &\ll (\log \log h)^2 \sum_{\ell,m = \tfrac 12(\log_2 (h/g))}^{\log_2(h/g)} \frac{2^{2\ell}2^{2m}}{2^{2\ell}2^{2m}}\frac{g^2}{h^2} \ll \frac{g^2}{h^2} (\log (h/g))^2 (\log \log h)^2,
\end{align*}
so that
\begin{align*}
T_3 &\ll h (\log \log h)^2 \sum_{\substack{g,x|q \\ gx \le h}} \frac{\mu(gx)^2 g^2}{\phi(g)^3\phi(x)^2} \sum_{\substack{a \le x \\ (a,x) = 1}} (\log (h/g))^2 \\
&\ll h(\log h)^2(\log \log h)^2 \sum_{\substack{g,x|q \\ gx \le h}} \frac{\mu(gx)^2 g^2}{\phi(g)^3\phi(x)} \\
&\ll h(\log h)^2(\log \log h)^2 \prod_{p \le h^4} \left(1 + \frac{p^2}{(p-1)^3} + \frac{1}{p-1}\right),
\end{align*}
recalling that $q = \prod_{p \le h^4} p$. Thus $T_3 \ll h(\log h)^4(\log \log h)^2$. 
\end{proof}

Putting Lemmas \ref{lem:T1}, \ref{lem:T2}, and \ref{lem:T3} together completes the proof of Theorem \ref{threetermintegerreducedresidues}.

\section{Function Field Analogues: Proof of Theorem \ref{thm:mvfuncfield}} \label{sec:funcfieldmv}

We now turn to considering analogous questions when working in $\F_q[t]$ rather than in $\Z$. To begin with, let's set up the situation in the function field case. Fix a finite field $\F_q$. Rather than primes in $\N$, consider monic irreducible polynomials in $\F_q[t]$. 

The \emph{norm} of a polynomial $F \in \F_q[t]$ is given by $|F| = q^{\deg F}$. We consider intervals in norm, where the interval $I(F,h)$ of degree $h$ is defined as
\[I(F,h) := \{G \in \F_q[t] : |F-G| < q^{h}\}.\]

For a fixed monic polynomial $Q$, we denote
\begin{align*}
\mathcal C(Q) &:= \left\{ \frac AQ \in \F_q[t]: |A| < |Q|\right\}, \\
\mathcal R(Q) &:= \left\{\frac AQ \in \F_q[t] : |A| < |Q|, (A,Q) = 1 \right\}.
\end{align*}
For $Q = 1$, we instead for convenience define $\mathcal C(Q) = \{1\} = \mathcal R(Q)$. If $\deg Q > 0$, the set of polynomials $F$ with $\deg F < \deg Q$ is a canonical set of representatives of $\F_q[t]/(Q)$; in what follows, we will identify $\{F \in \F_q[t] : \deg F < \deg Q\}$ with $\F_q[t]/(Q)$. If $Q = 1$, we will take $1$ to represent the unique equivalence class of $\F_q[t]/(Q)$. 

We consider the $k$th moment of the distribution of irreducible polynomials in intervals $I(F,h)$. As in the integer case, we begin by considering the related quantity of the distribution of reduced residues modulo a squarefree monic polynomial $Q$. That is, for $Q$ a fixed squarefree monic polynomial, we consider
\begin{equation}\label{eq:defn-of-mk-in-fcn-field}
m_k(Q;h) = \sum_{F \in \mathcal C(Q)} \Big(\Big( \sum_{\substack{G \in I(F,h) \\ (G,Q) = 1}}1 \Big) - \frac{q^h \phi(Q)}{|Q|}\Big)^k.
\end{equation}
Here we are taking the centered moment $m_k(Q;h)$ by subtracting $\frac{q^h \phi(Q)}{|Q|}$, which is the mean value of $\sum_{\substack{G \in I(F,h) \\ (G,Q) = 1}} 1$.

As in the integer case, we can express the moment $m_k(Q;h)$ in terms of exponential sums. For $\a = \frac{F}{G} \in \F_q(t)$ a rational function, let $\mathrm{res}(\a)$ denote the coefficient of $\frac 1t$ when $\a$ is written as a Laurent series with finitely many positive terms. Then define
\[e(\a) := e_q(\mathrm{res}(\a)) = \mathrm{exp}(2\pi i \cdot \mathrm{tr}(\mathrm{res}(\a))/p),\]
where $q$ is a power of the prime $p$ and $\mathrm{tr}:\F_q \to \F_p$ is the trace function. This exponential function, like its integer analog, satisfies the crucial property that for a monic polynomial $F \in \F_q[t],$
\[\sum_{\a \in \mathcal C(F)} e\left(\a\right) = \begin{cases}1 &\text{ if } F = 1 \\ 0 &\text{ otherwise}.\end{cases}\]

We then have the following lemma, analogous to \cite[Lemma 2]{MontgomeryVaughanReducedResidues}.

\begin{lemma}\label{lem:mvl2}
Let $Q \in \F_q[t]$ be squarefree and let $h \in \N_{\ge 1}$. Define $m_k(Q;h)$ by \eqref{eq:defn-of-mk-in-fcn-field}. Then
\[m_k(Q;h) = |Q| \left(\frac{\phi(Q)}{|Q|}\right)^k V_k(Q;h),\]
where 
\[V_k(Q;h) := \sum_{\substack{R_1, \dots R_k |Q \\ |R_i| > 1 \\ R_i \text{ monic}}} \prod_{i=1}^k \frac{\mu(R_i)}{\phi(R_i)} \sum_{\substack{\rho_1, \dots, \rho_k \\ \rho_i \in \mathcal R(R_i) \\ \sum_i \rho_i/R_i = 0}} E\left(\frac{\rho_1}{R_1}\right) \cdots E\left(\frac{\rho_k}{R_k}\right),\]
and where, for $\a \in \F_q(t)$ a rational function,
\[E(\a) := \sum_{M \in I(0,h)} e(M\a).\]
\end{lemma}
The proof follows that of \cite[Lemma 2]{MontgomeryVaughanReducedResidues} very closely.
\begin{proof}
Let $\kappa(R) = 1$ when $(R,Q) = 1$, $\kappa(R) = 0$ otherwise. Then
\begin{align*}
\kappa(R) = \sum_{S|(R,Q)} \mu(S) &= \sum_{S|Q} \frac{\mu(S)}{|S|} \sum_{\sigma \in \mathcal C(S)} e(R\sigma) \\ 
&= \sum_{T|Q} \Big(\sum_{\substack{A \in \mathcal C(T) \\ (A,T) = 1}}e(RA)\Big)\Big(\sum_{T|S|Q} \frac{\mu(S)}{|S|}\Big).
\end{align*}
Here the second factor is $\frac{\phi(Q)}{|Q|} \frac{\mu(T)}{|T|}.$ The function $\kappa(R)$ has mean value $\frac{\phi(Q)}{|Q|}$, so we subtract $\frac{\phi(Q)}{|Q|}$ from both sides, which removes the term when $T = 1$. We then substitute $R=M + N$, and sum over $M$ to see that
\[\sum_{\substack{|M| < q^h \\ (M+N,Q) = 1}} 1 - h\frac{\phi(Q)}{|Q|} = \frac{\phi(Q)}{|Q|} \sum_{\substack{R|Q \\ |R|>1}} \frac{\mu(R)}{\phi(R)} \sum_{\substack{A \in \mathcal C(R) \\ (A,R) = 1}} E\left(\frac{A}{R}\right) e(NA/R).\]
The argument is completed upon raising both sides to the $k$th power, summing over $N$, multiplying out the right hand side, and appealing to the fact that
\[\sum_{|N|<q^d} e(N(\a_1 + \cdots + \a_k)) = \begin{cases} q^d &\text{ if }\sum \a_i \in \Z \\ 0 &\text{ else}.\end{cases}\]
\end{proof}

One important difference between the integer setting and the function field setting is the behavior of the sums $E(\a)$, which are particularly well-behaved in $\F_q[t]$. These sums have also been studied by Hayes in \cite[Theorem 3.5]{Hayes1966}.
\begin{lemma}\label{lem:exponential-sums-in-Fq}
Let $\a \in \F_q(t)$ be a rational function with $\deg \a \le -1$. Then
\[E(\a) = \begin{cases}q^h &\text{ if } \deg \a < -h \\
0 &\text{ if } \deg \a \ge h.\end{cases}\]
\end{lemma}
\begin{proof}
Let $\mathcal P_h \ins \F_q[t]$ be the set of polynomials of degree less than $h$. Assume first that $\deg \a < -h$. Then for all $M \in \mathcal P_h$, $\deg M\a = \deg M + \deg \a \le h-1 - h-1 = -2$, so the Laurent series for $M\a$ has no $\frac 1t$ term, and thus $\mathrm{res}(M\a) = 0$. But then
\[E(\a) = \sum_{M \in \mathcal P_h} e(M\a) = \sum_{M \in \mathcal P_h} e_q(\mathrm{res}(M\a)) = \sum_{M \in \mathcal P_h} 1 = q^h.\]

Now assume that $\deg \a \ge -h$. Consider the map $\mathrm{res}_\a:\mathcal P_h \to \F_q$ which at a polynomial $M$ returns the residue of $M\a$. This map is linear over $\F_q$, so its image is either $0$ or all of $\F_q$. Let $M = t^{-\deg \a - 1}$. Since $-h \le \deg \a \le -1$, we have $0 \le -\deg \a - 1 \le h-1$, so $M$ indeed is a polynomial in $\mathcal P_h$. On the other hand, $\mathrm{res}(M\a)$ is precisely the leading coefficient of $\a$, which must be nonzero. Thus the image of $\mathrm{res}_\a$ is nonzero, so it is all of $\F_q$. In particular, $\mathrm{res}_\a(M)$ takes each value in $\F_q$ equally often. Thus
\[E(\a) = \sum_{M \in \mathcal P_h} e_q(\mathrm{res}(M\a))\]
is a balanced exponential sum, which has sum $0$.
\end{proof}

This fact and other properties of the sums $E(\a)$ mean that the analysis of Montgomery and Vaughan in \cite{MontgomeryVaughanReducedResidues} in the function field setting is more streamlined. In fact, their work automatically gives the analog of our desired bound for the third moment in the function field case.

\subsection{The analog of \cite{MontgomeryVaughanReducedResidues} in the function field setting}

We begin with the following fundamental lemma, with an identical proof to the integer case.

\begin{lemma}[Fundamental Lemma]
Let $R_1, \dots, R_k\in \F_q[t]$ be squarefree monic polynomials with $R = [R_1, \dots, R_k]$. Suppose for all irreducible $P|R$, $P$ divides at least two $R_i$'s. Let $G_i$ be positive real-valued function defined on $\mathcal C(R_i)$. Then
\[\left|\sum_{\substack{A_i \in \mathcal C(R_i) \\  \sum_i A_i/R_i = 0}} G_1\left(\frac{A_1}{R_1}\right) \cdots G_k\left(\frac{A_k}{R_k}\right)\right| \le \frac 1{|R|} \prod_{i=1}^k \left(|R_i| \sum_{\substack{A_i \in \mathcal C(R_i)}} \left|G_i\left(\frac{A_i}{R_i}\right)\right|^2 \right)^{1/2}.\]
\end{lemma}
The proof follows Montgomery-Vaughan very closely.
\begin{proof}
We proceed by induction on $k$.

Assume first that $k = 2$. Then we must have $R_1 = R_2 = R$. By Cauchy-Schwarz,
\[ \left|\sum_{\substack{|A| < |R|}} G_1\left(\frac{A}{R}\right)G_2\left(\frac{A}{R}\right)\right| \le \left(\sum_{\substack{|A| < |R|}} \left|G_1\left(\frac AR \right)\right|^2\right)^{1/2}\left(\sum_{\substack{|A| < |R|}} \left|G_2\left(\frac AR \right)\right|^2\right)^{1/2},\]
which after a bit of rearranging gives the desired result.

Now assume by induction that the result holds for $j \le k-1$. For arbitrary $k$, set $D = (R_1, R_2)$, and write $D = ST$ with $S|R_3 \cdots R_k$ and $(T, R_3 \cdots R_k) = 1$. Furthermore, write $R_1 = DR_1'$ and $R_2 = DR_2'$. Consider any term in the sum. Since $\sum_i \frac{A_i}{R_i} = 0,$ we have $T|\left(\frac{A_1}{R_1} + \frac{A_2}{R_2}\right)$. Thus $\frac{A_1}{STR_1'} + \frac{A_2}{STR_2'}$ can be expressed as a fraction $\frac{A}{R_1'R_2'S}$.

By the Chinese Remainder theorem, $\frac{A_1}{STR_1'} = \frac{\a_1}{R_1'} + \frac{\b_1}{ST}$ and $\frac{A_2}{STR_2'} = \frac{\a_2}{R_2'} + \frac{\b_2}{ST},$ where $\frac{\b_2}{ST} = -\frac{\b_1}{ST} +  \frac{\g}{S}$ because $T|\left(\frac{A_1}{R_1} + \frac{A_2}{R_2}\right)$. Thus $\frac{A_1}{R_1}$ and $\frac{A_2}{R_2}$ can be written as $\frac{A_1}{R_1} = \frac{A_1'}{R_1'} + \frac{\d}{D}$ and $\frac{A_2}{R_2}= \frac{A_2'}{R_2'} + \frac{\sigma}{S} - \frac{\d}{D},$ with each rational function of degree less than $0$.

Let $R^* = R_1'R_2'S$. For each $A^*$ with $|A^*| < |R^*|$, $\frac{A^*}{R^*}$ is uniquely of the form $\frac{A^*}{R^*}= \frac{A_1'}{R_1'} + \frac{A_2'}{R_2'} + \frac{\sigma}{S}.$ Define

\[G^*\left(\frac{A^*}{R^*}\right) = \sum_{\substack{\d \in \mathcal C(D)}} G_1\left(\frac{A_1'}{R_1'} + \frac{\d}{D}\right)G_2\left(\frac{A_2'}{R_2'}+\frac{\sigma}{S} - \frac{\d}{D}\right).\]
Then the sum in question is
\[\sum_{\substack{A^* \in \mathcal C(R^*) \\ A_i \in \mathcal C(R_i) \\ A^*/R^* + \sum_{i=3}^k A_i/R_i = 0}} G^*\left(\frac{A^*}{R^*}\right) G_3\left(\frac{A_3}{R_3}\right) \cdots G_k\left(\frac{A_k}{R_k}\right).\]
Via Cauchy-Schwarz as well as the induction hypothesis, the above is
\[\le \frac{|T|}{|R|} \left(|R^*| \sum_{\substack{A^* \in \mathcal C(R^*)}} G^*\left(\frac{A^*}{R^*}\right)^2 \right)^{1/2} \prod_{i=3}^k\left(|R_i| \sum_{\substack{A_i \in \mathcal C(R_i)}} G_i\left(\frac{A_i}{R_i}\right)^2 \right)^{1/2}.\]
It remains to bound the sum over $G^*$ in terms of $G_1$ and $G_2$. By Cauchy-Schwarz,
\[G^*\left(\frac{A^*}{R^*}\right)^2 \le \left(\sum_{\d \in \mathcal C(D)} G_1\left( \frac{A_1'}{R_1'} + \frac{\d}{D}\right)^2 \right) \left(\sum_{\d \in \mathcal C(D)} G_2 \left(\frac{A_2'}{R_2'} + \frac{\sigma}{S} - \frac{\d}{D}\right)^2\right),\]
so summing over $A^*$ gives
\[\sum_{A^* \in \mathcal C(R^*)} G^*\left(\frac{A^*}{R^*}\right)^2 \le |S| \left(\sum_{A_1 \in \mathcal C(R_1)} G_1\left(\frac{A_1}{R_1}\right)^2 \right)\left(\sum_{A_2 \in \mathcal C(R_2)} G_2 \left(\frac{A_2}{R_2}\right)^2 \right).\]
\end{proof}

We now present several preliminary lemmas about the sums $E(\a)$. The following lemma is analogous to \cite[Lemma 4]{MontgomeryVaughanReducedResidues}.
\begin{lemma}\label{lem:mvl4}
For any polynomial $R \in \F_q[t]$,
\[\sum_{\substack{S \in \mathcal C(R)}} E\left(\frac SR \right)^2 = \max\{q^{2h},|R|q^h\}.\]

Moreover, for any polynomial $R\in \F_q[t]$ and any rational function $\a\in \F_q(t)$, 
\[\sum_{\substack{S \in \mathcal C(R)}} E\left(\frac{S}{R} + \a\right)^2 \begin{cases} = \max\{q^{2h},|R|q^h\} &\text{ if } |\a| < q^{-h}\\ \le |R|q^{h-1} &\text{ if } |\a| \ge q^{-h}.\end{cases}\]
\end{lemma}
\begin{proof}
If $\deg R \le h$, then for all $S$ with $0 \ne |S| < |R|$, $h \ge \deg R - \deg S$, and thus $E\left(\frac SR \right)^2 = 0$. Meanwhile, $E(0)^2 = q^{2h}$, so in this case $\sum_{S \in \mathcal C(R)} E\left(\tfrac SR \right)^2 = q^{2h}$.

Now suppose $\deg R > h$. Then $E\left(\frac SR \right)$ is nonzero if and only if $\deg S < \deg R - h$. Thus
\[\sum_{\substack{S \in \mathcal C(R)}} E\left(\frac SR \right)^2 = \sum_{\substack{S \in \mathcal C(R) \\ |S| < |R|/q^h}} E\left( \frac SR \right)^2 = \sum_{\substack{S \in \mathcal C(R) \\ |S| < |R|/q^h}} q^{2h} = |R| q^h,\]
which completes the first portion.

Fix a rational function $\a$. For all $\frac{S}{R}$, $E\left(\frac{S}{R} + \a\right)$ is unchanged by replacing $\a$ with its fractional part; i.e, subtracting off the polynomial portion of $\a$ so that $|\a| < 1$, including the possibility that $\a = 0$. 

If a term $E\left(\frac SR + \a\right)$ is nonzero, then $\left|\frac SR + \a\right| < q^{-h}$. We'll split into two cases, when $|\a| < q^{-h}$ and when $|\a| \ge q^h$. First, if $|\a| < q^{-h}$, then $\left|\frac SR + \a\right| < q^{-h}$ if and only if $\left|\frac SR\right| < q^{-h}$. If $|R| \ge q^h$, there are $|R|/q^h$ values of $S$ satisfying this; if not, there is $1$ value. Thus if $|\a| < q^{-h}$, we have $\sum_{\substack{S\in \mathcal C(R)}} E\left(\frac{S}{R} + \a\right)^2 E\left(\frac SR + \a \right)=\mathrm{max}(q^{2h},|R|q^h)$.

Now assume $|\a| \ge q^{-h}$. If $\left|\frac SR + \a\right| < q^{-h}$, we must have $\left|\frac SR\right| = |\a| \ge q^{-h}$. Also, the first $\deg \a +h+1$ terms of $\frac SR$ are fixed, because they must cancel with the corresponding terms of $\a$ to yield a rational function of small enough degree. Correspondingly, the first $\deg \a + h + 1$ terms of $S$ are determined. Since $|S| = |R\a|$, there are at most $|R\a| \cdot \frac 1{|\a|\cdot|q^{h+1}|} = |R|q^{-h-1}$ nonzero choices of $S$. Thus in this case, $\sum_{\substack{S\in \mathcal C(R)}} E\left(\frac{S}{R} + \a\right)^2 \le |R|q^{h-1}$.
\end{proof}

The following lemma corresponds to Lemma 6 of Montgomery-Vaughan.
\begin{lemma}\label{lem:mvl6}
Let $R\in \F_q[t]$ be a polynomial, and let $\a,\b\in \F_q(t)$ be rational functions. Then
\[\sum_{S \in \mathcal C(R)} E\left(\frac SR + \a\right)E \left(\frac SR + \b\right) \ll E(\a-\b)q^{-h}\sum_{S\in \mathcal C(R)}E\left(\frac SR + \a\right)^2 \]
\end{lemma}
\begin{proof}
Again, we split into two cases. Assume first that $|\a-\b| \ge q^{-h},$ so $E(\a-\b) = 0$. Then either $\left|\frac SR + \b\right| \ge q^{-h}$, or $\left|\frac SR + \a\right| \ge q^{-h}.$ Thus for each $\frac SR$, either $E\left(\frac SR + \a\right) = 0$ or $E\left(\frac SR + \b\right) = 0$, so the product must be $0$, and thus the sum is $0$. 

Now assume that $|\a-\b| < q^{-h}$, so $E(\a-\b)=q^h$. By Lemma \ref{lem:exponential-sums-in-Fq}, if $|\a-\b| < q^{-h}$, then $E\left(\frac SR + \a\right) = E\left(\frac SR + \b\right)$ for all $S$. This gives the result.
\end{proof}

We are now ready to prove the following lemma, which is analogous to \cite[Lemma 7]{MontgomeryVaughanReducedResidues}. 

\begin{lemma}\label{lem:mvl7}
Let $k \ge 3$, and let $R_1, \dots, R_k\in \F_q[t]$ be squarefree polynomials with $|R_i| > 1$ for all $i$. Let $R = [R_1, \dots, R_k]$. Let $D=(R_1,R_2)$ and $D = ST$ with $S|R_3 \cdots R_k$ and $(T,R_3\cdots R_k) = 1$. Write $R_1 = DR_1'$, $R_2 = DR_2'$, and $R^* = R_1'R_2'S$. Define
\[S(R_1, \dots, R_k) := \sum_{\substack{A_i\in \mathcal R(R_i) \\ \sum_i A_i/R_i = 0}} \prod_{i=1}^k E\left(\frac{A_i}{R_i}\right).\]
If for some $i$, $|R_i| \le q^h$, then $S(R_1, \dots, R_k) = 0$. Otherwise,
\[S(R_1, \dots, R_k) \ll |R_1 \cdots R_k|\cdot|R|^{-1} (q^h)^{k/2}(X_1 + X_2 + X_3),\]
where
\begin{align*}
X_1 &= q^{-h/2}, \\
X_2 &= \begin{cases} |D|^{-1} &\text{ if } |R_1'| > q^h \\ 0 &\text{ otherwise, } \end{cases} \\
X_3 &= \begin{cases} |S|^{-1/2} &\text{ if } R_1 = R_2 \\ 0 &\text{ otherwise. } \end{cases}\\
\end{align*}
\end{lemma}
\begin{proof}
Assume first that for some $i$, $|R_i| \le q^h$. Then $E(A_i/R_i) = 0$ whenever $A_i \ne 0$, so in particular for all $A_i$ with $(A_i,R_i) = 1$, so the sum is $0$. Assume from now on that $|R_i| > q^h$ for all $i$.

We now return to the proof of the Fundamental Lemma. For $\frac{A^*}{R^*} = \frac{A_1'}{R_1'} + \frac{A_2'}{R_2'} + \frac{\sigma}{S}$, define
\[G^*\left(\frac{A^*}{R^*}\right) = \sum_{\substack{\d \in \mathcal C(D) \\ (DA_1' + \d R_1', R_1) = 1 \\ (DA_2' + R_2'T\sigma - R_2'\d, R_2) = 1}}E\left(\frac{A_1'}{R_1'} + \frac{\d}{D}\right) E\left(\frac{A_2'}{R_2'} + \frac{\sigma}{S} - \frac{\d}{D}\right).\]
For this sum to be nonempty, $(A_1',R_1') = (A_2',R_2') = 1$. Then
\[S(R_1, \dots, R_k) \le \frac{|T|}{|R|} \Big(|R^*| \sum_{A^*\in \mathcal C(R^*)}G^*\Big(\frac{A^*}{R^*}\Big)^2 \Big)^{1/2} \prod_{i=3}^k \Big(|R_i| \sum_{\substack{A_i \in \mathcal R(R_i) \\ |A_i|<|R_i|/q^h}} 1 \Big)^{1/2}\]
By Lemma \ref{lem:mvl4}, the product is $\ll |R_3 \cdots R_k|q^{hk/2-h}.$
Thus it suffices to show that
\[\sum_{A^*\in \mathcal C(R^*)} G^*\left(\frac{A^*}{R^*}\right)^2 \ll |R_1|\cdot|R_2|\cdot|S|q^{2h} (X_1^2 + X_2^2 + X_3^2).\]
By Lemma \ref{lem:mvl6}, 
\[G^*\left(\frac{A^*}{R^*}\right) \ll E\left(\frac{A^*}{R^*}\right)q^{-h}\sum_{\d \in \mathcal C(D)} E\left(\frac{\d}{D} + \frac{A_1'}{R_1'}\right),\]
so by Lemma \ref{lem:mvl4},
\[G^*\left(\frac{A^*}{R^*}\right) \ll \begin{cases} E\left(\frac{A^*}{R^*}\right) \max\{q^{h},|D|\} &\text{ if } \left|\frac{A_1'}{R_1'}\right| < q^{-h} \\
E\left(\frac{A^*}{R^*}\right) |D|q^{-1} &\text{ if } \left|\frac{A_1'}{R_1'}\right| \ge q^{-h}.\end{cases}\]
Summing over $A^*$ then gives
\begin{equation}\label{eq:mvl7-a*-sum}
\sum_{\substack{A^*\in \mathcal C(R^*)}} G^*\left(\frac{A^*}{R^*}\right)^2 \ll \sum_{\substack{A^*\in \mathcal C(R^*) \\ |A^*/R^*| < q^{-h} \\ |A_1'/R_1'| < q^{-h}}} E\left(\frac{A^*}{R^*}\right)^2\max\{q^{2h}, |D|^2\} +  \sum_{\substack{A^* \in \mathcal C(R^*) \\|A^*/R^*| < q^{-h} \\ |A_1'/R_1'| \ge q^{-h}}} E\left(\frac{A^*}{R^*}\right)^2|D|^2.
\end{equation}
Here as in the definition of $G^*$, for any nonzero term we must have $(A_1',R_1') = (A_2',R_2') = 1$. In particular, $A_1' \equiv 0 \mod R_1'$ only if $R_1' = 1$. We now split into cases based on whether or not $|R^*|>q^h$ and whether or not $|R_1'|>q^h$. 

First assume that $|R^*|>q^h$ and $|R_1'| > q^h$. Then
\begin{align*}
\sum_{\substack{A^* \in \mathcal C(R^*)}} G^*\left(\frac{A^*}{R^*}\right)^2 &\ll \max\{q^{2h},|D|^2\} q^{2h}\frac{|R_1'|}{q^h}\frac{|R_2'S|}{q^h} + |D|^2 \sum_{\substack{A^* \in \mathcal C(R^*) \\ |A^*/R^*|<q^{-h} \\ |A_1'/R_1'|\ge q^h}} E\left(\frac{A^*}{R^*}\right)^2 \\
&\ll \max\{q^{2h},|D|^2\}|R^*| + |D|^2 |R^*|q^h \\
&\ll |R_1|\cdot|R_2|\cdot|S|q^{2h}(X_1^2 + X_2^2).
\end{align*} 

Now assume that $|R^*|>q^h$ but $|R_1'| \le q^h$. The first sum in \eqref{eq:mvl7-a*-sum} is empty unless $R_1'=1$ (and $A_1' = 0$). If $R_1'= 1$, then $R_1 = D$, so $|D| > q^h$. Equation \eqref{eq:mvl7-a*-sum} then becomes
\[\sum_{\substack{A^* \in \mathcal C(R^*)}} G^*\left(\frac{A^*}{R^*}\right)^2 \ll q^{2h}|D|^2  + \frac{|R^*|}{q^h}q^{2h}|D|^2 = |R_1R_2S|q^{2h}\left(\frac{1}{|R^*|} + q^{-h}\right) \ll |R_1R_2S|q^{2h}(X_1^2).\]
If $R_1'\ne 1$, then the first sum is empty, so \eqref{eq:mvl7-a*-sum} becomes
\[\sum_{\substack{A^* \in \mathcal C(R^*)}} G^*\left(\frac{A^*}{R^*}\right)^2 \ll \frac{|R^*|}{q^h}q^{2h}|D|^2 = |R_1R_2S|q^{2h}(X_1^2).\]

Finally, assume that $|R^*|\le q^h$ and thus $|R_1'| \le q^h$. In this case the only nonzero term in \eqref{eq:mvl7-a*-sum} in either sum is when $A^* = 0$, which forces $A_1'=A_2'=\sigma=0$. But then since $(A_1',R_1') = (A_2',R_2')=1$, we also have $R_1'=R_2'= 1$, and thus $R_1 = R_2 = D$, which has magnitude $> q^h$. Thus
\[\sum_{\substack{A^* \in \mathcal C(R^*)}} G^*\left(\frac{A^*}{R^*}\right)^2 \ll q^{2h}|D|^2 = |R_1R_2S|q^{2h}\cdot|S|^{-1} =|R_1R_2S|q^{2h}X_3^2.\]
\end{proof}

We now turn to the proof of Theorem \ref{thm:mvfuncfield}, which corresponds to \cite[Lemma 8]{MontgomeryVaughanReducedResidues}. The main strategy here is a careful application of Lemma \ref{lem:mvl7}, keeping in mind that we can choose which variables play the roles of $R_1$ and $R_2$.

\begin{lemma}\label{lem:mvl8}
For any fixed $k \ge 3$, for $Q \in \F_q[t]$ squarefree, for $h \ge 1$ and $m_k(Q;h)$ defined by \ref{eq:defn-of-mk-in-fcn-field},
\[m_k(Q;h) \ll |Q| (q^h)^{k/2}\left(\frac{\phi(Q)}{|Q|}\right)^{k/2} \left(1 + \left((q^h)^{-1/2} + (q^h)^{-1/(k-2)} \right)\left(\frac{\phi(Q)}{|Q|}\right)^{-2^k+k/2}\right). \]
\end{lemma}
\begin{proof}
We begin with the bound that 
\[m_k(Q;h) \ll |Q| \left(\frac{\phi(Q)}{|Q|}\right)^k \sum_{\substack{R|Q \\ R \text{ monic}}}\sum_{\substack{R_i|Q \\ R_i \text{ monic} \\ |R_i| > 1 \\ [R_1, \dots, R_k] = R}} \frac{S(R_1, \dots, R_k)}{\phi(R_1) \cdots \phi(R_k)},\]
where $S(R_1, \dots, R_k) = \sum_{\substack{A_i \in \mathcal R(R_i) \\ \sum_i A_i/R_i = 0}} \prod_{i=1}^k E\left(\frac{A_i}{R_i}\right).$ 
We apply Lemma \ref{lem:mvl7}, but while using the fact that we have flexibility in how we label $R_1, \dots, R_k$ in our application of Lemma \ref{lem:mvl7}. For clarity, we will write $\tilde{R_1}$ and $\tilde{R_2}$ to be the $R_i$'s that serve as the first two in our application of Lemma \ref{lem:mvl7}. Choose $\tilde{R_1}$ and $\tilde{R_2}$ as follows.

If for any $i$, $|R_i|< q^h$, then $S(R_1, \dots, R_k)$ must be $0$, so assume that $|R_i| \ge q^h$ for all $i$. Let $R_{ij} = (R_i,R_j)$. For all $i$, since $R_i | \prod_{i \ne j} R_j$, $R_i|\prod_{i \ne j} R_{ij}$ as well. Thus for all $i$, there exists $j \ne i$ such that $|R_{ij}| \ge |R_i|^{1/(k-1)}$. If for some $i,j$, $|R_{ij}| \ge |R_i|^{1/(k-1)}$ but $R_i \ne R_j$, then pick $\tilde{R_1}$ and $\tilde{R_2}$ to be $R_i$ and $R_j$, respectively.

If no such $i$ exists, then for each $i$, there is some $j \ne i$ with $R_i = R_j$. If there exists any triple $R_i = R_j = R_l$, then pick $\tilde{R_1} = R_i$, $\tilde{R_2} = R_j$. If not, then the $R_i$'s must be equal in pairs and otherwise disjoint, and $k$ must be even. Without loss of generality, say that $R_1 = R_2, R_3 = R_4, \dots, R_{k-1} = R_k$. Write $R = UV$, where $V$ is the product of all primes dividing at least two $R_{2i}$'s, and $U$ is the product of all primes dividing exactly one $R_{2i}$. Then 
\[V^2 |\prod_{i=1}^{k/2} \Big(R_{2i},\prod_{j \ne i} R_{2j} \Big),\]
so there exists some $i$ with $\left|\left(R_{2i},\prod_{j \ne i} R_{2j}\right)\right| \ge |V|^{4/k}$. Take $\tilde{R_1}$ and $\tilde{R_2}$ to be $R_{2i}$ and $R_{2i-1}$. 

Now we return to our bound on $m_k(Q;h)$. We have
\[m_k(Q;h) \ll |Q| \left(\frac{\phi(Q)}{|Q|}\right)^k (q^h)^{k/2}\sum_{\substack{R|Q \\ R \text{ monic}}}\frac 1{|R|}\sum_{\substack{R_i|Q \\ R_i \text{ monic} \\ |R_i| > 1 \\ [R_1, \dots, R_k] = R}} \frac{|R_1 \cdots R_k|}{\phi(R_1)\cdots \phi(R_k)}(X_1 + X_2 + X_3),\]
where the $X_i$ arise by use of Lemma \ref{lem:mvl7} as described above. 

Consider the contribution from each $X_i$. Since $X_1 = q^{-h/2}$, the $X_1$ terms contribute
\begin{align*}
&\ll |Q| \left(\frac{\phi(Q)}{|Q|}\right)^k(q^h)^{k/2-1/2} \sum_{\substack{R|Q \\ R \text{ monic}}} \frac 1{|R|} \sum_{\substack{R_i|Q \\ R_i \text{ monic} \\|R_i| \ge q^h \\ [R_1, \dots, R_k] = R}} \frac{|R_1 \cdots R_k|}{\phi(R_1) \cdots \phi(R_k)} \\
&\ll |Q| \left(\frac{\phi(Q)}{|Q|}\right)^k(q^h)^{k/2-1/2} \prod_{P|Q} \left(1 + \frac 1{|P|} \left(2 + \frac 1{|P|-1}\right)^k\right) \\
&\ll |Q| (q^h)^{k/2-1/2} \left(\frac{\phi(Q)}{|Q|}\right)^{-2^k + k}.
\end{align*}

Now consider the $X_2$ contribution. If $X_2 \ne 0$, then $|R_1'|> q^h$, and by our choice of $R_1, R_2$, $|D| \ge |R_1|^{1/(k-1)} = |R_1' \cdot D|^{1/(k-1)}$. But then $|D|^{-1} \le q^{-h/(k-2)}$, so in turn $X_2 \le q^{-h/(k-2)}$. By the same logic as for the $X_1$ terms, the $X_2$ terms contribute $\ll |Q| (q^h)^{k/2-1/(k-2)}\left(\frac{\phi(Q)}{|Q|}\right)^{-2^k+k}.$

Finally, consider $X_3$. If $X_3 \ne 0$, then $R_1 = R_2$. By our choice of $R_1$ and $R_2$ for the application of Lemma \ref{lem:mvl7}, in this case each $R_i$ is equal to some $R_j$. If there exists some $R_i = R_1 = R_2$, with $i \ge 3$, then $S = R_1 = R_2$, so $|S| > q^h$, and thus for these terms we get a saving of $q^{-h/2}$ and the bound for $X_1$ applies. If not, then $k$ is even and the $R_i$'s must be equal in pairs. Let $R = UV$ as above, where $U$ is the product of irreducibles $P$ dividing exactly one pair of $R_i$'s, and $V$ is the product of all other irreducibles $P$ dividing $R$. Write $R_i = U_iV_i$, where $U_i = (R_i,U)$ and $V_i = (R_i,V)$. For fixed $U,V$, let $C(U,V)$ be the set of $k$-tuples $(R_1, \dots, R_k)$ yielding $U$ and $V$. There are at most $\tau_{k/2}(U)$ choices for $U_2, U_4, \dots, U_k,$ where $\tau_{k/2}$ is the $\frac k2$-fold divisor function. Since $V_i|V$, there are at most $\tau(V)^{k/2}$ choices for $V_2, V_4, \dots, V_k$. Thus $\#|C(U,V)| \le \tau_{k/2}(U)d(V)^{k/2}$. In our application of Lemma \ref{lem:mvl7} we have $|S| \ge |V|^{4/k},$ so
\begin{align*}
\sum_{\substack{UV|Q \\ \text{monic}}} \frac 1{|UV|} \sum_{(R_1, \dots, R_k) \in C(U,V)} \left(\prod_{i=1}^k \frac{|R_i|}{\phi(R_i)}\right) X_3 &\ll \sum_{\substack{UV|Q \\ \text{monic}}}\frac{\tau_{k/2}(U)(|U|/\phi(U))^2 \tau(V)^{k/2} (|V|/\phi(V))^k}{|U|\cdot|V|^{1+2/k}} \\
&= \prod_{P|Q} \left( 1+ \frac{k|P|}{2(|P|-1)^2} + \frac{2^{k/2}(|P|/(|P|-1))^k}{|P|^{1+2/k}}\right) \\
&\ll \left(\frac{\phi(Q)}{|Q|}\right)^{-k/2},
\end{align*}
so the $X_3$ terms contribute $\ll |Q| (q^h)^{k/2}\left(\frac{\phi(Q)}{|Q|}\right)^{k/2}$, which completes the proof.
\end{proof}

The final contribution of $X_3$ only arises when $k$ is even, so when $k$ is odd we have the estimate
\[m_k(Q;h) \ll |Q| ((q^h)^{k/2-1/2} + (q^h)^{k/2-1/(k-2)}) \left(\frac{\phi(Q)}{|Q|}\right)^{k-2^k}.\]
For $k = 3$ this implies that
\[m_3(Q;h) \ll |Q| q^h \left(\frac{\phi(Q)}{|Q|}\right)^{-5}.\]
In the case when $k = 5$, we can bound $m_5(Q;h)$ via a more involved argument.

\section{The fifth moment of reduced residues in the function field setting} \label{sec:funcfieldfifth}

Our goal in this section is to prove Theorem \ref{thm:funcfieldfifth}, which is a stronger bound on $m_5(Q;h)$ when $Q = \prod_{|P| \le q^{6h}} P$. We will also prove Corollary \ref{cor:singseriesff}, bounding $R_3(q^h)$ and $R_5(q^h)$ in the ring $\F_q[t]$. 

Lemma \ref{lem:mvl8} already implies a bound on $m_5(Q;h)$, showing that $m_5(Q:h) \ll |Q| (q^h)^{13/6} \left(\frac{\phi(Q)}{|Q|}\right)^{-27}$. Our goal is a bound where the power of $q^h$ is $2 + \ep$ for all $\ep > 0$; note that Conjecture \ref{conj:rkoddbound} would predict a bound where the power of $q^h$ is $2$. In turn, this will allow us to prove Corollary \ref{cor:singseriesff}, that $R_5(q^h) \ll q^{(2+\ep)h}$. 

\subsection{Proof of Theorem \ref{thm:funcfieldfifth}}

As in the proof of Lemma \ref{lem:mvl8}, we begin by bounding
\begin{equation*}
m_5(Q;h) \ll |Q| \left(\frac{\phi(Q)}{Q}\right)^5 \sum_{\substack{R|Q \\ R \text{ monic}}} \sum_{\substack{R_i|Q \\ R_i \text{ monic} \\ |R_i| > 1 \\ [R_1, \dots, R_5] = R}} \frac{S(R_1, \dots, R_5)}{\phi(R_1) \cdots \phi(R_5)},
\end{equation*}
where $S(R_1, \dots, R_5) = \sum_{\substack{A_i \in \mathcal R(R_i) \\ \sum_i A_i/R_i = 0}} \prod_{i=1}^5 E\left(\frac{A_i}{R_i}\right)$.

Our goal is to apply Lemma \ref{lem:mvl7} to bound the size of $S(R_1, \dots, R_5)$. But, when applying this lemma, we can freely choose which of the $R_i$'s plays the roles of $R_1$ and $R_2$. As in the previous section, we will denote our choice by $\tilde{R_1}$ and $\tilde{R_2}$. If any $R_i$ satisfies $|R_i|<q^h$, the choice is immaterial, so assume that $|R_i|\ge q^h$ for all $i$. If there is any triple $R_i, R_j, R_\ell$ with $R_i = R_j=R_\ell$, pick $\tilde{R_1} = R_i$ and $\tilde{R_2} = R_j$. In this case $X_2$ will have no contribution, and $X_3$ and $X_1$ will each be $\ll q^{-h/2}$, for a total contribution to $m_5(Q;h)$ from these terms (as in the proof of Lemma \ref{lem:mvl8}) of $\ll |Q| q^{2h} \left(\frac{\phi(Q)}{|Q|}\right)^{-27}$. If there is no such triple, but there exists $R_i \ne R_j$ with either $\left|\frac{R_i}{(R_i,R_j)}\right|< q^h$, or $\left|\frac{R_i}{(R_i,R_j)} \right| \ge q^h$ and $|(R_i,R_j)| \ge q^{h/2}$, then we choose $\tilde{R_1} = R_i$ and $\tilde{R_2} = R_j$. In this case we have $X_3 = 0$ and $X_1, X_2$ each contributing $\ll q^{-h/2}$, and again the total contribution to $m_5(Q;h)$ from these terms is $\ll |Q| q^{2h} \left(\frac{\phi(Q)}{|Q|}\right)^{-27}$. So, it remains to bound what happens in the remaining cases. We first show that in the remaining cases, up to some reordering, certain factors of $R_2$ and $R_3$ are bounded.

\begin{lemma}\label{lem:notmvtuples}
For fixed squarefree $Q \in \F_q[t]$, let $(R_1, \dots, R_5)$ be a tuple of divisors of $Q$ such that
\begin{itemize}
	\item $|R_i| \ge q^h$ for all $i$,
	\item no three $R_i$'s are equal,
	\item for any $R_i$, $R_j$, either $R_i = R_j$, or $\left|\frac{R_i}{(R_i,R_j)}\right| \ge q^h$ and $|(R_i,R_j)| < q^{h/2}$, and
	\item $R_1$, $R_2$, and $R_3$ are all distinct.
\end{itemize}
Then 
\begin{itemize}
	\item $\left|\frac{R_2}{(R_1,R_2)}\right| \ge q^h$, and 
	\item $\left|\frac{R_3}{(R_3,R_1R_2)}\right| \ge q^{h/2}$.
\end{itemize}
\end{lemma}
Loosely, this lemma states that in the cases that we cannot already bound by the tools of the previous section, prime factors must ``spread out'' among the first three $R_i$'s.
\begin{remark}
The bound on $\left|\frac{R_3}{(R_3,R_1R_2)}\right|$ above is worse than the bound on $\left|\frac{R_2}{(R_1,R_2)}\right|$. In order to apply Lemma \ref{lem:intervalbound} below, we will need both of them to be at least of size $q^{h/2}$, so the bound on $\left|\frac{R_2}{(R_1,R_2)}\right|$ is better than necessary.

However, the fact that these bounds get worse is precisely what prevents us from applying our technique to bound higher moments. If instead we applied the same argument to a $7$-tuple $(R_1, \dots, R_7)$ of divisors of $Q$, we would not be able to guarantee that $\left|\frac{R_4}{(R_4,R_1R_2R_3)}\right| \ge q^{h/2}$, even if we weaken the conditions to allow reordering. This threshold is crucial for our argument, which does not generalize to $7$-tuples.
\end{remark}
\begin{proof}
The fact that $\left|\frac{R_2}{(R_1,R_2)}\right| \ge q^h$, follows directly from the third assumption, since $R_1 \ne R_2$. 

For the second conclusion, let $R_{123} = \gcd(R_1,R_2,R_3)$ and let $R_{13} = \frac{(R_1,R_3)}{\gcd(R_1,R_2,R_3)}$ and $R_{23} = \frac{(R_2,R_3)}{(R_1,R_2,R_3)}$, so that $R_{13}$ is the product of all primes dividing $R_1$ and $R_3$ but not $R_2$, and vice versa. Then $(R_3,R_1R_2) = R_{13}R_{23}R_{123}$. By assumption, $|(R_2,R_3)| < q^{h/2}$, so $|R_{23}R_{123}|<q^{h/2}$, and in particular $|R_{23}|<q^{h/2}$. Now assume by contradiction that $\left|\frac{R_3}{(R_3,R_1R_2)}\right| < q^{h/2}$. Then
\begin{align*}
\left| \frac{R_3}{(R_1,R_3)}\right| &= \left|\frac{R_3}{R_{13}R_{123}} \right| = \left|\frac{R_3}{R_{13}R_{23}R_{123}}\right| \cdot |R_{23}| < q^{h/2} \cdot q^{h/2} = q^h,
\end{align*}
which contradicts the third assumption because $R_1 \ne R_3$. 
\end{proof}

The following auxiliary lemma provides a standard bound on $\tau_k$, the $k$-fold divisor function, in the function field setting. We will also use that $\phi(F) \gg \frac{|F|}{\log\log|F|}$ for all $F \in \F_q[t]$. 

\begin{lemma}\label{lem:tkbound}
Fix $k \ge 1$. Let $M = \max_{b \ge 1} (\tau_k(t^b))^{1/b}$. Then
\[\limsup_{\deg F \to \infty} \frac{\log \tau_k(F) \log \log |F|}{\log |F|} = \log M,\]
and thus for all $\ep > 0$, $\tau_k(F) \ll_\ep |F|^\ep$. 
\end{lemma}
\begin{proof}
The proof of the above lemma follows closely along the lines of Shiu \cite{pshiu}. We will show one direction of the statement, adapted to our setting; the other direction also follows very closely, so we omit it. Note first that 
\[1 \le (\tau_k(t^b))^{1/b} = \binom{b + k-1}{b}^{1/b} < \left(\frac{(b+k-1)e}{k-1}\right)^{(k-1)/b} \to 1\]
as $b \to \infty$, so $M$ exists. 

We now show that $\limsup_{\deg F \to \infty} \frac{\log \tau_k(F) \log \log |F|}{\log |F|} \ge \log M$. Fix $b$ such that $\tau_k(t^b) = M^b$. Let 
\[F = \prod_{\substack{\deg P = d \\ P \text{ irred.}}} P^b,\]
so that $\tau_k(F) = \prod_{\deg P = d} \tau_k(P^b) = (\tau_k(t^b))^{\pi(d;\mathbb F_q)} = M^{b\pi(d;\mathbb F_q)}$. We have that $\pi(d;\F_q) \sim \frac{q^d}{d}$ as $d \to \infty$, so that
\[\log |F| = bd \log q \pi(d;\F_q) \sim b q^d\log q,\]
and
\[\log \log |F| = d \log q + O(1).\]
Thus as $d \to \infty$, 
\begin{align*}
\log \tau_k(F) &= b \pi(d;\F_q) \log M \\
&\sim b \log M \cdot \frac{q^d}{d} \sim \frac{\log M \log |F|}{\log \log |F|}, 
\end{align*}
so $\limsup_{\deg F \to \infty} \frac{\log \tau_k(F) \log \log |F|}{\log |F|} \ge \log M$. 

As mentioned above, the proof that $\limsup_{\deg F \to \infty} \frac{\log \tau_k(F) \log \log |F|}{\log |F|} \le \log M$ also follows Shiu's proof in \cite{pshiu} closely, so we omit it.
\end{proof}
The above bound implies that for all $\ep > 0$, $\tau_k(F) = |F|^{O(1/\log\log |F|)} = O_{\ep}(|F|^{\ep})$.

Here we have a final preparatory lemma before the main proposition leading to the bound on $m_5(Q;h)$. In what follows, our main strategy will be carefully isolating factors of the $R_i$'s in order to bound the number of terms in our sum. In doing so, we will make use of the following bound. 

\begin{lemma}\label{lem:intervalbound}
Let $Q\in \F_q[t]$ be a squarefree polynomial, and let $n \in \N_{\ge 2}$. Let $\mathcal I \ins \F_q(t)$ be an interval of size $q^{-h}$. That is to say, for some rational function $\a \in \F_q(t)$, let $\mathcal I := \{\b \in \F_q(t):|\a-\b|<q^{-h}\}$. Assume in the following that $X_i,Y_i \in \F_q[t]$ for all $i$. Then for any $\ep > 0$,
\[\sum_{\substack{Y_1, \dots, Y_n|Q \\ X_i \in \mathcal R(Y_i) \\ \sum_i X_i/Y_i \in \mathcal I \\ q^{h/2}\le \left|\prod_i Y_i\right| \le q^{2h}}} \frac{\mu\left(\prod_i Y_i\right)^2}{\prod_i \phi(Y_i)^2} \ll_{n,\ep} q^{-h(1-\ep)}. \]
\end{lemma}
\begin{proof}
For given $X_1, \dots, X_n$ and $Y_1,\dots,Y_n$, let $X$ and $Y$ be defined so that $Y = \prod_i Y_i$ and $\frac{X}{Y} = \sum_i \frac{X_i}{Y_i}$. Then for all tuples considered in the sum, $\frac XY \in \mathcal I$ and $q^{h/2} \le |Y| \le q^{2h}$. Proceed by counting the number of possibilities for $\frac XY$ satisfying this constraint, which is bounded above by the number of points in $\mathcal I$ with denominator smaller than $q^{2h}$, and finally count the number of ways of splitting $Y$ up into $Y_1, \dots, Y_n$. However, we want to also consider the weighting in the sum of $\frac 1{\phi(Y)^2}$, so we start by splitting the sum up into different sizes of $Y$, and then applying bounds on $\phi(Y)$. 

To begin with, we rewrite the sum in terms of $X$ and $Y$. Note that all $Y_i$ in our sum are relatively prime, because of the M\"obius factor. Thus $Y$ is squarefree and $\phi(Y) = \prod_i \phi(Y_i)$. Moreover, a choice of $X$, $Y$, and a decomposition $Y = Y_1 \cdots Y_n$ determines $X_i$ for each $i$ by the Chinese Remainder Theorem. Our sum is thus equal to 
\[\sum_{\substack{Y|Q \\ q^{h/2} \le |Y| \le q^{2h}}} \sum_{\substack{X \in \mathcal R(Y) \\ X/Y \in \mathcal I}}\frac{\mu(Y)^2}{\phi(Y)^2} \#\{Y_1, \dots, Y_n : Y_1 \cdots Y_n = Y\}.\]
Now split the sum up according to $|Y|$, defining $m := \deg Y$. The sum is then equal to
\begin{align*}
&\sum_{m=h/2}^{2h} \sum_{\substack{Y|Q \\ |Y| = q^m}} \sum_{\substack{X \in \mathcal R(Y)\\X/Y \in \mathcal I}} \frac{\mu(Y)^2}{\phi(Y)^2} \tau_n(Y) \\
&\ll_{n,\ep} \sum_{m=h/2}^{2h} (q^m)^{\ep/3} \sum_{\substack{Y|Q \\ |Y| = q^m}} \sum_{\substack{X \in \mathcal R(Y)\\X/Y \in \mathcal I}} \frac{\mu(Y)^2(\log \log |Y|)^2}{|Y|^2},
\end{align*}
by Lemma \ref{lem:tkbound} and the fact that $\phi(Y)^{-2} \ll \left(\frac{|Y|}{\log \log |Y|}\right)^{-2}$. We can further relax the condition that $|Y| = q^m$ to the condition that $|Y| \le q^m$. The number of $X/Y$ with $|Y| \le q^m$ in the interval $\mathcal I$ is $q^{2m-h} + O(1)$; since $m \ge h/2$, this is $\ll q^{2m-h}$. Thus the sum is
\begin{align*}
&\ll_{n,\ep} \sum_{m=h/2}^{2h} q^{m(\ep/3)}\frac{(\log \log (q^m))^2}{q^{2m}} q^{2m-h} \ll q^{-h} \sum_{m=h/2}^{2h} q^{m(2\ep/3)} \ll q^{-h(1 -\ep)},
\end{align*}
as desired. 
\end{proof}

We now turn to bounding the contribution to the fifth moment $m_5(Q;h)$ coming from tuples $(R_1, \dots, R_5)$ satisfying the conclusions of Lemma \ref{lem:notmvtuples}.

\begin{proposition}\label{prop:fifthmomentmachine}
Fix $h \ge 1$ and let $Q \in \F_q[t]$ be squarefree. Let $\mathcal S$ be the set of tuples $(R_1, \dots, R_5)$ such that
\begin{itemize}
	\item $R_i|Q$ for all $i$,
	\item $q^h \le |R_i| \le q^{2h}$ for all $i$,
	\item $\left|\frac{R_2}{(R_1,R_2)}\right| \ge q^{h/2}$, and
	\item $\left|\frac{R_3}{(R_3,R_1R_2)}\right| \ge q^{h/2}$.
\end{itemize}
Then for all $\ep > 0$, 
\[\sum_{\substack{(R_1, \dots, R_5) \in \mathcal S}} \prod_{i=1}^5 \frac{1}{\phi(R_i)} \sum_{\substack{A_i \in \mathcal R(R_i) \\ |A_i/R_i|<q^{-h} \\ \sum A_i/R_i = 0  \\ 1 \le i \le 5}} q^{5h} \ll q^{(2+\ep)h}\frac{|Q|}{\phi(Q)}.\]
\end{proposition}
\begin{proof}
We begin by sketching an overview of the strategy. For each subset $I \ins [5]$, let $R_I = \prod_{\substack{P|R_i \forall i \in I \\ P\nmid R_i \forall i \not\in I}} P$ be the product of the irreducible factors dividing $R_i$ if and only if $i \in I$. Note that these $R_I$'s must be pairwise relatively prime. 

We start by using the constraint that $\left|\frac{A_1}{R_1}\right|<q^{-h}$. We will count the total number of rational functions in this interval with denominator of degree at most $2h$. For each option of $\frac{A_1}{R_1}$, we can decompose $R_1 = \prod_{I\ni 1} R_I$, so the number of ways to decompose $R_1$ into these $R_I$ factors is $\tau_{2^{k-1}-1}(R_1)$, which we can bound based on the degree of $R_1$. We then also get $\frac{A_1}{R_1} = \sum_{I \ni 1} \frac{A_I}{R_I}$, where the $A_I$'s are determined by the Chinese Remainder Theorem. 

We will then focus on the constraint that $\left|\frac{A_2}{R_2}\right| < q^{-h}$. However, $(R_1,R_2) = \prod_{1,2 \in I} R_I$ has already been fixed, so the same analysis as used for $R_1$ applies to the remaining factors of $R_2$. Crucially, $\frac{R_2}{(R_1,R_2)}$ remains relatively large by assumption, which will ensure that we save enough by doing this. Finally, the constraint on $\frac{A_3}{R_3}$, using our assumption that $\frac{R_3}{(R_3,R_1R_2)}$ is large enough, yields savings in the same way.

We begin by rewriting our sum in terms of the $R_I$. For each subset $I \ins [5]$, and for a fixed $R_1,\dots,R_5$, we again define $R_I$ to be the product of all primes $P$ so that $P$ divides $R_i$ for each $i \in I$ \emph{and} $P$ does \emph{not} divide $R_j$ for all $j \not\in I$. The $R_I$ are a system of \emph{relative greatest common divisors}; see \cite{MR4091066} for details. For example, $R_{\{1,2\}}$ is the product of all primes dividing $R_1$ and $R_2$, but $(R_{\{1,2\}}, R_j) = 1$ for $j = 3,4,5$. The polynomials $R_I$ must satisfy the following properties, implied by the constraints on the $R_i$'s:
\begin{itemize}
	\item Each $R_I$ divides $Q$, and for each $I \ne J \ins [5]$, $(R_I, R_J) = 1$.
	\item Each irreducible polynomial dividing an $R_i$ must divide at least two of them in order for the sum over $A_i$ to be nonempty, so $R_I = 1$ unless $|I| \ge 2$. We will always assume that $|I| \ge 2$.
	\item Each choice of $A_i$ is equivalent to a choice of $A_{i,I}$ for all subsets $I$ containing $i$, that is, $\frac{A_i}{R_i} = \sum_{I \ni i} \frac{A_{i,I}}{R_I}.$
	\item The quantity $(A_i,R_i) = 1$ for all $i$ if and only if $(A_{i,I},R_I) = 1$ for all $I,i$.
	\item The constraint that for all $i$, $|A_i/R_i|<q^{-h}$, implies that for each index $i$,
	\[\left|\sum_{I \ni i} \frac{A_{i,I}}{R_I} \right| < q^{-h}. \]
	\item The constraint that $\sum_{i=1}^5 A_i/R_i = 0$ implies that for each subset $I$,
	\[\sum_{i \in I} A_{i,I} = 0.\]
Finally, define $\ell_I$ to be the minimum element of a subset $I \ins [5]$. The requirement that $(R_1, \dots, R_5) \in \mathcal S$ implies the following:
	\item For all $i$, 
	\[q^h \le \left|\prod_{I \ni i} R_I \right| \le q^{2h}.\]
	\item Since $\frac{R_2}{(R_1,R_2)} = \prod_{\ell_I = 2} R_I$, and $\frac{R_3}{(R_1,R_2,R_3)} = \prod_{\ell_I = 3} R_I$,
	\[\left|\prod_{\ell_I = 2} R_I \right| \ge q^{h/2} \quad \text{ and } \quad \left|\prod_{\ell_I = 3} R_I \right| \ge q^{h/2}. \]
\end{itemize}
The sum under consideration is then 
\[\ll q^{5h} \sum_{\substack{R_I|Q \\ I \ins [5] \\ q^h \le \left|\prod_{I \ni i} R_I \right| \le q^{2h} \\ \left|\prod_{\ell_I = 2} R_I \right| \ge q^{h/2} \\ \left|\prod_{\ell_I = 3} R_I \right| \ge q^{h/2}}} \frac{\mu\left(\prod_I R_I\right)^2}{\prod_I \phi(R_I)^{|I|}} \sum_{\substack{I, i \in I \\ A_{i,I} \in \mathcal R(R_I) \\ \forall i, \left| \sum_{I \ni i} A_{i,I}/R_I \right| < q^{-h} \\ \forall I, \sum_{i \in I} A_{i,I} = 0}} 1.\]

Note first that if $m_i$ is the maximum element of a subset $I$, then $A_{m_i,I}$ is fully determined by the other $A_{i,I}$ and the fact that $\sum_{i \in I} A_{i,I} = 0$. Then for $i \in I$ with $\ell_I < i < m_I$, we will use the trivial bound on the number of options for $A_{i,I}$; namely that there are at most $R_I$ choices for $A_{i,I}$. For the rest of this bound, we treat $A_{i,I}$ as fixed when $\ell_I < i < m_I$. 

We finally consider the number of options for the remaining $A_{\ell_I,I}$, where $\ell_I$ is the smallest element in $I$, which is where the savings in the argument will come from. We will proceed by ordering the intervals $I$ in our sum by $\ell_I$; we will first sum over options for $A_I$ when $I = \{4,5\}$, with $\ell_I = 4$, and then over $A_{i,I}$ for all $I$ with $\ell_I = 3$, and so on. As we do this, we will need at each step to satisfy the constraints that for each $i$,
\begin{equation}\label{eq:foralliAiIinterval}
\left| \sum_{I \ni i} \frac{A_{i,I}}{R_I} \right| < q^{-h},
\end{equation}
where as we split up the sums over different $A_{i,I}$'s, some of the values in this sum will be fixed and others will still be free to vary in our sum. But even if some of the terms in the sum above are fixed, the remaining terms are still constrained to lie in some interval of size $q^{-h}$, possibly an interval centered at a non-zero rational function. In particular, the constraints in \eqref{eq:foralliAiIinterval} are equivalent to the constraints that for all $i$,
\begin{equation*}
\Big| F_i + \sum_{\substack{I \subseteq [5] \\ \ell_I = i}} \frac{A_{i,I}}{R_I}\Big| < q^{-h},
\end{equation*}
where $F_i$ is a fixed rational function determined by the values of $A_{i,I}$ when $\ell_I < i < m_I$. The bounds we use are independent $F_i$, only requiring that the size of the interval is $q^{-h}$, so we can replace $F_i$ by $0$. This yields the following sum.
\begin{equation}\label{eq:sumjustoverellis}
\ll q^{5h} \sum_{\substack{R_I|Q \\ I \ins [5] \\ q^h \le \left|\prod_{I \ni i} R_I \right| \le q^{2h} \\ \left|\prod_{\ell_I = 2} R_I \right| \ge q^{h/2} \\ \left|\prod_{\ell_I = 3} R_I \right| \ge q^{h/2}}} \frac{\mu\left(\prod_I R_I \right)^2}{\prod_I \phi(R_I)^{|I|}} \prod_I \phi(R_I)^{|I|-2} \sum_{\substack{A_{\ell_I,I} \in \mathcal R(R_I) \\ I \ins [5] \\ \forall i, \left|\sum_{\ell_J = i} A_{i,J}/R_J \right|<q^{-h}}}1.
\end{equation}

The only terms $A_{i,I}$ that remain in \eqref{eq:sumjustoverellis} are of the form $A_{\ell_I,I}$, there is only one term for each subset $I$, so to simplify our notation we will write $A_I := A_{\ell_I,I}$ from now on. 

Consider subsets $I$ with $\ell_I = 4$. There is only one of these, namely $\{4,5\}$, so we rewrite the sum as follows:
\[\ll q^{5h} \sum_{\substack{R_I|Q \\ I \ins [5], I \ne \{4,5\} \\ q^h \le \left|\prod_{I \ni i} R_I \right| \le q^{2h} \\ \left|\prod_{\ell_I = 2} R_I \right| \ge q^{h/2} \\ \left|\prod_{\ell_I = 3} R_I \right| \ge q^{h/2}}} 
\frac{\mu\left(\prod_I R_I \right)^2}{\prod_I \phi(R_I)^2} \sum_{\substack{A_{I} \in \mathcal R(R_I)\\ I \ins [5], I \ne \{4,5\} \\ \forall i, \left|\sum_{\ell_J = i} A_{J}/R_J \right|<q^{-h}}} 
\sum_{\substack{R_{\{4,5\}} | Q \\ A_{\{4,5\}} \in \mathcal R(R_{\{4,5\}})}} \frac 1{\phi(R_{\{4,5\}})^2}. 
\]
In the inside sum, we have dropped the additional constraint that $\frac{A_{\{4,5\}}}{R_{\{4,5\}}}$ must lie in an interval of size $q^{-h}$, since ignoring it only increases the size of the sum. For each $R_{\{4,5\}}$, there are $\phi(R_{\{4,5\}})$ choices of $A_{\{4,5\}}$, so the inner sum becomes
\[\sum_{\substack{R_{\{4,5\}} | Q}} \frac 1{\phi(R_{\{4,5\}})} = \frac{|Q|}{\phi(Q)}, \]
since $Q$ is squarefree. 

Now consider subsets $I$ with $\ell_I = 3$, i.e. $\{3,4\}$, $\{3,4,5\}$, and $\{3,5\}$. We first bookkeep by isolating these terms in the sum, yielding
\[\ll q^{5h} \frac{|Q|}{\phi(Q)} \sum_{\substack{R_I|Q \\ I \ins [5], \ell_I < 3 \\ q^h \le \left|\prod_{\ell_I = 1} R_I \right| \le q^{2h} \\ q^{h/2} \le \left|\prod_{\substack{\ell_I = 2}} R_I \right| \le q^{2h}}} 
\frac{\mu\left(\prod_I R_I \right)^2}{\prod_I \phi(R_I)^2} 
\sum_{\substack{A_I \in \mathcal R(R_I)\\ I \ins [5], \ell_I < 3 \\ \forall i, \left|\sum_{\ell_J = i} A_J/R_J \right|<q^{-h}}}
\sum_{\substack{\ell_I = 3 \\ R_I|Q \\ A_I \in \mathcal R(R_I) \\ q^{h/2} \le \left|\prod_{\substack{\ell_I = 3}} R_I \right| \le q^{2h} \\ \left|\sum_{\ell_I = 3} A_I/R_I\right|<q^{-h}}} \frac {\mu\left(\prod_{\ell_I = 3} R_I\right)^2}{\prod_{\ell_I = 3} \phi(R_I)^2}.\]

We now bound the inner sum using Lemma \ref{lem:intervalbound}. The inner sum comprises three terms $R_I$, so apply the lemma with $n = 3$, to get that the inner sum is $\ll q^{-h(1+\ep)}$. 

We repeat the process, now considering subsets $I$ with $\ell_I = 2$. Isolating these terms yields
\[\ll q^{4h+\ep h} \frac{|Q|}{\phi(Q)} \sum_{\substack{R_I|Q \\ I \ins [5], \ell_I =1 \\ q^h \le \left|\prod_{\ell_I = 1} R_I \right| \le q^{2h}}} 
\frac{\mu\left(\prod_{\ell_I=1} R_I \right)^2}{\prod_{\ell_I=1} \phi(R_I)^2} 
\sum_{\substack{A_I \in \mathcal R(R_I)\\ I \ins [5], \ell_I =1 \\ \left|\sum_{\ell_I = 1} A_I/R_I \right|<q^{-h}}}
\sum_{\substack{\ell_I = 2 \\ R_I|Q \\ A_I \in \mathcal R(R_I) \\ q^{h/2} \le \left|\prod_{\substack{\ell_I = 2}} R_I \right| \le q^{2h} \\ \left|\sum_{\ell_I = 2} A_I/R_I\right|<q^{-h}}} \frac {\mu\left(\prod_{\ell_I = 2} R_I\right)^2}{\prod_{\ell_I = 2} \phi(R_I)^2}.\]
Here there are seven $R_I$ terms and seven $A_I$ terms in the inner sum, so, again applying Lemma \ref{lem:intervalbound}, the inner sum is $\ll q^{-h+\ep h}$. Lastly, we address the terms with $\ell_I = 1$:
\[\ll q^{3h} q^{2\ep h} \frac{|Q|}{\phi(Q)} \sum_{\substack{R_I|Q \\ I \ins [5], \ell_I =1 \\ q^h \le \left|\prod_{\ell_I = 1} R_I \right| \le q^{2h}}} 
\frac{\mu\left(\prod_{\ell_I=1} R_I \right)^2}{\prod_{\ell_I=1} \phi(R_I)^2} 
\sum_{\substack{A_I \in \mathcal R(R_I)\\ I \ins [5], \ell_I =1 \\ \left|\sum_{\ell_I = 1} A_I/R_I \right|<q^{-h}}} 1.\]
We apply Lemma \ref{lem:intervalbound} one final time, this time with $n = 15$, since there are $15$ sets $I \ins [5]$ with $|I| \ge 2$ and $\ell_I = 1$. This yields
\[\ll q^{2h+3\ep h}\frac{|Q|}{\phi(Q)},\]
as desired.
\end{proof}

We are now ready to prove a general bound on $m_5(Q;h)$.
\begin{theorem}
Fix $\ep > 0$ and let $Q \in \F_q[t]$ be squarefree. Define $m_5(Q;h)$ by \eqref{eq:defn-of-mk-in-fcn-field}. Then
\[m_5(Q;h) \ll |Q|q^{2h + \ep}\left(\frac{|Q|}{\phi(Q)}\right)^{-4} + |Q|q^{2h}\left(\frac{|Q|}{\phi(Q)}\right)^{27}.\]
\end{theorem}
\begin{proof}
Using Lemma \ref{lem:mvl2}, we can express
\[m_5(Q;h) = |Q| \left(\frac{\phi(Q)}{|Q|}\right)^5 V_5(Q;h),\]
where 
\[V_5(Q;h) = \sum_{\substack{R_1, \dots R_5 |Q \\ |R_i| > 1 \\ R_i \text{ monic}}} \prod_{i=1}^5 \frac{\mu(R_i)}{\phi(R_i)} \sum_{\substack{A_1, \dots, A_5 \in \mathcal R(R_i) \\ \sum_i A_i/R_i = 0}} E\left(\frac{A_1}{R_1}\right) \cdots E\left(\frac{A_5}{R_5}\right).\]
Now apply Lemma \ref{lem:mvl7} to bound the contribution to $V_5(Q;h)$ from many tuples $R_1, \dots, R_5$. If $|R_i| < q^h$ for any $i$, then these terms contribute $0$; assume from now on that $|R_i| \ge q^h$. If for any triple $i,j,k$ we apply Lemma \ref{lem:mvl7} with $R_1 = R_i$ and $R_2 = R_j$; in this case $X_2 = 0$ and $X_1$ and $X_3$ are $O(q^{-h/2})$, so these terms contribute $O\left( q^{2h} \left(\frac{|Q|}{\phi(Q)}\right)^{32}\right)$. If there exist $R_i \ne R_j$ such that either $\left|\frac{R_i}{(R_i,R_j)}\right| < q^h$ or $|(R_i,R_j)| \ge q^{h/2}$; in this case, $X_3 =0$, and $X_1$ and $X_2$ are each $O(q^{-h/2})$, so these terms contribute $O\left( q^{2h} \left(\frac{|Q|}{\phi(Q)}\right)^{32}\right)$ as well.

Assume now that $(R_1, R_2, R_3, R_4, R_5)$ does not fall into either of the above cases. Then for all $i$, $|R_i| < q^{2h}$. To see this, assume that $(R_1,R_2,R_3,R_4,R_5)$ has no $i,j,k$ with $R_i = R_j = R_k$, and that for all $R_i \ne R_j$, $\left|\frac{R_i}{(R_i,R_j)}\right| \ge q^h$ and $|(R_i,R_j)| < q^{h/2}$. Assume, relabeling if necessary, that $R_1 \ge q^{2h}$. Since $R_1|\prod_{j \ne 1} (R_1,R_j)$, we must have $|(R_1,R_j)|\ge q^{h/2}$ for some $j \ne 1$. This cannot be true for some $j$ with $R_j \ne R_1$, so we have $R_j = R_1$. At the same time, there can only be one $j \ne 1$ with $R_j = R_1$, so without loss of generality our tuple must be of the form $(R_1, R_1, R_3, R_4,R_5)$. There cannot be an additional equal pair among $R_3, R_4,$ and $R_5$; if there is (without loss of generality $R_3 = R_4$), then $R_5|(R_1,R_5)(R_3,R_5)$, so since $|R_5| \ge q^h$ either $|(R_1,R_5)| \ge q^{h/2}$ or $|(R_3,R_5)|\ge q^{h/2}$, which along with the lack of equal triples yields a contradiction. Now consider $R_3$. Note that $R_3|(R_1,R_3)(R_4,R_3)(R_5,R_3)$, and that $\frac{R_3}{(R_1,R_3)}|(R_4,R_3)(R_5,R_3)$. But by assumption, $\left|\frac{R_3}{(R_1,R_3)}\right| \ge q^h$ and $|(R_4,R_3)(R_5,R_3)| < (q^{h/2})^2 = q^h$, which yields a contradiction.

So, the only terms remaining are those with $|R_i| < q^{2h}$ for all $i$, no equal triple, and either $\left|\frac{R_i}{(R_i,R_j)}\right| < q^h$ or $|(R_i,R_j)|\ge q^{h/2}$ whenever $R_i \ne R_j$. By Lemma \ref{lem:notmvtuples}, $(R_1, \dots, R_5)$ satisfies the constraints of Proposition \ref{prop:fifthmomentmachine}. By Proposition \ref{prop:fifthmomentmachine}, these terms contribute $O\left(q^{(2+\ep)h}\frac{|Q|}{\phi(Q)}\right)$ to $V_5(Q;h)$ for all $\ep > 0$. Thus for all $\ep > 0$,
\[V_5(Q;h) \ll q^{(2+\ep)h}\frac{|Q|}{\phi(Q)} + q^{2h} \left(\frac{|Q|}{\phi(Q)}\right)^{32},\]
so $m_5(Q;h) \ll |Q|q^{(2+\ep)h}\left(\frac{|Q|}{\phi(Q)}\right)^{-4} + |Q|q^{2h}\left(\frac{|Q|}{\phi(Q)}\right)^{27}$.
\end{proof}
As in the integer case, we particularly want to consider $Q$ to be the product of irreducible polynomials $P$ with $|P| \le q^{2h}$. In this case, $\frac{|Q|}{\phi(Q)} \ll h$, so that we get the following corollary.
\begin{corollary}\label{cor:m5Qh-bound-for-primorial-Q}
Fix $\ep > 0$ and let $Q \in \F_q[t]$ be given by $Q = \prod_{\substack{P \text{irred.} \\ |P| \le q^{2h}}} P$. Then
\[m_5(Q;h) \ll_\ep |Q| q^{(2+\ep)h}.\] 
\end{corollary}

\subsection{Proof of Corollary \ref{cor:singseriesff}: Bounds on $R_k(q^h)$}
In this subsection, we discuss the transition from bounds on $V_k(Q;h)$, from Theorem \ref{thm:funcfieldfifth} and Lemma \ref{lem:mvl8}, to bounds on sums of singular series in function fields, in order to prove Corollary \ref{cor:singseriesff}. Much of this is similar to the integer case discussion in Section \ref{sec:threeterminteger}.

As in the integer case, for $\mathcal D = \{D_1, \dots, D_k\}$ a set of distinct polynomials in $\mathbb F_q[T]$, we define the singular series
\[
\mathfrak S(\mathcal D) := \prod_{P \text{ monic, irred.}} \frac{(1-\nu_P(\mathcal D)/|P|)}{(1-1/|P|)^k}
= \sum_{\substack{R_1, \dots, R_k \\ 1 \le |R_i|}} \left(\prod_{i=1}^k \frac{\mu(R_i)}{\phi(R_i)} \right) \sum_{\substack{A_1, \dots, A_k \\ A_i \in \mathcal R(R_i) \\ \sum_i A_i/R_i = 0}} e \left(\sum_{i=1}^k \frac{A_iD_i}{R_i}\right),
\]
where $\nu_P(\mathcal D)$ is the number of equivalence classes of $\mathbb F_q[T]/(P)$ occupied by elements of $\mathcal D$. We also define $\mathfrak S_0(\mathcal D)$, given by $\mathfrak S_0(\mathcal D) := \sum_{\mathcal J \ins \mathcal D} (-1)^{|\mathcal D \setminus \mathcal J|}\mathfrak S(\mathcal J)$, and consider
\begin{equation}\label{eq:fcn-field-def-of-rk-qh}
R_k(q^h) := \sum_{\substack{D_1, \dots, D_k \\ D_i \text{ distinct} \\ |D_i| \le q^h}} \mathfrak S_0(\{D_1, \dots, D_k\}).
\end{equation}

Our results on $m_k(Q;h)$ (and equivalently $V_k(Q;h)$) imply bounds on these sums of $k$-fold singular series, just as in the integer case in Section \ref{sec:threeterminteger}. We set $Q$ to be the product of all monic irreducible polynomials of degree at most $2h$, so that $\frac{|Q|}{\phi(Q)} \ll_q h$. Just as in the integer case, we can truncate the expression for $\mathfrak S_0(\mathcal D)$ to only contain terms dividing $Q$, with an acceptable error term. In particular, we get
\[R_k(h) = \sum_{\substack{D_1, \dots, D_k \\ D_i \text{ distinct} \\ |D_i| \le q^h}} \sum_{\substack{R_1, \dots, R_k \\ |R_i| > 1 \\ R_i|Q}} \prod_{i=1}^k \frac{\mu(R_i)}{\phi(R_i)} \sum_{\substack{A_1, \dots, A_k \\ A_i \in \mathcal R(R_i) \\ \sum_i A_i/R_i = 0}} e\left(\sum_{i=1}^k \frac{D_iA_i}{R_i}\right) + O(1).\]

It will again be helpful for us to define the singular series of a $k$-tuple $\mathcal D = (D_1, \dots, D_k)$ relative to the modulus $Q$. Here the $k$-tuple can have repeated elements; since the Euler product is finite, convergence is not a concern. We define
\[\mathfrak S(\mathcal D;Q) := \prod_{\substack{P|Q \\ P \text{ monic}}} \frac{(1-\nu_P(\mathcal D)/|P|)}{(1 - 1/|P|)^{k}} = \sum_{\substack{R_1, \dots, R_k |Q \\ R_i \text{ monic}}} \left(\prod_{i=1}^k \frac{\mu(R_i)}{\phi(R_i)} \right) \sum_{\substack{A_1, \dots, A_k \\ A_i \in \mathcal R(R_i) \\ \sum_i A_i/R_i = 0}} e\left(\sum_{i=1}^k \frac{A_iD_i}{R_i}\right).\]
If $\mathcal D$ has a repeated element, so that $\mathcal D = \{D, D, D_3, \dots, D_k\}$, then $\mathfrak S(\mathcal D;Q) = \frac{|Q|}{\phi(Q)} \mathfrak S(\{D,D_3,\dots, D_k\};Q)$, so we can remove repeated elements from $\mathcal D$ at the expense of a factor of $\frac{|Q|}{\phi(Q)}$. We define $\mathfrak S_0(\mathcal D;Q)$ to be the alternating sum $\sum_{\mathcal J \subset \mathcal D} (-1)^{|\mathcal D\setminus \mathcal J|} \mathfrak S(\mathcal J;Q),$
so we have
\[R_k(q^h) = \sum_{\substack{D_1, \dots, D_k \\ D_i \text{ distinct} \\ |D_i| \le q^h}} \mathfrak S_0(\{D_1, \dots, D_k\};Q) + O(1).\]
This is quite close to the quantity $V_k(Q;h)$, except with the added constraint that the $D_i$'s must be distinct. It suffices to remove this condition. To do so, we put $\delta_{ij} = 1$ if $D_i = D_j$ and $0$ otherwise, so that
\[\sum_{\substack{D_1, \dots, D_k \\ D_i \text{ distinct} \\ |D_i| \le q^h}} \mathfrak S_0(\{D_1, \dots, D_k\};Q) = \sum_{\substack{D_1, \dots,D_k \\ |D_i| \le q^h}} \left(\prod_{1 \le i < j \le k} (1-\d_{ij})\right) \mathfrak S_0(\{D_1, \dots, D_k\};Q). \]
We can expand the product and group terms according to which $D_i$'s are required to be equal, noting that, for example, $\delta_{12}\delta_{23} = \delta_{13}\delta_{23}$. We can also combine terms according to symmetry; the term $\delta_{12}$ and the term $\delta_{34}$ will have identical contributions in the final sum.

Let us now proceed with analyzing $R_5(q^h)$. After some counting, we get that 
\begin{equation*}
\sum_{\substack{D_1, \dots, D_5 \\ D_i \text{ distinct} \\ |D_i| \le q^h}} \mathfrak S_0(\{D_1, \dots, D_5\};Q) = \sum_{\substack{D_1, \dots,D_5 \\ |D_i| \le q^h}} f((\d_{i,j})_{i,j \in [5]}) \mathfrak S_0(\{D_1, \dots, D_5\};Q),
\end{equation*}
where
\[f((d_{i,j})_{i,j\in[5]}) = 1 - 10 \delta_{12} + 20 \delta_{12}\delta_{13} + 15 \delta_{12}\delta_{34} -20\delta_{12}\delta_{13}\delta_{45} -30\delta_{12}\delta_{13}\delta_{14} + 24\delta_{12}\delta_{13}\delta_{14}\delta_{15}.\]
We will consider the contribution from each term in $f$. The term $1$ gives us precisely $V_5(Q;h)$, which we have already analyzed. We can then bound each of the remaining six terms by expanding $\mathfrak S_0$ into a sum of $\mathfrak S$, removing any repeated terms in the appropriate tuple, and applying Lemma \ref{lem:mvl8} to bound $V_k(Q;h)$ for some $k < 5$. These computations are summarized in the following lemma.

\begin{lemma}
Using the notation of this section,
\begin{enumerate}[(a)]
	\item $\sum_{\substack{D_1, D_3, D_4, D_5 \\ |D_i| \le q^h}} \mathfrak S_0(\{D_1, D_1, D_3, D_4, D_5\};Q) \ll q^{2h} \left(\frac{|Q|}{\phi(Q)}\right)^{9},$
	\item $\sum_{\substack{D_1, D_4, D_5 \\ |D_i| \le q^h}} \mathfrak S_0(\{D_1, D_1, D_1, D_4, D_5\};Q) \ll q^{2h} \left(\frac{|Q|}{\phi(Q)}\right)^{3} + q^h \left(\frac{|Q|}{\phi(Q)}\right)^{10},$
	\item $\sum_{\substack{D_1, D_3, D_5 \\ |D_i| \le q^h}} \mathfrak S_0(\{D_1, D_1, D_3, D_3, D_5\};Q) \ll q^{2h} \left(\frac{|Q|}{\phi(Q)}\right)^{3} + q^h \left(\frac{|Q|}{\phi(Q)}\right)^{10},$
	\item $\sum_{\substack{D_1, D_4 \\ |D_i| \le q^h}} \mathfrak S_0(\{D_1, D_1, D_1, D_4, D_4\};Q) \ll q^{2h} \left(\frac{|Q|}{\phi(Q)}\right)^3 + q^{h} \left(\frac{|Q|}{\phi(Q)}\right)^{4},$
	\item $\sum_{\substack{D_1, D_5 \\ |D_i| \le q^h}} \mathfrak S_0(\{D_1, D_1, D_1, D_1, D_5\};Q) \ll q^h \left(\frac{|Q|}{\phi(Q)}\right)^4,$
	\item $\sum_{\substack{D_1 \\ |D_1| \le q^h}} \mathfrak S_0(\{D_1, D_1, D_1,D_1,D_1\};Q) \ll q^{h} \left(\frac{|Q|}{\phi(Q)}\right)^4.$
\end{enumerate}
\end{lemma}
\begin{proof}
For the sake of brevity we omit most of these computations, which are very similar, but we will show that the term corresponding to $\delta_{12}$, in part (a), is $\ll q^{2h}\left(\frac{|Q|}{\phi(Q)}\right)^{9} $.

Assume we have a tuple $\mathcal D = \{D_1,D_1,D_3, D_4,D_5\}$, with one repeated term. As mentioned above, $\mathfrak S(\mathcal D;Q) = \frac{|Q|}{\phi(Q)}\mathfrak S(\{D_1,D_3,D_4,D_5\};Q)$. Expanding $\mathfrak S_0$ and applying this relation shows that
\[\mathfrak S_0(\mathcal D;Q) = \left(\frac{|Q|}{\phi(Q)}-2\right)\mathfrak S_0(\{D_1, D_3, D_4, D_5\};Q) + \left(\frac{|Q|}{\phi(Q)} - 1\right)\mathfrak S_0(\{D_3, D_4, D_5\};Q),\]
so in this way we can remove repeated elements from our sum. The term we want to bound is
\begin{align*}
&\sum_{\substack{D_1, D_3, D_4, D_5 \\ |D_i| \le q^h}} \mathfrak S_0(\{D_1, D_1, D_3, D_4, D_5\};Q) \\
&= \sum_{\substack{D_1, D_3, D_4, D_5 \\ |D_i| \le q^h}} \left(\frac{|Q|}{\phi(Q)}-2\right)\mathfrak S_0(\{D_1, D_3, D_4, D_5\};Q) + \left(\frac{|Q|}{\phi(Q)} - 1\right)\mathfrak S_0(\{D_3, D_4, D_5\};Q) \\
&= \left(\left(\frac{|Q|}{\phi(Q)}-2\right)V_4(Q;h) + q^h\left(\frac{|Q|}{\phi(Q)}-1\right)V_3(Q;h)\right) \\
&\ll \left(\frac{|Q|}{\phi(Q)}\right)^3q^{2h} + \left(\frac{|Q|}{\phi(Q)}\right)^{9} q^{2h},
\end{align*}
where in the last step the bounds follow from Lemma \ref{lem:mvl8}.
\end{proof}
This lemma gives the following corollary.

\begin{corollary}
Let $Q = \prod_{\substack{P \text{ irred.} \\ |P| \le q^{6h}}} P$. For all $\ep > 0$,
\[R_5(q^h) \ll V_5(Q;h) + q^{2h}\left(\frac{|Q|}{\phi(Q)}\right)^{9} \ll q^{(2+\ep)h}.\]
\end{corollary}

Performing the same analysis when $k = 3$ yields the bound
\begin{corollary}
Let $Q = \prod_{\substack{P \text{ irred.} \\ |P| \le q^{6h}}} P$. Then
\[R_3(q^h) \ll V_3(Q;h) +q^h\left(\frac{|Q|}{\phi(Q)}\right)^2 \ll q^h \left(\frac{|Q|}{\phi(Q)}\right)^{8}.\]
\end{corollary}

\section{Numerical Evidence for Odd Moments} \label{sec:numerics}

Here we present several charts supporting our conjectures on the sizes of the odd moments. To begin with, we have computed $\frac 16 R_3(h) = \sum_{1 \le d_1< d_2< d_3 \le h} \mathfrak S_0(\{d_1, d_2, d_3\})$. Below, $\frac 16 R_3(h)$ is plotted in black. We expect $R_3(h)$, and thus also $\frac 16 R_3(h)$, to be of the shape $Ah(\log h)^2$, for some constant $A$. We found an experimental best fit value of $A = 0.373727$, and for this $A$ have plotted $Ah(\log h)^2$ alongside $\frac 16 R_3(h)$, as a dashed red line.

\begin{figure}[H]
\centering
\includegraphics[width=0.7\textwidth]{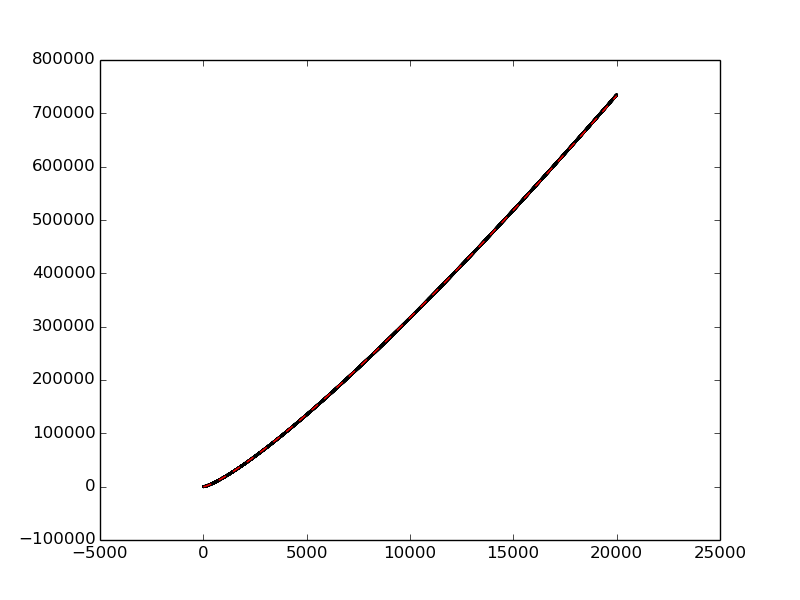}
\caption{$\frac 16 R_3(h)$ for $3 \le h \le 20000$}
\end{figure}

The fit of the theoretical red dashed curve is quite close, but there are lower-order fluctuations; below we plot the difference between $\frac 16 R_3(h)$ and $Ah(\log h)^2$.

\begin{figure}[H]
\centering
\includegraphics[width=0.7\textwidth]{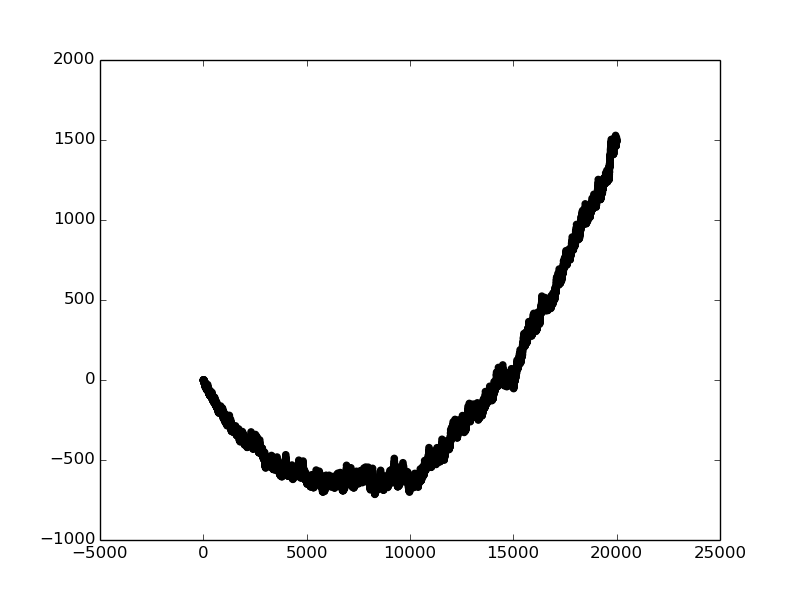}
\caption{$\frac 16 R_3(h) - Ah(\log h)^2$ for $3 \le h \le 20000$}
\end{figure}

Our analysis above includes relatively little discussion about the moments of the distribution of primes themselves. We have computed several third, fifth, and seventh moments of the distribution of primes. Specifically, we have computed $\tilde M_k(N;N^{\delta}) = \frac 1N\sum_{n=N}^{2N} (\psi(n+N^{\delta}) - \psi(n) - N^{\delta})^k$, for each of $\delta = 0.25,0.5$ and $0.75$, and for each of $k = 3,5,7$. For a fixed $\delta$ and $k$, we plot $\tilde M_k(N;N^{\delta})$ for values of $N$ ranging from $1$ to $10^7$, and growing exponentially.

Each of the plots below is drawn with both $x$- and $y$-axes on a logarithmic scale. We expect the $k$th moment to be of size approximately $O(H^{(k-1)/2} (\log \frac NH)^{(k+1)/2}$, where $H = N^{\delta}$, so to give a sense of size, for each plot, $N^{\delta (k-1)/2}(\log N^{1-\delta})^{(k+1)/2}$ is plotted in dashed red. We have also plotted the reflection of the red dashed curve across the $x$-axis, since the odd moments are frequently negative.

\begin{figure}[H]
\centering
\begin{minipage}{0.32\textwidth}
\centering
\includegraphics[width=\textwidth]{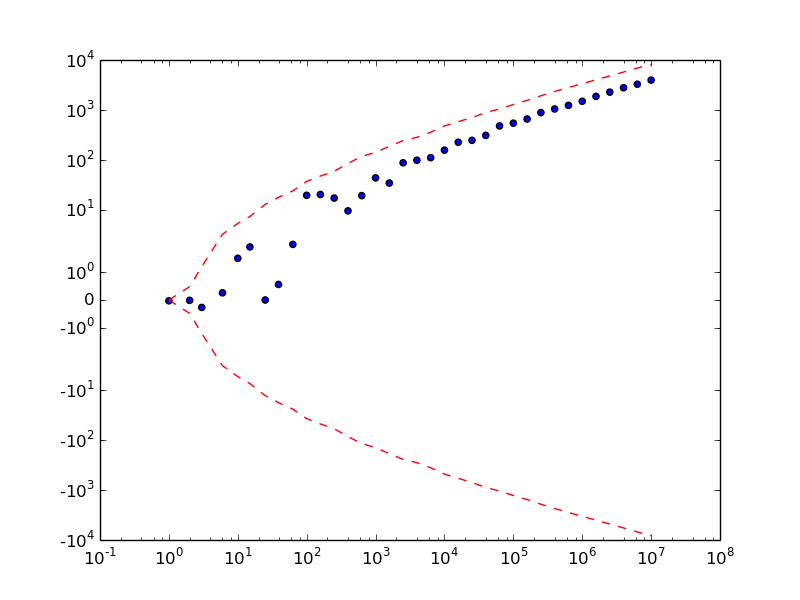}
\caption*{$\d = 0.25$}
\end{minipage}
\begin{minipage}{0.32\textwidth}
\centering
\includegraphics[width=\textwidth]{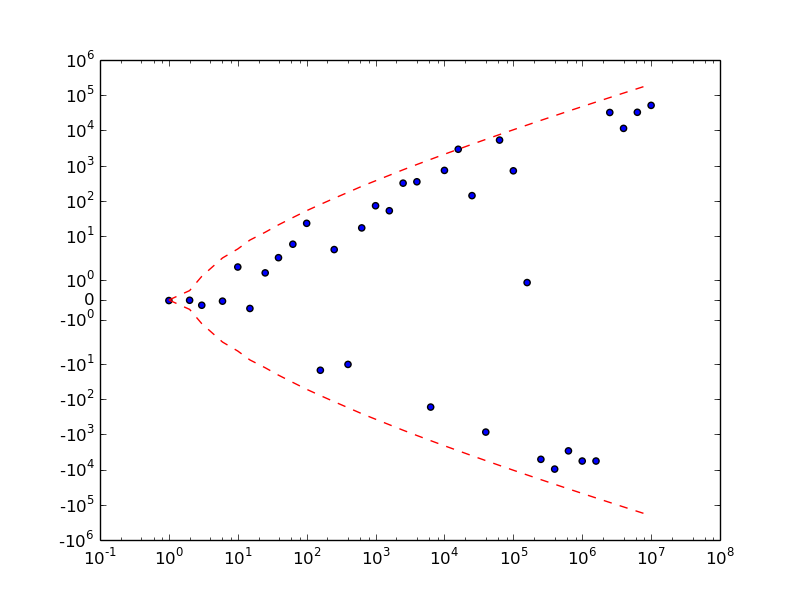}
\caption*{$\d = 0.5$}
\end{minipage}
\begin{minipage}{0.32\textwidth}
\centering
\includegraphics[width=\textwidth]{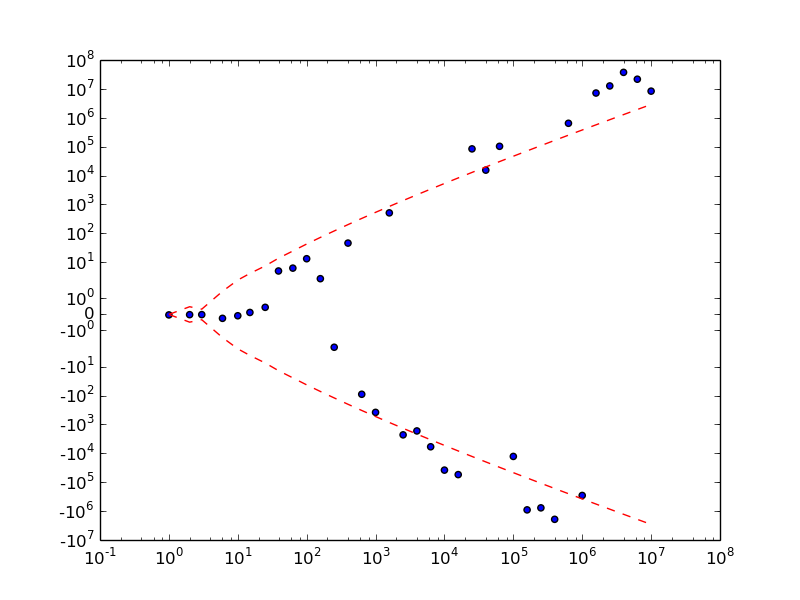}
\caption*{$\d = 0.75$}
\end{minipage}
\caption{Plots of the third moment $M_3(N;N^{\delta})$ for $N \le 10^7$.}
\end{figure}

\begin{figure}[H]
\centering
\begin{minipage}{0.32\textwidth}
\centering
\includegraphics[width=\textwidth]{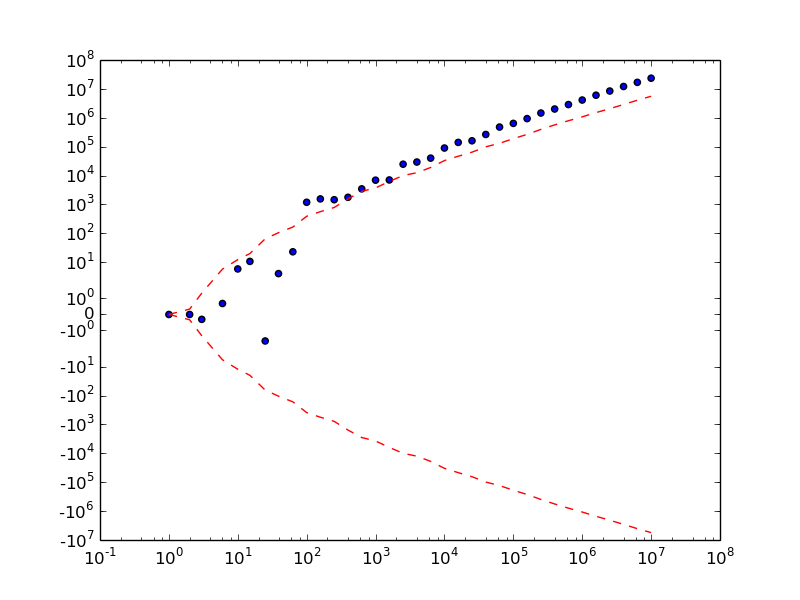}
\caption*{$\d = 0.25$}
\end{minipage}
\begin{minipage}{0.32\textwidth}
\centering
\includegraphics[width=\textwidth]{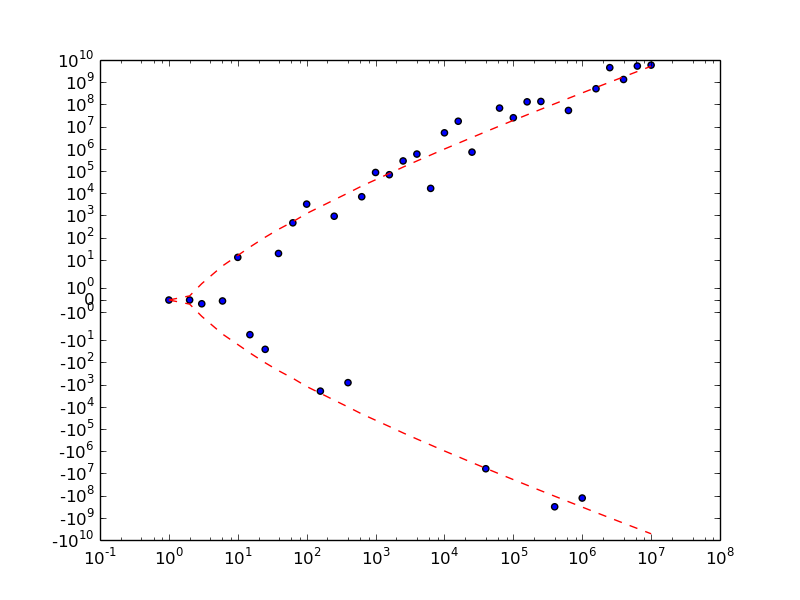}
\caption*{$\d = 0.5$}
\end{minipage}
\begin{minipage}{0.32\textwidth}
\centering
\includegraphics[width=\textwidth]{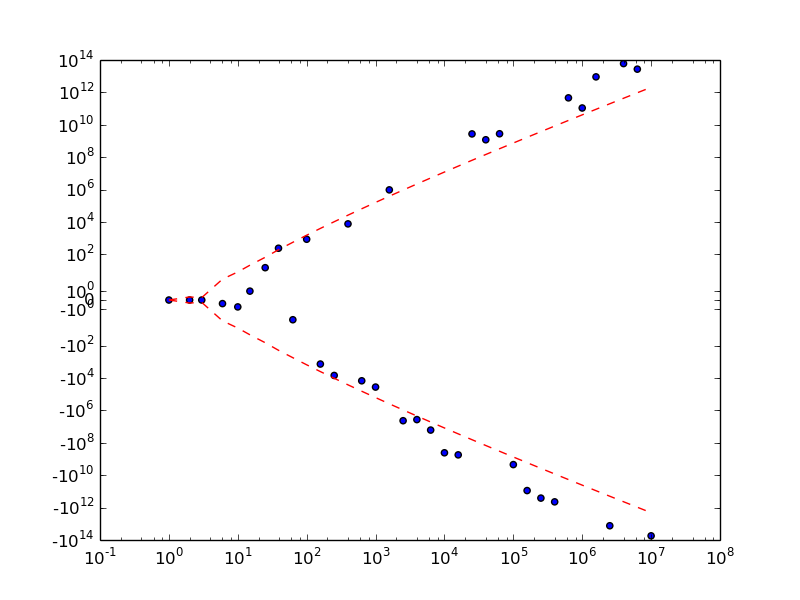}
\caption*{$\d = 0.75$}
\end{minipage}
\caption{Plots of the fifth moment $M_5(N;N^{\delta})$ for $N \le 10^7$.}
\end{figure}

\begin{figure}[H]
\centering
\begin{minipage}{0.32\textwidth}
\centering
\includegraphics[width=\textwidth]{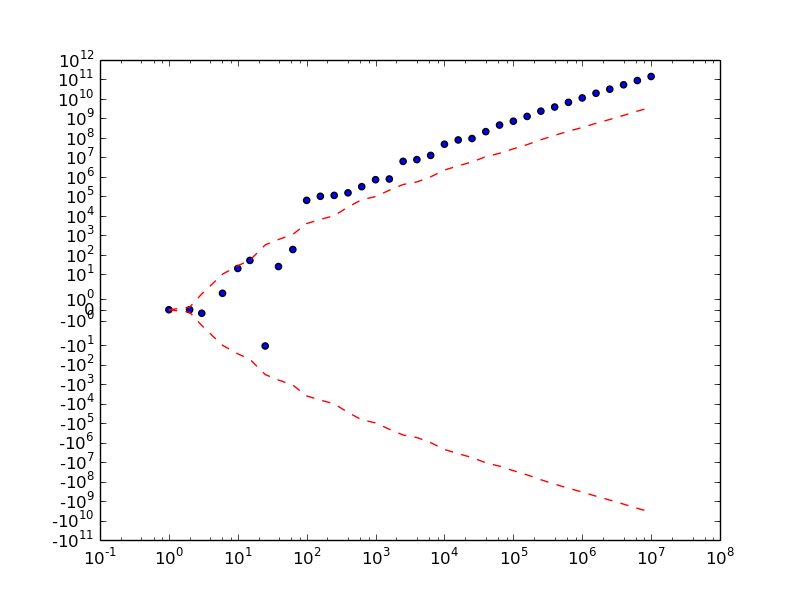}
\caption*{$\d = 0.25$}
\end{minipage}
\begin{minipage}{0.32\textwidth}
\centering
\includegraphics[width=\textwidth]{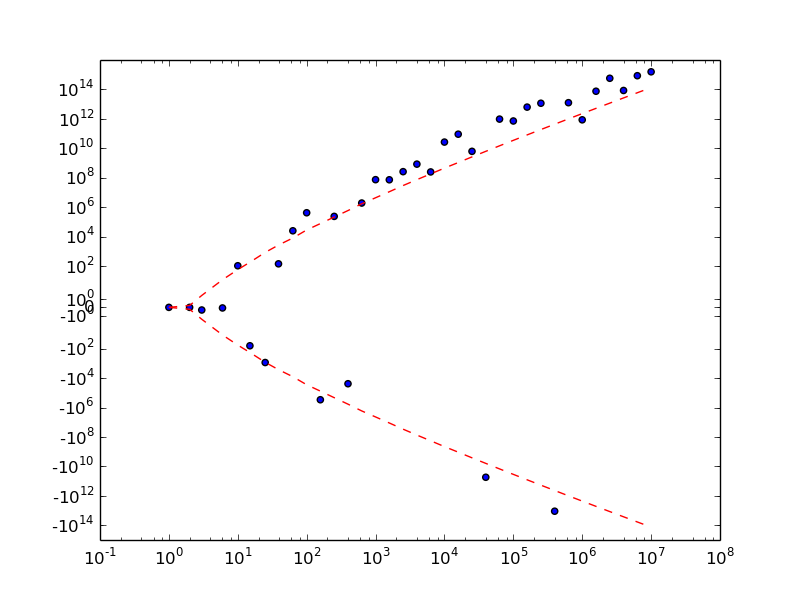}
\caption*{$\d = 0.5$}
\end{minipage}
\begin{minipage}{0.33\textwidth}
\centering
\includegraphics[width=\textwidth]{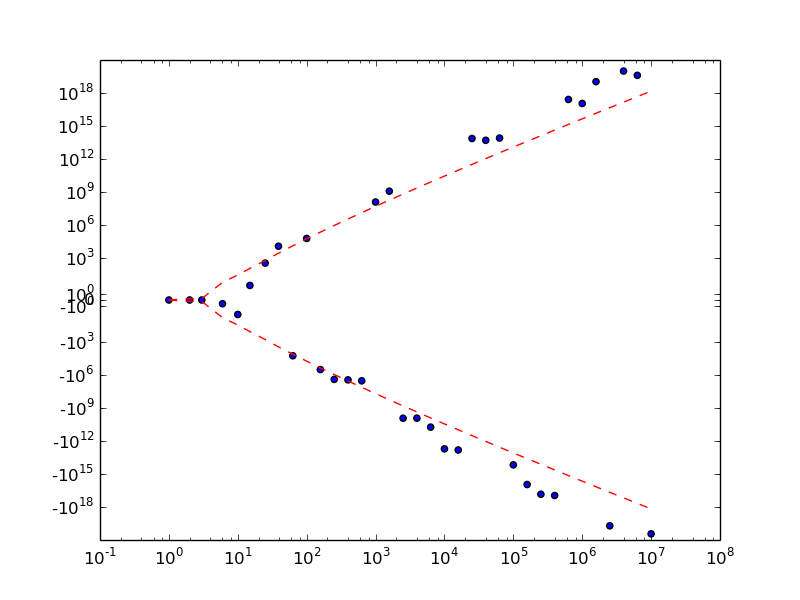}
\caption*{$\d = 0.75$}
\end{minipage}
\caption{Plots of the seventh moment $M_7(N;N^{\delta})$ for $N \le 10^7$.}
\end{figure}

The fit of the red line is reasonably good in all cases, but not perfect. In every case here we seem to see that the odd moments are more frequently positive than negative, but still take on negative values. For $\d = 0.25$, the odd moments seem to be positive for sufficiently large $N$; it is possible that this effect occurs for all sufficiently large $N$, where the threshold depends on $k$ and $\d$. 

\section{Toy Models and Open Problems} \label{sec:toys}

Throughout, we have studied the sum
\[R_k(h) = \sum_{\substack{q_1, \dots, q_k \\ 1 < q_i}} \left(\prod_{i=1}^k \frac{\mu(q_i)}{\phi(q_i)}\right) \sum_{\substack{a_1, \dots, a_k \\ 1 \le a_i \le q_i \\ (a_i,q_i) = 1 \\ \sum a_i/q_i \in \Z}} \prod_{i=1}^k E\left(\frac{a_i}{q_i}\right), \]
where $E(\a) = \sum_{m=1}^h e(m\a)$. The sums $E(\a)$ approximately detect when $\|\a\| \le \frac 1h$; the analogous sum in the function field case precisely detects when $\a$ has small degree. As a result, much of our understanding boils down to answering the following key question.

\begin{question}\label{qn:fractionsums}
Let $\d > 0$ and let $Q > 1/\d$. What is
\[\#\left\{q_1, \dots, q_k \in [Q,2Q], a_i \mod{q_i} : \left\|\frac{a_i}{q_i}\right\| \le \d, \sum_i \frac{a_i}{q_i} \in \Z\right\} ?\]
\end{question}
We conjecture that the answer to this question is as follows.
\begin{conjecture} \label{conj:fractionsums}
Let $\d > 0$ and let $Q > 1/\d$. Let $S$ be the size of the set in Question \ref{qn:fractionsums}. Then for any $\ep > 0$,
\[S \ll \begin{cases} Q^{k+\ep} \d^{k/2} &k \text{ even} \\ Q^{k+\ep} \d^{(k+1)/2} &k \text{ odd}.\end{cases}\]
\end{conjecture}
As we discussed in the introduction, Montgomery and Vaughan \cite{MontgomeryVaughanReducedResidues} considered the related problem of moments of reduced residues modulo $q$. Their work depends on the following answer to Question \ref{qn:fractionsums} above.

\begin{theorem}\label{thm:mv}
Let $S$ be the size of the set in Question \ref{qn:fractionsums}. Then
\[S \ll \begin{cases} \d^{k/2} \sum_{\substack{Q \le r_i \le 2Q \\ 1 \le i \le k/2}} \frac{r_1^2 \cdots r_{k/2}^2}{\lcm(r_i)} + \d^{k/2-1/7k} \sum_{\substack{Q \le r_i \le 2Q \\ 1 \le i \le k}} \frac{r_1 \cdots r_k}{\lcm(r_i)} &k \text{ even} \\ \d^{k/2-1/7k} \sum_{\substack{Q \le r_i \le 2Q \\ 1 \le i \le k}} \frac{r_1 \cdots r_k}{\lcm(r_i)} &k \text{ odd} \end{cases} \]
\end{theorem}
The proof of the above theorem is identical to the proof in \cite{MontgomeryVaughanReducedResidues}. This agrees with Conjecture \ref{conj:fractionsums} for the case when $k$ is even, but gives a weaker bound when $k$ is odd. 

We can also consider generalizations of Question \ref{qn:fractionsums}. For example, instead of specifying that $\left\|\frac{a_i}{q_i}\right\| \le \d$, we may ask that it lie in any specified interval.
\begin{question}
Let $Q > 1/\d$ and let $I_1, \dots, I_k$ be $k$ intervals in $[0,1]$ with $|I_j| \ge \d$ for all $j$. What is
\[\#\left\{q_1, \dots, q_k \in [Q,2Q], a_i \mod{q_i} : \left\|\frac{a_i}{q_i}\right\| \in I_i, \sum_i \frac{a_i}{q_i} \in \Z\right\} ?\]
\end{question}
Answers to these questions would give us more refined understanding of sums of singular series. The conjectures above are related to sums over $\S(\{h_1, \dots, h_k\})$, where each $h_i$ lies in the same interval $[0,h]$. We can instead ask about sums of singular series restricted to arbitrary intervals, or along arithmetic progressions. We state the following questions using smooth cutoff functions as opposed to intervals. 

\begin{question}\label{qn:sumsinintervals}
Let $\Phi_1, \dots, \Phi_k$ be smooth functions with compact support on $\R$, and let $H \in \R_{>0}$. What is
\[\sum_{h_1, \dots, h_k \in \Z} \S_0(\{h_1, \dots, h_k\})\Phi_1\left(\frac{h_1}{H}\right) \cdots \Phi_k\left(\frac{h_k}{H}\right) ?\]
\end{question}
\begin{question}\label{qn:sumsinprogressions}
Let $\Phi_1, \dots, \Phi_k$ be smooth functions with compact support on $\R$, and let $H \in \R_{>0}$. For arithmetic progressions $a_1 \mod{q_1}, \dots, a_k \mod{q_k}$, what is
\[\sum_{\substack{h_1, \dots, h_k \in \Z \\ h_i \equiv a_i \mod{q_i}}} \S_0(\{h_1, \dots, h_k \})\Phi_1\left(\frac{h_1}{H}\right) \cdots \Phi_k\left(\frac{h_k}{H}\right) ?\]
\end{question}

Question \ref{qn:sumsinintervals} addresses the correlations of $\psi(x+h)-\psi(x)$ and $\psi(x+h_1 + h)-\psi(x+h_1)$; in other words, the correlations of the number of primes in intervals in different places. Question \ref{qn:sumsinprogressions} addresses the correlations of the number of primes in distinct arithmetic progressions. For both of these questions, the main term ought to come from diagonal terms where $h_1 = h_2$, for example, thus collapsing the weight function, whereas the error term ought to arise from off-diagonal contributions. 

In the case when $k = 2$, Question \ref{qn:sumsinprogressions} has been widely studied in the context of prime number races. The ``Shanks-R\'enyi prime number race'' is the following problem: let $\pi(x;q,a)$ denote the number of primes $p \le x$ with $p \equiv a \mod q$. Then for any $n$-tuple $(a_1, \dots, a_n)$ of equivalence classes mod $q$ that are relatively prime to $q$, will we have the ordering
\[\pi(x;q,a_1) > \pi(x;q,a_2) > \cdots > \pi(x;q,a_n)\]
for infinitely many integers $x$? Many aspects of this question have been studied; see for example the expositions of Granville and Martin \cite{MR2202918}, and Ford and Konyagin \cite{MR1985941}.

In \cite{MR3961320}, Ford, Harper, and Lamzouri show that, although any ordering appears infinitely often, for $n$ large with respect to $q$, the prime number races among orderings can exhibit large biases. They rely on the fact that counts of primes in distinct progressions have negative correlations, which they arrange to produce a bias. This analysis is also connected to the work of Lemke Oliver and Soundararajan in \cite{lemkeoliversound}, who use averages of two-term singular series in arithmetic progressions to show bias in the distribution of consecutive primes. It is plausible that a more precise understanding of the questions above would lead to an extension of the work of Lemke Oliver and Soundararajan.

\bibliographystyle{amsplain}
\bibliography{singseries}
\end{document}